\documentclass{amsart}
\usepackage[utf8]{inputenc}
\usepackage{amsmath, amsfonts, amsthm}
\usepackage{mathtools}
\usepackage{booktabs}
\usepackage{hyperref}
\usepackage{bbm}
\usepackage{cleveref}
\usepackage{stmaryrd}
\usepackage{enumerate}
\usepackage{xcolor}
\usepackage{comment}
\usepackage{enumitem}
\usepackage{algorithm,algpseudocode,algorithmicx}
\usepackage[left=1.25in,right=1.25in,top=1in]{geometry}
\numberwithin{equation}{section}
\title[Energy stable schemes for Stochastic Galerkin SWE]{Energy Stable and Structure-Preserving Schemes for the Stochastic Galerkin Shallow Water Equations}
\author{Dihan Dai, Yekaterina Epshteyn and Akil Narayan}
\address[Dihan Dai]%
{San Diego, CA, USA}
\email{dihandai2022@gmail.com}
\address[Yekaterina Epshteyn]%
{Department of Mathematics,
The University of Utah,
Salt Lake City, UT 84112, USA}
\email{epshteyn@math.utah.edu}
\address[Akil Narayan]%
{Department of Mathematics and Scientific Computing and Imaging (SCI) Institute,
The University of Utah,
Salt Lake City, UT 84112, USA}
\email{akil@sci.utah.edu}
\date{October 2023}

\newcommand{\R}{\mathbbm{R}}
\newcommand{\N}{\mathbbm{N}}

\newcommand{\cP}{\mcp}
\newcommand{\hh}{\widehat{h}}
\newcommand{\bbh}{\overline{\boldsymbol{h}}}
\newcommand{\bbu}{\overline{\boldsymbol{u}}}

\newcommand{\bh}{\boldsymbol{h}}
\newcommand{\bu}{\boldsymbol{u}}
\newcommand{\bq}{\boldsymbol{q}}

\newcommand{\bw}{\boldsymbol{w}}
\newcommand{\hU}{\widehat{U}}
\newcommand{\hV}{\widehat{V}}
\newcommand{\bs}[1]{\boldsymbol{#1}}
\newcommand{\bS}[1]{\bs{S}}
\newcommand{\bE}[1]{\bs{E}}

\newcommand{\bV}{\boldsymbol{V}}
\newcommand{\bU}{\boldsymbol{U}}

\newcommand{\bbV}{\overline{\boldsymbol{V}}}

\newcommand{\bR}{\boldsymbol{R}}
\newcommand{\hu}{\widehat{u}}

\newcommand{\hq}{\widehat{q}}
\newcommand{\hqx}{\widehat{q}}

\newcommand{\hB}{\widehat{B}}
\newcommand{\hS}{\widehat{S}}
\newcommand{\hF}{\widehat{F}}

\newcommand{\mcp}{\mathcal{P}}
\newcommand{\ph}{_{i+\frac{i}{2}}}

\newtheorem{lemma}{Lemma}[section]
\newtheorem{proposition}{Proposition}[section]
\newtheorem{theorem}{Theorem}[section]
\newtheorem{remark}{Remark}[section]
\newcommand{\dbracket}[1]{\left\llbracket#1\right\rrbracket}
\newcommand{\dabracket}[1]{\llangle#1\rrangle}
\newtheorem{definition}{Definition}[section]
\makeatletter
\renewcommand*\env@matrix[1][\arraystretch]{%
  \edef\arraystretch{#1}%
  \hskip -\arraycolsep
  \let\@ifnextchar\new@ifnextchar
  \array{*\c@MaxMatrixCols c}}
\makeatother
\makeatletter
\newsavebox{\@brx}
\newcommand{\llangle}[1][]{\savebox{\@brx}{\(\m@th{#1\langle}\)}%
  \mathopen{\copy\@brx\kern-0.5\wd\@brx\usebox{\@brx}}}
\newcommand{\rrangle}[1][]{\savebox{\@brx}{\(\m@th{#1\rangle}\)}%
  \mathclose{\copy\@brx\kern-0.5\wd\@brx\usebox{\@brx}}}
\makeatother

\newcommand{\ppx}[1]{\frac{\partial}{\partial #1}}
\newcommand{\pfpx}[2]{\frac{\partial #1}{\partial #2}}
\newcommand{\pfnpx}[3]{\frac{\partial^{#2} #1}{\partial #3^{#2}}}

\begin{document}

\maketitle

\begin{abstract}
The shallow water flow model is widely used to describe water flows in rivers, lakes, and coastal areas. Accounting for uncertainty in the corresponding transport-dominated nonlinear PDE models presents theoretical and numerical challenges that motivate the central advances of this paper. Starting with a spatially one-dimensional hyperbolicity-preserving, positivity-preserving stochastic Galerkin formulation of the parametric/uncertain shallow water equations, we derive an entropy-entropy flux pair for the system. We exploit this entropy-entropy flux pair to construct structure-preserving second-order energy conservative, and first- and second-order energy stable finite volume schemes for the stochastic Galerkin shallow water system. The performance of the methods is illustrated on several numerical experiments.
\end{abstract}
\section{Introduction}
The one-dimensional Saint-Venant system of shallow water equations (SWE) is a popular model of water flows where vertical length scales are much smaller than horizontal ones \cite{barredesaint-venant_theorie_1871}. This system in conservative form is given by,
\begin{align}\label{eq:swesg1-1d}
  U_t + F(U)_x &= S(U), & U = (h, q)^\top \in \R^2,
\end{align}
where $U = U(x,t)$ is the vector of conservative variables; $h(x,t)$ is the water height (a mass-like variable) and $q(x,t)$ is the water discharge (a momentum-like variable). The flux $F$ and source term $S$ are given by,
\begin{align}\label{eq:swe-flux-1d}
  F(U) &= \left(\begin{array}{c} 
    q \\ \frac{{(q)}^2}{h} + \frac{g h^2}{2} 
  \end{array}\right), & 
  S(U) &= \left(\begin{array}{c} 
  0 \\ -g h B' \end{array}\right).
\end{align}
where $B(x)$ is the (assumed known) bottom topography function and $g > 0$ is the gravitational constant. The system \eqref{eq:swesg1-1d} is supplemented with initial and boundary data that we omit for the time being.

The one-dimensional SWE model \eqref{eq:swesg1-1d} is hyperbolic system of partial differential equations (PDE) if $h > 0$, and hence with the non-zero source $S$, then \eqref{eq:swesg1-1d} is a nonlinear hyperbolic balance law. Because of this, it inherits the standard challenges in developing numerical methods for such models: solutions generically develop discontinuities in finite time even with smooth initial data, non-uniqueness of weak solutions should be rectified by an implicit or explicit numerical imposition of entropy conditions, and implicit time-integration solvers are challenging to implement due to the nonlinearity \cite{dafermos2016hyperbolic,leveque_numerical_1992,leveque_finite_2002}. In addition to all this, the SWE has challenges that are somewhat specific to its particular form: positivity of the water height $h$ should be maintained, and numerical schemes should accurately capture near-equilibrium dynamics, which is typically achieved by imposing the \textit{well-balanced} property \cite{bermudez_upwind_1994}, i.e., that the PDE equilibrium states are exactly captured at the discrete level. 

A more nebulous and hence more frustrating challenge is that of \textit{uncertainty} in the model. For example, one may have incomplete, partial information about the initial data or the bottom topography function $B$. In such cases, one frequently models this data as a random variable or process, and hence the solution $U$ to \eqref{eq:swesg1-1d} is random. We consider the somewhat more simple situation when the input uncertainty is encoded with a finite-dimensional random variable, in which case \eqref{eq:swesg1-1d} becomes a parametric model (with the input random variables serving as the parameters). Even with this simplification, the parametric or stochastic nature of the solution exacerbates many of the previously described numerical challenges. A particularly successful approach for handling such problems that we will employ is the polynomial Chaos (PC) method, wherein $U$ is approximated as a polynomial function of the input parameters \cite{xiu2010numerical,smith_uncertainty_2013,sullivan2015introduction}. 

The class of \textit{non-intrusive} PC strategies construct the polynomial by collecting an ensemble of solutions to \eqref{eq:swesg1-1d} at a collection of fixed values of the parameters. This approach is attractive since it can exploit existing and trusted legacy solvers for \eqref{eq:swesg1-1d}, for which there are several effective choices \cite{zhou_surface_2001,rogers_mathematical_2003,crnjaric-zic_balanced_2004,xing_high_2006,CiCP-1-100,kurganov_centralupwind_2002,xing_high_2005,xing_chapter_2017,liu_wellbalanced_2018,bryson_wellbalanced_2011,kurganov_secondorder_2007-2,xing_high_2016,epshteyn_adaptive_2023,kurganov_finite-volume_2018}. However, this approach suffers from the disadvantage that making concrete statements about the quality or properties of the resulting polynomial approximation can be challenging. For example, one cannot guarantee that entropy conditions are satisfied if the polynomial approximation is evaluated away from the parameter ensemble used to construct the approximation.

This paper is concerned with an alternative \textit{intrusive} approach, the stochastic Galerkin (SG) method for PC approximation, which addresses the parametric dependence in a Galerkin fashion, e.g., by enforcing that certain probabilistic moments of \eqref{eq:swesg1-1d} vanish. This approach has the potential to provide pathways to mathematical rigor of numerical methods through weak enforcement of the parametric dependence. SG methods transform a parametric model \eqref{eq:swesg1-1d} into a new non-parametric model of larger system size. Since the new SG formulation is non-parametric, one can apply typical deterministic numerical methods for systems of PDEs to solve the SG problem. Such approaches have shown particular success for modeling parametric dependence in elliptic problems; see, e.g., \cite{cohen_approximation_2015}. However, the notable drawback of SG methods when applied to (nonlinear) hyperbolic PDEs is that the new non-parametric SG system need not be a hyperbolic PDE itself, which changes the essential character of the SG system relative to the original system. Recent work has developed an SG formulation for the SWE in conservative form that involves a special SG treatment for the nonlinear, non-polynomial terms \cite{doi:10.1137/20M1360736}. Such an approach can be used to develop a well-balanced, hyperbolicity-preserving, and positivity-preserving finite volume method to solve the SG SWE system. This approach has also been extended to two-dimensional SWE systems \cite{dai2021hyperbolicity}.

\subsection{Contributions of this article}
We make the following contributions that build on \cite{doi:10.1137/20M1360736}:
\begin{itemize}[itemsep=0pt,topsep=0pt,leftmargin=15pt]
  \item We derive an entropy-entropy flux pair for the spatially one-dimensional hyperbolicity-preserving, positivity-preserving SG SWE system derived in \cite{doi:10.1137/20M1360736}, see \cref{thm:eef-sgswe-1d}. Entropy-entropy flux pairs are the theoretical starting point for proposing entropy admissibility criteria to resolve non-uniqueness of weak solutions. 
  \item Using the entropy-entropy flux pair, we devise second-order energy conservative, and first- and second-order energy stable finite volume schemes for the SG SWE, all of which are also well-balanced. See \cref{thm:ec-scheme,thm:es1,thm:es2}, with the procedure in \cref{alg:swesg}. The designed energy conservative and energy stable schemes are the stochastic extensions of the schemes developed in \cite{fjordholm2011well,fjordholm2012arbitrarily}.
  \item We provide numerical experiments that explore the simulation capabilities of the new schemes. To the best of our knowledge, these are the first schemes for any SG SWE system that boast energy stability, the well-balanced property, while also being positivity- and hyperbolicity-preserving.
\end{itemize}

\noindent An outline of this paper is as follows: \Cref{sec:prelim}
introduces our notation, along with background on PC methods and the
SG SWE system from \cite{doi:10.1137/20M1360736}. \Cref{sec3} provides
our entropy-entropy pair construction for the SG SWE
system. \Cref{sec:schemes} provides the statement of the energy
conservative and energy stable schemes that we develop, along with
proofs of their theoretical properties, as well as their algorithmic details. \Cref{sec:results} compiles
numerical examples that demonstrate the performance of our
scheme. \Cref{sec:conclusion} gives brief summary of the main results
and some future research directions. We summarize our notation in this article in \cref{tab:notation}.

\begin{table}[htbp]
  \begin{center}
  \resizebox{\textwidth}{!}{
    \renewcommand{\tabcolsep}{0.4cm}
    {\scriptsize
    \renewcommand{\arraystretch}{1.3}
    \renewcommand{\tabcolsep}{12pt}
    \begin{tabular}{@{}lp{0.8\textwidth}@{}}
      \toprule
      $K$& number of terms in PCE expansion,\\
      $\hU = (\hh^\top, \widehat{q}^\top)^\top \in \mathbbm{R}^K\times\mathbbm{R}^K$ & PCE state vector and its components for 1D SG SWE,  \\  
      $\bU = (\bs{h}^\top, \bs{q}^\top)^\top \in \mathbbm{R}^K\times\mathbbm{R}^K$ & quantities related to cell averages/reconstructed values, etc. of the PCE state vector and its components for 1D SG SWE,  \\
      $(E,H)$& entropy pair for 1D SG SWE,\\
 %     $\hU = (\hh^\top, \hqx^\top, \hqy^\top)^\top \in \mathbbm{R}^K\times\mathbbm{R}^K\times\mathbbm{R}^K$ & PCE state vector and its components for 2D SG SWE,  \\
  %    $\bU = (\bs{h}^\top, \bs{q^x}^\top, \bs{q^y}^\top)^\top \in \mathbbm{R}^K\times\mathbbm{R}^K\times\mathbbm{R}^K$ & quantities related to cell averages/reconstructed values, etc. of the PCE state vector and its components for 2D SG SWE,  \\                  
 %     $(E,H_1,H_2)$& entropy pair for 2D SG SWE,\\
  %    $(\hu, \hv)$& PCE vector for $x$- and $y$-velocity,\\
   %   $(\bs{u}, \bs{v})$& quantities related to cell
   %   averages/reconstructed values, etc. of the PCE vector for $x$-
   %   and $y$-velocity,\\
         $\hu$& PCE vector for velocity,\\
      $\bs{u}$& quantity related to cell averages/reconstructed values, etc. of the PCE vector for velocity,\\
      $\hV, \bV$& the PCE vector for the entropic variable and quantities related to cell averages/reconstructed values, etc. of the entropic variable,\\
      $\mathcal{P}(\cdot)$& the operator that maps a PCE vector to a PCE triple-product matrix.\\
      $\overline{\bh}_{i+\frac{1}{2}}$& arithmetic average of $\bh_{i}$ and $\bh_{i+1}$. Similar notation is applied to other bold letters, e.g, $\overline{\bs{U}}_{i+\frac{1}{2}}$.\\
      $\dbracket{\bh}_{i+\frac{1}{2}}$& first-order jump defined via cell averages, \\
      $\dabracket{\bh}_{i+\frac{1}{2}}$& (notationally) second-order jump defined via reconstructed values.\\
      $\widehat{F}$& flux,\\
      $\mathcal{F}$& numerical flux,\\
      $\bs{Q}^{EC}, \bs{Q}^{ES1}, \bs{Q}^{ES2}$& energy-conservative flux, 1st-order energy-stable flux, and 2nd-order energy stable flux, respectively,\\
      $\bs{w}^{\pm},\widetilde{\bs{w}}^{\pm}$&scaled variables,\\
    \bottomrule
    \end{tabular}
  }
    }
  \end{center}
\caption{Notation and terminology used throughout this article.}\label{tab:notation}
\end{table}

\section{Preliminaries}\label{sec:prelim}

\subsection{Notation}
We use $\|\cdot\|$ to denote the standard Euclidean ($\ell^2)$ norm operating on vectors. If $f: \R^m \rightarrow \R^n$ for $m, n \in \N$, then we write $f(x)$ for $x = (x_1, \ldots, x_m)$, and $f(x) = (f_1, \ldots, f_n)$. We use the following notation for the gradient:
\begin{align*}
  f_x \coloneqq \pfpx{f}{x} = \left(\begin{array}{cccc} 
    \pfpx{f_1}{x_1} & \pfpx{f_1}{x_2} &\cdots & \pfpx{f_1}{x_m} \\
    \pfpx{f_2}{x_1} & \pfpx{f_2}{x_2} &\cdots & \pfpx{f_2}{x_m} \\
    \vdots & \vdots & & \vdots \\
    \pfpx{f_n}{x_1} & \pfpx{f_n}{x_2} &\cdots & \pfpx{f_n}{x_m}
  \end{array}\right) \in \R^{n \times m}.
\end{align*}
When $n = 1$ (i.e., $f$ is scalar-valued) then $\pfnpx{f}{2}{x}$ is the $n \times n$ Hessian of $f$. If $\bs{A}$ is a square matrix, then we write $\bs{A} > 0$ and $\bs{A} \geq 0$ when $\bs{A}$ is positive definite and positive semi-definite, respectively.

In work on the SWE system \eqref{eq:swesg1-1d} it is common to introduce the water velocity (equilibrium) variable
\begin{align}\label{eq:u-SWE}
  u &\coloneqq \frac{q}{h},
\end{align}
and we also make use of this variable in what follows.

\subsection{Polynomial Chaos Expansion}
In this section, we briefly review the results and notation for polynomial chaos expansion. More comprehensive results can be found in \cite{debusschere2004numerical,xiu2010numerical,sullivan2015introduction}, etc.

%Let $\bx:= x\in \mathbbm{R}^1$ or $\bx:= (x, y)\in \mathbbm{R}^2$ be a space variable in the physical space and 
Let $\xi\in\mathbbm{R}^d$ be a random variable associated with Lebesgue density function $\rho$. Define the function space
\begin{align*}
    L^2_\rho(\R^d) \coloneqq \left\{f: \R^d\to\R\middle\vert\left(\int_{\R^d}f^2(s)\rho(s)ds\right)^{\frac{1}{2}}<+\infty\right\}.   
\end{align*}
Assuming finite polynomial moments of all orders for $\rho$, there exists an orthonormal basis $\{\phi_k\}_{k = 1}^{\infty}$ of $L^2_\rho$, i.e.,
\begin{align}\label{eq:orthocond}
  \langle\phi_k,\phi_\ell\rangle_{\rho} &\coloneqq \int_{\R^d} \phi_k(s) \phi_\ell(s) {\rho}(s)d s  = \delta_{k, \ell}, & \phi_1(\xi) &\equiv 1,
\end{align}
for all $k,\ell\in\N$, where $\delta_{k,\ell}$ is the Kronecker delta. PCE seeks a representation of a random field $z(\cdot, \cdot, \xi)\in L^2_\rho$ in terms of a series of orthonormal polynomials for $\xi$,
\begin{equation}\label{eq:PCE}
    z(x, t, \xi) \stackrel{L^2_\rho}{=} \sum_{k = 1}^{\infty} \widehat{z}_i(x,t)\phi_i(\xi),
\end{equation}
where $x,t$ are the deterministic spatial and temporal variables, and $\widehat{z}_i(x,t)$ are deterministic Fourier-like coefficients. The equation \eqref{eq:PCE} holds true for all $z(x,t;\cdot)\in L^2_\rho$ under mild conditions \cite{ernst_convergence_2012}. In practice, a finite truncation of \eqref{eq:PCE} is usually considered. Let $P$ be a $K$-dimensional polynomial subspace of $L^2_\rho$, 
\begin{align}\label{eq:P-def}
  P = \mathrm{span} \left\{ \phi_k, \;\; k=1, \ldots, K \right\},
\end{align}
i.e., we let $\phi_k$ be an orthonormal basis for $P$. We make the common assumption that $1 \in P$, and for convenience we assume that,
\begin{align*}
  \phi_1(\xi) \equiv 1.
\end{align*}
A popular choice for $P$ is the total degree space, but several other options are possible.

%Let $\bs{\zeta} = (\zeta_1,\ldots,\zeta_d)\in \R^d$ be a variable, and $\bs{\nu} = (\nu_1,\ldots,\nu_d)\in\N_0^d$ be a multi-index. The monomials associated with $\nu$ is denoted by
%\begin{align*}
%  \zeta^\nu &\coloneqq \prod_{j=1}^d \zeta_j^{\nu_j}, & \zeta^0 = \zeta^{(0,0,\ldots,0)} &= 1.
%\end{align*}
%Define the polynomial space:
%\begin{align*}
%  P_\Lambda &= \mathrm{span} \{ \bs{\zeta}^{\bs{\nu}}\;\; \big|\;\; \nu \in \Lambda\}, & \dim P_\Lambda &= K \coloneqq |\Lambda|,
%\end{align*}
%where $\Lambda\subseteq\mathbb{N}^d_0$ is any non-empty, size-$K$ finite set of multi-indices with $(0,\ldots,0)\in \Lambda$. 

%If $P_\Lambda = \mathrm{span}\{\phi_k\;\;\big|\;\; 1 \leq k \leq K\}$, then 
The $K$-term PCE \textit{approximation} of a random field $z$ onto $P$ is defined by the projection of \eqref{eq:PCE} onto $P$:
\begin{align}\label{eq:PCEex} 
    \Pi_P[z](x,t,\xi)\coloneqq\sum_{k=1}^K \widehat{z}_k(x,t)\phi_k(\xi).
\end{align}
Using the orthogonality of the basis function, the statistics of $\Pi_P[z]$ can be expressed in terms of the expansion coefficients. For example, the mean and the variance of $\Pi_P[z]$ are given by:
\begin{equation}\label{eq:expvar}
  \mathbb{E}[\Pi_P[z](x,t,\xi)] = \widehat{z}_1(x,t),\quad \text{Var}[\Pi_P[z](x,t,\xi)] = \sum_{k=2}^{K}\widehat{z}_k^2(x,t).
\end{equation}
Let $\widehat{z} = \left(\widehat{z}_1,\cdots,\widehat{z}_k\right)\in \mathbbm{R}^K$ be the vector of the expansion coefficients in \eqref{eq:expvar}. Define the linear operator $\mathcal{P}: \R^{K} \rightarrow \R^{K\times K}$ as
\begin{align}\label{eq:pmatrix}
  \mathcal{P}(\widehat{z}) &\coloneqq \sum_{k=1}^K\widehat{z}_k\mathcal{M}_k, &
  \mathcal{M}_k &\in \R^{K \times K}, &
  (\mathcal{M}_k)_{\ell, m} &= \langle\phi_k,\phi_\ell\phi_m\rangle_{\rho}.
\end{align}
Fixing $\widehat{z} \in \R^K$, then $\mathcal{P}(\widehat{z})$ is the (symmetric) quadratic form matrix representation of the bilinear operator $(\widehat{a}, \widehat{b}) \mapsto \left\langle a_P\, z_P, b_P \right\rangle_\rho$, where $z_P \coloneqq \sum_{k=1}^K \widehat{z}_k \phi_k(\xi)$ and similarly for $a_P, b_P$ with $\widehat{a}, \widehat{b} \in \R^K$. Using the fact that $(\mathcal{M}_k)_{\ell,m}$ is commutative in $(k,m)$ a direct computation shows:
%It can be verified by direct computation that
\begin{align}\label{eq:pmatrixproperty}
  \mathcal{P}(\widehat{a}) &= \begin{pmatrix}\mathcal{M}_1\widehat{a}, \; \mathcal{M}_2\widehat{a}, \; \ldots, \; \mathcal{M}_K\widehat{a} \end{pmatrix}, %& 
%  \mathcal{P}(\widehat{a})\widehat{b} &= \mathcal{P}(\widehat{b})\widehat{a}.
\end{align}
A useful lemma is given as follows.
\begin{lemma}
For any two vectors $\widehat{a},\widehat{b}\in \mathbb{R}^K$, 
\begin{align}\label{eq:commute}
  \mcp(\widehat{a})\widehat{b} &=  \mcp(\widehat{b})\widehat{a}, & \widehat{b}^\top\mcp(\widehat{a}) &=  \widehat{a}^\top\mcp(\widehat{b}).
\end{align}
\end{lemma}
The proof is straightforward using \eqref{eq:pmatrix} and
\eqref{eq:pmatrixproperty} along with the symmetry of
$\mathcal{P}(\cdot)$. This result is a ``commutative'' property of the
operator $\mcp(\cdot)$. %that will prove very useful in what
For example: For any $a, b, c \in \R^K$,
\begin{align}\label{eq:acb}
  \ppx{c} a^\top \mcp(c) b^\top = a^\top \mcp(b),
\end{align}
%The second part in \eqref{eq:commute} uses the symmetry of $\mcp(\cdot)$.

A \textit{stochastic Galerkin} (SG) formulation of a $\xi$-parameterized PDE corresponds to making the ansatz that the state variable lies in the space $P$, and projecting the PDE residual onto the same space. Straightforward applications of this procedure to (nonlinear) hyperbolic PDEs typically do not result in hyperbolic SG formulations.

\subsection{Hyperbolic-Preserving Stochastic Galerkin Formulation for Shallow Water Equation}
In \cite{doi:10.1137/20M1360736}, we have derived a hyperbolicity-preserving stochastic Galerkin formulation for the shallow water equations. We briefly recall the results in this section. %giving details in the one-dimensional case and summarizing results for the two-dimensional case.
%\subsubsection{One-dimensional case}
%We describe a stochastic Galerkin (SG) formulation of the one-dimensional SWE model \eqref{eq:swesg1-1d} that was derived in \cite{doi:10.1137/20M1360736}. 
We make the ansatz that the solutions to $h, q$ lie in the polynomial space $P$,
\begin{subequations}\label{eq:sg-ansatz1d}
\begin{align}
  h \simeq h_P &\coloneqq \sum_{k=1}^K \hh_k(x,t) \phi_k(\xi), \\
  q \simeq q_P &\coloneqq \sum_{k=1}^K (\hqx)_k(x,t) \phi_k(\xi), 
\end{align}
\end{subequations}
and use these to formulate a $\xi$-variable Galerkin projection of the SWE. We make a special choice of how the Galerkin projection of the nonlinear, non-polynomial term $(q)^2/h$ is truncated, which results in a \textit{new} (stochastic Galerkin) system of balance laws whose state variables are the expansion coefficients in \eqref{eq:sg-ansatz1d} \cite{doi:10.1137/20M1360736}:
%. This results in a \textit{new} (stochastic Galerkin) system of balance laws whose state variables are the the expansion coefficients in \eqref{eq:sg-ansatz1d},
\begin{equation}\label{eq:swesg41d}
   \hU_t+(\widehat{F}(\hU))_x = \widehat{S}(\hU),
\end{equation}
Here, $\hU \coloneqq (\hh^\top, \hqx^\top)^\top \in \R^{2 K}$, where $\hh, \hqx$ are each length-$K$ vectors of the expansion coefficients in \eqref{eq:sg-ansatz1d}. The flux and the source terms are, 
\begin{align}\label{eq:sgfluxessource1d}
    &\widehat{F}(\hU) = \begin{pmatrix}
     \hqx\\\cP(\hqx)\cP^{-1}(\hh)\hqx+\frac{1}{2}g\cP(\hh)\hh
   \end{pmatrix},&&
&\widehat{S}(\hU) = \begin{pmatrix}0\\-g\cP(\hh)\widehat{B_x}\end{pmatrix},
\end{align}
cf. \eqref{eq:swe-flux-1d}. The flux Jacobian, written in $K\times K$ blocks, is given by 
 \begin{align}\label{eq:x-jacobian1d}
   \frac{\partial \widehat{F}}{\partial \hU}
   &= \begin{pmatrix}O&I\\g\cP(\hh)-\cP(\hqx)\cP^{-1}(\hh)\cP(\hu)&\cP(\hqx)\cP^{-1}(\hh)+\cP(\hu)\end{pmatrix}. 
 \end{align}
We have introduced the term
\begin{align}\label{eq:uPCE1d}
  \hu &= \cP^{-1}(\hh)\hqx,
\end{align} 
which we view as the vector of the PCE coefficients of the $x$-velocity $u$ introduced in \eqref{eq:u-SWE}, and is well-defined if $\mcp(\hh)$ is invertible.

%In our previous work \cite{doi:10.1137/20M1360736}, we derived the following \textit{sufficient condition} to guarantee \eqref{eq:swesg41d} to be hyperbolic.
The deterministic SWE are hyperbolic if the water height $h > 0$; there is a natural extension of this property to the SGSWE.
\begin{theorem}[Theorem 3.1, \cite{doi:10.1137/20M1360736}]\label{thm:hyperbolicity}
If the matrix $\mathcal{P}(\hh)$ is strictly positive definite for every point $(x,t)$ in the computational spatial-temporal domain, then the SG formulation \eqref{eq:swesg41d} is hyperbolic.
\end{theorem}
This is proven by identifying a stochastic extension of the known eigenvector matrix for the deterministic SWE flux Jacobian $\frac{\partial F}{\partial U}$, and using this to show that $\pfpx{\widehat{F}}{\hU}$ is similar to a symmetric matrix and hence \eqref{eq:swesg41d} is hyperbolic \cite{doi:10.1137/20M1360736}.
%, we first find a stochastic variant of the eigenmatrix that diagnolizes the Jacobian of the deterministic shallow water equations. Then we show that the Jacobian $\frac{\partial \widehat{F}}{\partial \hU}$ is similar to a symmetric matrix and thus is hyperbolic.

%In \cite{dai2021hyperbolicity}, we first find a stochastic variant of the eigenmatrix that diagnolizes the Jacobians of the deterministic shallow water equations. Then we show that any linear combinations of the Jacobians $\frac{\partial \widehat{F}}{\partial \hU}$ and $\frac{\partial \widehat{G}}{\partial \hU}$ is similar to a symmetric matrix and thus is hyperbolic.

\section{An Entropy-Entropy Flux Pair for SGSWE systems}\label{sec3}
The formulation \eqref{eq:swesg41d} 
%and \eqref{eq:swesg4} (one- and two-dimensional, respectively), 
will be considered in what follows. Our goal will be to derive entropy-entropy flux pairs for these formulations. %We will show the derivation for the one-dimensional case in detail, and place the details of the two-dimensional case in \Cref{app:proof-eef-sgswe-2d}. 
The first step is for us to recall a known entropy-entropy flux pair for the \textit{deterministic} SWE system. 

%In this section, we first recall an entropy-entropy flux pair for the deterministic shallow water equations. Then we will use it to motivate the derivation of an entropy-entropy flux pair for system \eqref{eq:swesg41d}.
\subsection{Entropy-Entropy Flux Pairs for Deterministic Shallow Water Equations}
It is well-known that solutions to systems of conservation/balance laws can develop shock discontinuities in finite time for generic initial data. Therefore, weak solutions, i.e., solutions in the sense of distributions, are usually considered. However, weak solutions are not necessarily unique, and to mitigate this issue, an additional \textit{entropy admissibility criteria} is imposed \cite{dafermos2016hyperbolic, benzoni2007multi} to identify the physically meaningful solution. 

%\subsubsection{The one-dimensional case}
For a general balance law in one space dimension
\begin{align}\label{eq:1d-balance}
U_t+F(U)_x = S(U),
\end{align}
its entropy-entropy flux pair $(E(U), H(U))$ satisfies a \textit{companion} balance law
\begin{align}\label{eq:econd-11d}
  E(U)_t+H(U)_x = 0%K(U,S),
\end{align}
where the \textit{entropy} $E(U)$ is a scalar function that is
convex in $U$, and $H$ is an \textit{entropy flux} function. 
%To be
%consistent with the original balance law, one defines the right-hand
%side scalar function $K(U,S)$. 
In order to be consistent with the original balance law for smooth $U$, the entropy-entropy flux pair $(E, H)$ should satisfy the following \textit{compatibility condition},
\begin{align}\label{eq:compatibility1d}
  \pfpx{E}{U} (F_x - S) = H_x,
  %\nabla_U H &= (\nabla_U E)^T \nabla_U f, & (\nabla_U E)^T S &= J_x , %\partial_B(H(U,B))B_x.
  %\nabla_U H &= (\nabla_U E)^T \nabla_U f, & (\nabla_U E)^T S &= J_x , %\partial_B(H(U,B))B_x.
\end{align}    
which is simply the condition ensuring that multiplying \eqref{eq:1d-balance} by $\pfpx{E}{U}$ recovers \eqref{eq:econd-11d} when solutions are smooth.
In the case of $S\equiv 0$ and $(E,H)=(E(U), H(U))$, equation \eqref{eq:compatibility1d} is the usual entropy condition for conservation laws. For a general system of balance laws in several spatial dimensions, an entropy-entropy flux pair need not exist. However, for a hyperbolic system of balance laws emerging from continuum physics, the companion balance law \eqref{eq:econd-11d} is usually related to the Second Law of thermodynamics, and the total energy of the system often serves as the entropy function. A variety of examples can be found in \cite[Section 3.3]{dafermos2016hyperbolic}. For the \textit{deterministic} SWE system in \eqref{eq:swesg1-1d}, the total energy \cite{fjordholm2011well} is 
\begin{align}\label{eq:d-efun1d}
E^d(U) = \underbrace{\frac{1}{2}qu}_{\text{kinetic energy}}+\underbrace{\frac{1}{2}gh^2+ghB}_{\text{potential energy}}.
\end{align}
where we recall that $u$ is the velocity defined in \eqref{eq:u-SWE}. For any smooth solution $U$, a direct calculation yields,
\begin{align}\label{eq:sweecond-eq1d}
  E^d(U)_t + H^d(U)_x = 0,
\end{align}
where 
\begin{align}\label{eq:d-eflux1d}
  H^d(U) &= \frac{1}{2}qu^2+gq h + g q B.%, & 
  %J^d(U) &= g q^x B
\end{align}
This, along with the fact that $E^d$ is convex in $U$, establishes that $(E^d, H^d)$ is a valid entropy-entropy flux pair for \eqref{eq:swesg1-1d}. For (weak) solutions with shocks, the entropy admissibility criteria is that energy should dissipate in accordance with a vanishing viscosity principle,
\begin{align}\label{eq:sweecond-dissipative1d}
  E^d(U)_t +  H^d(U)_x \le 0.    
\end{align}

In what follows we will identify entropy-entropy flux pairs for the SGSWE model. This amounts to verifying that (i) such a pair satisfies the companion balance law (an equality for smooth solutions) and (ii) that the entropy function is convex in the state variable.

\subsection{An Entropy-Entropy Flux Pair for the one-dimensional SGSWE}
This section is dedicated to identifying an entropy-entropy flux pair for the SG system \eqref{eq:swesg41d}. 
In this section, we will return to the notation $\hU$ (containing PC expansion coefficients) for the derivation of an entropy-entropy flux pair for the SG system. Our main result in this section is the following entropy entropy-flux pair for the one-dimensional SGSWE:
\begin{theorem}\label{thm:eef-sgswe-1d}
  Define the function,
  \begin{subequations}\label{eq:sg-epair1d}
  \begin{align}\label{eq:sg-efun1d}
    E(\hU) = \frac{1}{2}\left((\hqx)^\top\hu+g\Vert\hh\Vert^2\right)+g\hh^\top\hB,
  \end{align}
  and also the flux function,
  \begin{align}\label{eq:sg-eflux1d}
    H(\hU) &= \frac{1}{2} \hu^\top \mcp(\hqx) \hu + g \hqx^\top \hh + g \hqx^\top \hB, 
  \end{align}
  \end{subequations}
  If $\mcp(\hh) > 0$, then $(E, H)$ is an entropy-entropy flux pair for the one-dimensional SGSWE \eqref{eq:swesg41d}.
\end{theorem}
Recall that $\hu$ above is defined in \eqref{eq:uPCE1d}, and contains PC expansion coefficients for the velocity $u$ defined in \eqref{eq:u-SWE}.
%Motivated by the deterministic total energy \eqref{eq:d-efun1d}, we propose the following function,
%\begin{equation}\label{eq:sg-efun}
%    \begin{aligned}
%        E(\hU) = \frac{1}{2}\left((\hqx)^\top\hu+g\Vert\hh\Vert^2\right)+g\hh^\top\hB   ,
%    \end{aligned}
%\end{equation}
%as an entropy function for the SG system \eqref{eq:swesg41d}-\eqref{eq:sgfluxessource1d}. ()
In the absence of uncertainty, \eqref{eq:sg-efun1d} reduces to the
deterministic total energy \eqref{eq:d-efun1d}. 
The rest of this section is devoted to proving \cref{thm:eef-sgswe-1d}, which amounts to showing that, if $\mcp(\hh) > 0$, then $E$ is convex in $\hU$ and $(E, H)$ satisfy the companion balance law,
%In the rest of
%section, we will show that $E(\hU)$ is indeed a convex function
%in $\hU$ under the sufficient condition $\mcp(\hh) > 0$ for hyperbolicity. Then, we will determine the corresponding entropy flux functions $H(\hU)$ and $J(\hU)$ % the entropy source term $K(\hU, \hB)$ associated with $E(\hU, \hB)$ 
\begin{align}\label{eq:efun-general}
  E(\hU)_t + H(\hU)_x = 0,
\end{align}
for smooth solutions $\hU$. Note that for non-smooth solutions, \eqref{eq:efun-general} holds with $=$ replaced by $\leq$.
%and 
%\begin{align}\label{eq:efun-general}
%E(\hU, \hB)_t + (H(\hU,\hB))_x \le 0,
%\end{align}
%for general weak solutions. 
We prove \Cref{thm:eef-sgswe-1d} with three intermediate results. Our first result is a technical condition that facilitates later computations.
%We start from introducing two lemmas.
\begin{lemma}[Gradient of $\hu$]\label{lem:velderivative}
  Let $\hqx \in \R^K$ be arbitrary, and let $\hh \in \R^K$ be such that $\mcp(\hh)$ is invertible. Defining $\hu$ as in \eqref{eq:uPCE1d}, 
%Assume $\mcp(\hh)$ is invertible. Let 
%\begin{align}\label{eq:vel}
%  \hu &= \mcp^{-1}(\hh)\hqx, 
%\end{align}
  then %the following relations hold
\begin{align}\label{eq:velderivative}
  \pfpx{\hu}{\hU} = \left[ \pfpx{\hu}{\hh}, \;\; \pfpx{\hu}{\hqx} \right] = \left[ -\mcp^{-1}(\hh)\mcp(\hu), \;\;\mcp^{-1}(\hh) \right]
%  \frac{\partial{\hu}}{\partial \hh} &= -\mcp^{-1}(\hh)\mcp(\hu),  &
%    \frac{\partial{\hu}}{\partial \hqx} &= \mcp^{-1}(\hh).
\end{align}
\end{lemma}
\begin{proof}
  If $A(t)$ is a $t$-parameterized matrix, then for any $t$ at which $A$ is invertible,
  \begin{align*}
    \ppx{t} A^{-1}(t) = -A^{-1}(t) \pfpx{A(t)}{t} A^{-1}(t).
  \end{align*}
  Applying this this to $\mcp$, we have,
  \begin{align}\label{eq:Pinvdiff}
    \frac{\partial \mcp^{-1}(\hh)}{\partial \hh_\ell} = 
    -\mcp^{-1}(\hh) \pfpx{\mcp(\hh)}{\hh_\ell}\mcp^{-1}(\hh) \overset{\eqref{eq:pmatrixproperty}}{=} -\mcp^{-1}(\hh)\mathcal{M}_\ell\mcp^{-1}(\hh),
  \end{align}
and hence,
\begin{align}
  \frac{\partial \hu}{\partial \hh_\ell} = \pfpx{\mcp^{-1}(\hh)}{\hh_\ell} \hqx \overset{\eqref{eq:Pinvdiff}}{=} -\mcp^{-1}(\hh)\mathcal{M}_\ell\mcp^{-1}(\hh) \hqx \overset{\eqref{eq:uPCE1d}}{=} \mcp^{-1}(\hh)\mathcal{M}_\ell\hu. 
%    -\mcp^{-1}(\hh)\mathcal{M}_\ell\mcp^{-1}(\hh)\hqx \overset{\eqref{eq:vel}}{=} -\mcp^{-1}(\hh)\mathcal{M}_\ell\hu. 
\end{align}
Therefore, 
\begin{equation}
\begin{aligned}
  &\frac{\partial \hu}{\partial \hh} = \left[-\mcp^{-1}(\hh)\mathcal{M}_1\hu,\;\; \cdots\;\; -\mcp^{-1}(\hh)\mathcal{M}_K\hu\right]
    \overset{\eqref{eq:pmatrixproperty}}{=}-\mcp^{-1}(\hh)\mcp(\hu),
\end{aligned}    
\end{equation}
  proving the desired relation for $\pfpx{\hu}{\hh}$. The relation for $\pfpx{\hu}{\hqx}$ is immediate from the definition \eqref{eq:uPCE1d}.
%Recall the symmetric matrix
%\begin{align}\label{eq:ph}
%&\mcp(\hh) = \sum_{\ell = 1}^{K}\hh_\ell \mathcal{M}_\ell = \left[\mathcal{M}_1\hh|\cdots|\mathcal{M}_K\hh\right],    && (\mathcal{M}_\ell)_{i,j} = \langle\phi_i,\phi_j\phi_\ell\rangle_\rho.
%\end{align}
%Using the identity
%\begin{equation}
%\begin{aligned}
%    &\frac{\partial (\mcp(\hh)\mcp^{-1}(\hh))}{\partial \hh_\ell} = \frac{\partial I}{\partial \hh_\ell}= 0,\\
%    \Rightarrow\;\;&\frac{\partial \mcp(\hh)}{\partial \hh_\ell}\mcp^{-1}(\hh) 
%    +  \mcp(\hh)\frac{\partial \mcp^{-1}(\hh)}{\partial \hh_\ell}= 0,\\
%    \overset{\eqref{eq:ph}}{\Rightarrow}\;\;&\mathcal{M}_\ell\mcp^{-1}(\hh)+\mcp(\hh)\frac{\partial \mcp^{-1}(\hh)}{\partial \hh_\ell} = 0,\\
%    \Rightarrow\;\;& \frac{\partial \mcp^{-1}(\hh)}{\partial \hh_\ell} = -\mcp^{-1}(\hh)\mathcal{M}_\ell\mcp^{-1}(\hh),
%\end{aligned}    
%\end{equation}
%we have
%\begin{align}
%    \frac{\partial \hu}{\partial \hh_\ell} = -\mcp^{-1}(\hh)\mathcal{M}_\ell\mcp^{-1}(\hh)\hqx \overset{\eqref{eq:vel}}{=} -\mcp^{-1}(\hh)\mathcal{M}_\ell\hu. 
%\end{align}
%Therefore, 
%\begin{equation}
%\begin{aligned}
%    &\frac{\partial \hu}{\partial \hh} = \left[-\mcp^{-1}(\hh)\mathcal{M}_1\hu|\cdots|-\mcp^{-1}(\hh)\mathcal{M}_K\hu\right]\\
%    =\;\;&-\mcp^{-1}(\hh)\left[\mathcal{M}_1\hu|\cdots|\mathcal{M}_K\hu\right]\\
%    \overset{\eqref{eq:ph}}{=}\;\;&-\mcp^{-1}(\hh)\mcp(\hu).
%\end{aligned}    
%\end{equation}
%The second equation in \eqref{eq:velderivative} is straightforward. 
\end{proof}

\begin{lemma}[Convexity of $E(\hU)$]\label{lem:sg-econvex}
    If $\mcp(\hh)$ is positive definite, then the function $E(\hU)$ defined in \eqref{eq:sg-efun1d} is convex in $\hU$.
\end{lemma}

\begin{proof}
  Using the definition \eqref{eq:uPCE1d} of $\hu$, note that,
  \begin{align}\label{eq:E-decomp}
    E(\hU) = \underbrace{\frac{1}{2} (\hqx)^\top\mathcal{P}^{-1}(\hh)\hqx}_{f_1(\hU)} + \underbrace{\frac{g}{2} \hh^\top \hh + g \hh^\top \hB}_{f_2(\hU)},
  \end{align}
  and therefore in particular,
  \begin{align}\label{eq:Hessian-decomp}
    \pfnpx{E}{2}{\hU} = \pfnpx{f_1}{2}{\hU} + \pfnpx{f_2}{2}{\hU}.
  \end{align}
  We will show that this Hessian is positive definite. Clearly we have,
    \begin{align}\label{eq:f2-hessian}
 \pfpx{f_2}{\hU} &= \left( g \hh^\top+g \hB^\top, \;\; 0 \right) \in \R^{1 \times 2 K}, & 
    \pfnpx{f_2}{2}{\hU} &= \left(\begin{array}{cc} g I & 0 \\ 0 & 0 \end{array}\right) \in \R^{2 K \times 2 K}.
  \end{align}
  Using the previous lemma, we can directly compute,
  \begin{align}\label{eq:f1-p1}
    \pfpx{f_1}{\hh} &= \frac{1}{2} (\hqx)^\top \pfpx{\hu}{\hh} \overset{\eqref{eq:velderivative}, \eqref{eq:uPCE1d}}{=} -\frac{1}{2} \hu^\top \mcp(\hu), &
    \pfpx{f_1}{\hqx} &=  (\hqx)^\top \mcp^{-1}(\hh) = \hu^\top, & 
  \end{align}
  which in turn implies,
  \begin{align*}
    \pfnpx{f_1}{2}{\hqx} &= \mcp^{-1}(\hh),  &
    \frac{\partial^2 f_1}{\partial \hh \partial \hqx} &\overset{\eqref{eq:velderivative}}{=} \left(-\mcp^{-1}(\hh) \mcp(\hu)\right)^\top = -\mcp(\hu) \mcp^{-1}(\hh),
  \end{align*}
  and finally,
  \begin{align*}
    \pfnpx{f_1}{2}{\hh} = \frac{1}{2} \ppx{\hh} \left( -\hu^\top \mcp(\hu)\right) \overset{\eqref{eq:velderivative}}{=} \mcp(\hu) \mcp^{-1}(\hh) \mcp(\hu)
  \end{align*}
  Hence, the Hessian of $f_1$ is,
  \begin{align*}
    \pfnpx{f_1}{2}{\hU} = \left(\begin{array}{cc} \mcp(\hu) \mcp^{-1}(\hh) \mcp(\hu) & -\mcp(\hu) \mcp^{-1}(\hh) \\ -\mcp^{-1}(\hh) \mcp(\hu) & \mcp^{-1}(\hh) \end{array}\right).
  \end{align*}
  A direct computation of the quadratic form associated to this Hessian using an arbitrary vector $(w_1^\top, w_2^\top)^\top \in R^{2 K}$ yields,
  \begin{align*}
    (w_1^\top, w_2^\top) \pfnpx{f_1}{2}{\hU} \left(\begin{array}{c} w_1 \\ w_2 \end{array}\right) = \left( \mcp(\hu) w_1 - w_2\right)^\top \mcp^{-1}(\hh) \left( \mcp(\hu) w_1 - w_2\right) \geq 0.
  \end{align*}
  Finally, combining the above with \eqref{eq:Hessian-decomp} and \eqref{eq:f2-hessian} yields,
    \begin{align*}
    (w_1^\top, w_2^\top)\pfnpx{E}{2}{\hU} \left(\begin{array}{c} w_1 \\ w_2 \end{array}\right) = g \|w_1\|^2 + \left( \mcp(\hu) w_1 - w_2\right)^\top \mcp^{-1}(\hh) \left( \mcp(\hu) w_1 - w_2\right),
  \end{align*}
  which is non-negative since $\mcp(\hh)$ is positive-definite. Therefore, $E$ is convex, as desired. In addition, since the above expression vanishes if and only if $w_1 = w_2 = 0$, then $E$ is also strictly convex.
%  since $\mcp(\hh)$ is positive definite, 
%  with equality if and only if $w_2 = \mcp(\hu) w_1$. Hence, $f_1$ is convex.
%  We will show that $f_1$ and $f_2$, and $f_3$ are all convex, and that their sum corresponds to a strictly convex function. The function $f_3$ is clearly (strictly) convex. Since $f_1$ does not depend on $\hqy$, we need only analyze $(\hh^\top, \hqx^\top)^\top \mapsto f_1(\hU)$. From the form of $f_1$, we have,
\end{proof}

%\begin{proof}
%Recall that $\hU = (\hh^\top,\hqx^\top)^\top$. Using \eqref{eq:velderivative}, we have
%\begin{equation}\label{eq:sg-evar-comp}
%\begin{aligned}
%    \frac{\partial E}{\partial \hh} &= \frac{1}{2}(-(\hqx)^\top\mcp^{-1}(\hh)\mcp(\hu) + g(\hh+\hB)^\top\\
%    & = -\frac{1}{2}\hu^\top\mcp(\hu)+ g(\hh+\hB)^\top,\\
%    \frac{\partial E}{\partial \hqx} &= \frac{1}{2}(\hu^\top + (\hqx)^\top\mcp^{-1}(\hh))= \hu^\top.
%\end{aligned}
%\end{equation}
%Applying \eqref{eq:velderivative} again, we have 
%
%\begin{equation*}
% \begin{aligned}
%     \frac{\partial^2 E}{\partial \hU^2} &= \begin{pmatrix}
%     \dfrac{\partial^2 E}{\partial \hh^2}&\dfrac{\partial^2 E}{\partial \hh \partial \hq}\\[2pt]
%     \dfrac{\partial^2 E}{\partial \hq\partial \hh}&\dfrac{\partial^2 E}{\partial \hq^2}
%     \end{pmatrix}
%     \\&=\begin{pmatrix}
%     \mcp(\hu)\mcp^{-1}(\hh)\mcp(\hu) + gI&-\mcp(\hu)\mcp^{-1}(\hh)\\
%     -\mcp^{-1}(\hh)\mcp(\hu)&\mcp^{-1}(\hh).
%     \end{pmatrix}
% \end{aligned}
% \end{equation*}
%Consider an arbitrary nonzero vector $(x^\top,y^\top)^\top\in\mathbb{R}^{2K}$. Then, using the positive-definiteness of $\mcp(\hh)$, we have
% {\small
% \begin{align*}
%     &(x^\top,y^\top)^\top\frac{\partial^2 E}{\partial \hU^2}\begin{pmatrix}
%     x\\y
%     \end{pmatrix}\\
%     = &g(x^\top x) + (\mcp(\hu)x - y)^\top\mcp^{-1}(\hh)(\mcp(\hu)x - y)>0,
% \end{align*}
% }
% i.e., $E(\hU,\hB)$ is a convex function with respect to $\hU$.
%\end{proof}
The final piece needed to prove \Cref{thm:eef-sgswe-1d} is to establish that the entropy function $E$ along with the flux function $H$ defined in \eqref{eq:sg-eflux1d} satisfy the companion balance law.
%Using \Cref{lem:velderivative} and \Cref{lem:sg-econvex}, we can determine an entropy-entropy flux pair for \eqref{eq:swesg41d}, which is essentially a stochastic variant to \eqref{eq:d-efun1d}-\eqref{eq:d-eflux1d}. 
\begin{lemma}[$(E,H)$ satisfy the companion balance law]\label{eq:eef-1d-companion}
  When $U$ is a smooth function, the pair $(E,H)$ defined in \eqref{eq:sg-epair1d} satisfies
%    Let 
%    \begin{equation}\label{eq:sg-eflux}
%        \begin{aligned}
%            &H(\hU, \hB) = \frac{1}{2}\hu^\top\mcp(\hqx)\hu +g\hqx^\top(\hh+\hB).\\       
%        \end{aligned}
%    \end{equation}    
%    Then for any smooth solution $\hU$ of \eqref{eq:swesg41d}, we have
    \begin{align}\label{eq:sgecond-eq}
        E(\hU)_t + H(\hU)_x = 0.
    \end{align}
\end{lemma}
\begin{proof}
  The compatibility condition we seek to show, equivalent to \eqref{eq:sgecond-eq}, is,
  \begin{align}\label{eq:compatibility-1d}
    \pfpx{E}{\hU} \left( \pfpx{\hF}{\hU} \pfpx{\hU}{x} - \hS \right) = \pfpx{H}{x},
  \end{align}
  cf. \eqref{eq:compatibility1d}. To proceed we split both entropy functions into two pieces:
  \begin{subequations}
  \begin{align}
    E(\hU) &= E_1(\hU) + E_2(\hU), & E_1(\hU) &\coloneqq \frac{1}{2} \left( (\hqx)^\top \hu + g \|\hh\|^2 \right), & E_2(\hU) &\coloneqq g \hh^\top \hB, \\\label{eq:H-decomp}
    H(\hU) &= H_1(\hU) + H_2(\hU), & H_1(\hU) &\coloneqq \frac{1}{2} \hu^\top \mcp(\hqx) \hu + g (\hqx)^\top \hh, & H_2(\hU) &= g (\hqx)^\top \hB
  \end{align}
  \end{subequations}
  From \eqref{eq:E-decomp}, \eqref{eq:f2-hessian}, and \eqref{eq:f1-p1}, we have already computed the gradient of $E$:
  \begin{align}\label{eq:E-gradient}
    \pfpx{E_1}{\hU} &= \left( -\frac{1}{2} \hu^\top \mcp(\hu) + g \hh^\top, \;\; \hu^\top \right), & 
    \pfpx{E_2}{\hU} &= \left( g \hB^\top, \;\; 0 \right).
  \end{align}
  Combining these expressions with the flux Jacobian in \eqref{eq:x-jacobian1d} and the source term in \eqref{eq:sgfluxessource1d} yields,
  \begin{subequations}\label{eq:compatibility-decomp}
  \begin{align}
    -\pfpx{E_1}{\hU} \hS + \pfpx{E_2}{\hU} \left( \pfpx{\hF}{\hU}\pfpx{\hU}{x} - \hS \right) = g \hB^\top \pfpx{\hqx}{x} + g \hqx^\top \hB_x \stackrel{\eqref{eq:H-decomp}}{=} \pfpx{H_2}{x}
  \end{align}
  Note then that if we are able to show,
  \begin{align}\label{eq:cl-compatibility}
    \pfpx{E_1}{\hU} \pfpx{\hF}{\hU} = \pfpx{H_1}{\hU},
  \end{align}
  \end{subequations}
  then the expressions \eqref{eq:compatibility-decomp} are equivalent to \eqref{eq:compatibility-1d}. Therefore, we are left only to show \eqref{eq:cl-compatibility}. A direct computation with \eqref{eq:E-gradient} and \eqref{eq:x-jacobian1d} yields,
  \begin{align*}
    \pfpx{E_1}{\hU} \pfpx{\hF}{\hU} = \left( g (\hqx)^\top - \hu^\top \mcp(\hqx) \mcp^{-1}(\hh) \mcp(\hu), \;\;\; g \hh^\top + \frac{1}{2} \hu^\top \mcp(\hu) + \hu^\top \mcp(\hqx) \mcp^{-1}(\hh)  \right)
  \end{align*}
  On the other hand, we have the expressions,
  \begin{align*}
    \ppx{\hh} \frac{1}{2} \hu^\top \mcp(\hqx) \hu &= \hu^\top \mcp(\hqx) \pfpx{\hu}{\hh} \stackrel{\eqref{eq:velderivative}}{=} -\hu^\top \mcp(\hqx) \mcp^{-1}(\hh) \mcp(\hu), \\
    \ppx{\hqx} \frac{1}{2} \hu^\top \mcp(\hqx) \hu &= \hu^\top \mcp(\hqx) \pfpx{\hu}{\hqx} + \frac{1}{2} \left( \ppx{\hqx} z^\top \mcp(\hqx) z\right)\Big|_{z\gets \hu} \stackrel{\eqref{eq:acb},\eqref{eq:velderivative}}{=} \hu^\top \mcp(\hqx) \mcp^{-1}(\hh) + \frac{1}{2} \hu^\top \mcp(\hu) 
  \end{align*}
  and using these to compute $\pfpx{H_1}{\hU}$ shows that \eqref{eq:cl-compatibility} is true, completing the proof.
\end{proof}
The proof of \Cref{thm:eef-sgswe-1d} is complete: Lemmas \eqref{lem:sg-econvex} and \eqref{eq:eef-1d-companion} imply that $(E,H)$ as defined in \eqref{eq:sg-epair1d} are an entropy-entropy flux pair for \eqref{eq:swesg41d}.

\begin{remark}
The quantities,
\begin{align}\label{entrvar}
  \hV \coloneqq  \left(\pfpx{E} {\hU}\right)^\top &=\left(-\frac{1}{2}\hu^\top\mcp(\hu) + g(\hh+\hB)^\top,\hu^\top\right)^\top, & %^\top, & 
  \Psi &\coloneqq \hV \widehat{F} - H \stackrel{\eqref{eq:uPCE1d},\eqref{eq:sg-eflux1d}}{=} \frac{1}{2}g\hu^\top\mcp(\hh)\hh,
\end{align} 
  are called the \textit{entropy variable} and \textit{stochastic energy potential}, respectively. These variables serve important roles in the construction of the energy conservative and the energy stable schemes that we develop later.
\end{remark}

\section{Well-Balanced Energy Conservative And Energy Stable Schemes
  for SG 1D SWE}\label{sec:schemes}
In this section, we present several well-balanced energy conservative
and energy stable numerical scheme for SG SWE. 
%For notational
%simplicity, we consider the SG formulation for 1D SWE
%\eqref{eq:swesg41d} in this section and extend the results for
%\eqref{eq:swesg4} in the later section. 
The schemes designed below are
stochastic extensions of the schemes developed in
\cite{fjordholm2011well}. Our entropy-entropy flux pairs developed in \Cref{sec3} will be crucial ingredients for energy conservative and energy stable schemes for the SG formulation \eqref{eq:swesg41d}-\eqref{eq:uPCE1d}.
%constructed entropy pair and entropy based formulation in
%Section~\ref{sec3} which will be used to construct schemes below.

    %     \subsection{SG Formulation for 1D SWE}

We also need to specify the well-balanced property we are interested in: By ``well-balanced'', we mean that the scheme can preserve the stochastic ``lake-at-rest'' state exactly at the discrete level. 
\begin{definition}[Well-Balanced SGSWE Property, \cite{doi:10.1137/20M1360736}]\label{def:stochastic-lake-at-rest-1d}
  We say that a solution $(h_P, q_P)$ to \eqref{eq:swesg41d} is well-balanced if it satisfies the \emph{stochastic} ``lake-at-rest'' solution,
    \begin{equation}\label{eq:lake-at-rest}
      q_P(x,t,\xi) \equiv 0,\quad h_P(x,t,\xi) + \Pi_P[B](x,t,\xi) \equiv C(\xi),
    \end{equation}    
    where $C(\xi)$ is a random scalar depending only on $\xi$, $\Pi_P$ corresponds to a polynomial truncation, cf. \eqref{eq:PCEex}, and subscripts $P$ refer to the (stochastic) discrete solution on the subspace $P$.
%    represents the truncations of the original random field (see \eqref{eq:sg-ansatz1d}). 
  In terms of our previous notation for $P$-expansion coefficients, equation \eqref{eq:lake-at-rest} is equivalent to the following vector equation
    \begin{equation}\label{eq:lake-at-rest-vec}
      \hq \equiv \boldsymbol{0},\quad \hh + \hB \equiv \widehat{C},\;\; \forall (x,t)\in \mathcal{D}\times[0,T],
    \end{equation}        
    where $\mathcal{D}$ is the spatial domain and $T$ is the terminal time.
\end{definition}
We emphasize that even without introducing the lake-at-rest definition \eqref{eq:lake-at-rest}, the vector equation \eqref{eq:lake-at-rest-vec} itself is a steady state of the SG system \eqref{eq:swesg41d}.
\subsection{Energy Conservative Schemes}
We consider the semi-discrete form for FV schemes for \eqref{eq:swesg41d} over a uniform mesh in the $x$ variable:
\begin{align}\label{eq:fvsemi-discrete}
  \dfrac{d}{dt}\bU_{i} = -\dfrac{\mathcal{F}_{i+\frac{1}{2}}-\mathcal{F}_{i-\frac{1}{2}}}{\Delta x}+\boldsymbol{S}_{i}.
\end{align}
Here, $\bU_{i}\approx
\frac{1}{\Delta x}\int_{\mathcal{I}_{i}} \hU(x,t)dx$
is the approximation of the cell averages of $\hU$ over cells
$\mathcal{I}_{i} \coloneqq [ x_{i-1/2}, x_{i+1/2}]$ at time $t$, and $\Delta x = |\mathcal{I}_i| = x_{i+1/2} - x_{i-1/2}$. The terms $\mathcal{F}_{i\pm1/2}$ are numerical fluxes at the boundaries of the
cells, which are functions of neighboring states, e.g., $\mathcal{F}_{i+1/2}$ is a function of $\bU_i$ and $\bU_{i+1}$.
The term $\bs{S}_{i}\approx
\frac{1}{\Delta x}\int_{\mathcal{I}_{i}}
\widehat{S}(\hU,\hB)dx$ is a discretization of the
source term, which we will design below to be well-balanced. To reiterate our notation: normal typeset capital letters (sometimes with ``hat'' notation) refers to degrees of freedom associated to discretizing \textit{only} the stochastic variable $\xi$, i.e., $(\hU, \hh, \hq, \hB)$. Boldface notation with subscripts $i$ refers to degrees of freedom associated to a subsequent discretization of the spatial variable $x$ over cell $\mathcal{I}_i$, i.e., $(\bs{U}_i, \bs{h}_i, \bs{q}_i, \bs{B}_i)$. %Similar hat versus boldface notation will apply to other dependent variables, e.g., $\hB$ versus $\bB_i$.
We define the discrete velocity variable $\bs{u}_{i}$ in a manner analogous to \eqref{eq:uPCE1d}:
\begin{align}\label{eq:ui-def}
  \bs{u}_i \coloneqq \mcp(\bs{h}_i)^{-1} \bs{q}_i.
\end{align}
Discrete entropic quantities are derived from the discrete conservative variables $\bs{U}_i$ and velocity variable $\bs{u}_i$. I.e., the following are direct generalizations of the definition of $E(\hU)$ in \eqref{eq:sg-efun1d}, and of $(\hV,\Psi)$ in \eqref{entrvar}:
\begin{subequations}
\begin{align}
  \label{eq:Ei-def} \bs{E}_i &\coloneqq \frac{1}{2} \left( \bs{q}_i^\top \bs{u}_i + g \| \bs{h}_i\|^2 \right) + g \bs{h}_i^\top \bs{B}_i, \\
  \label{eq:Vi-def} \bs{V}_i &\coloneqq \left(\pfpx{\bs{E}_i}{\bs{U}_i}\right)^\top = \left(-\frac{1}{2} \bs{u}_i^\top \mcp(\bs{u}_i) + g(\bs{h}_i + \bs{B}_i)^\top, \;\; \bs{u}_i^\top \right)^\top, \\
  \label{eq:Psii-def} \bs{\Psi}_i &\coloneqq \bs{V}_i \hF(\bs{U}_i) - H(\bs{U}_i) = \frac{1}{2} g \bs{u}_i^\top \mcp(\bs{h}_i) \bs{h}_i
\end{align}
\end{subequations}

%Note that, since the
%``hat'' vector notation with subscript has been used for the notation
%of polynomial moments, we will use the bold letter $\boldsymbol{U}$
%with subscripts to denote the cell averages of the vector
%$\hU$. Similar notation will be used for other variables, for example,
%$(\boldsymbol{B}, \hB), (\bV,\hV)$, and etc. 

We now introduce some notation that is used in \cite{fjordholm2011well} for averages and jumps at cell interfaces:
\begin{align}\label{eq:jump-and-average}
  \overline{\bs{a}}_{i+1/2} &\coloneqq \frac{1}{2}(\bs{a}_{i+1} + \bs{a}_{i}), &
  \llbracket \bs{a}\rrbracket_{i+1/2} &\coloneqq \bs{a}_{i+1} - \bs{a}_{i},
\end{align}
where $\bs{a}_i$ is the cell average over $\mathcal{I}_i$. The expressions above are equivalent to,
\begin{align}\label{eq:jump-average-property}
  \bs{a}_i = \overline{\bs{a}}_{i+1/2} - \frac{\llbracket \bs{a}\rrbracket_{i+1/2}}{2} = \overline{\bs{a}}_{i-1/2} + \frac{\llbracket \bs{a}\rrbracket_{i-1/2}}{2},
\end{align}
and all these expressions are valid regardless of the size of $\bs{a}$ (e.g., both row and column vectors are allowed). We will require some additional technical results for interfacial averages and jumps.
\begin{lemma}
  Let $\bs{a}_i, \bs{b}_i$ be any spatially discrete quantities. Then:
  \begin{subequations}\label{eq:relations}
  \begin{align}\label{eq:relation1}
    \mcp(\overline{\boldsymbol{a}}\ph)\dbracket{\boldsymbol{a}}\ph &= 
    \frac{1}{2}\dbracket{\mcp(\boldsymbol{a})\boldsymbol{a}}\ph. \\\label{eq:relation2}
    \dbracket{\boldsymbol{a}}\ph^\top\overline{\boldsymbol{b}}\ph+\dbracket{\boldsymbol{b}}\ph^\top\overline{\boldsymbol{a}}\ph &= 
    \dbracket{\boldsymbol{a}^\top\boldsymbol{b}}\ph
  \end{align}
  \end{subequations}
\end{lemma}
\begin{proof}
Due to linearity of $\mcp$, then,
    \begin{align*}
        \mcp(\overline{\boldsymbol{a}}_{i+\frac{1}{2}})\dbracket{\boldsymbol{a}}_{i+\frac{1}{2}} &\overset{\eqref{eq:pmatrix}}{=} \mcp\left(\frac{1}{2}(\boldsymbol{a}_{i+1}+\boldsymbol{a}_{i})\right)(\boldsymbol{a}_{i+1}-\boldsymbol{a}_{i})
        \overset{\eqref{eq:commute}}{=} \frac{1}{2}\left(\mcp(\boldsymbol{a}_{i+1})\boldsymbol{a}_{i+1}-\mcp(\boldsymbol{a}_{i})\boldsymbol{a}_{i}\right)
        = \frac{1}{2}\dbracket{\mcp(\boldsymbol{a})\boldsymbol{a}}_{i+\frac{1}{2}},
    \end{align*}        
     which proves \eqref{eq:relation1}. Similarly, \eqref{eq:relation2} can be proven directly:%For any two quantity $\boldsymbol{a}$ and $\boldsymbol{b}$,
     \begin{align*}
       \dbracket{\boldsymbol{a}}_{i+\frac{1}{2}}^\top\overline{\boldsymbol{b}}_{i+\frac{1}{2}}+\dbracket{\boldsymbol{b}}_{i+\frac{1}{2}}^\top\overline{\boldsymbol{a}}_{i+\frac{1}{2}} &= \frac{1}{2}\left(\left(\boldsymbol{a}_{i+1}-\boldsymbol{a}_{i}\right)^\top\left(\boldsymbol{b}_{i+1}+\boldsymbol{b}_{i}\right)+\left(\boldsymbol{b}_{i+1}-\boldsymbol{b}_{i}\right)^\top\left(\boldsymbol{a}_{i+1}+\boldsymbol{a}_{i}\right)\right) \\
         &= \boldsymbol{a}_{i+1}^\top\boldsymbol{b}_{i+1} - \boldsymbol{a}_{i}^\top\boldsymbol{b}_{i}
         = \dbracket{\boldsymbol{a}^\top\boldsymbol{b}}_{i+\frac{1}{2}}.
     \end{align*}
\end{proof}

%Using the \textit{entropy variables} $\hV$ defined in \eqref{entrvar}, we introduce the \textit{stochastic energy potential} $\Psi$,
%\begin{equation}\label{eq:energypotential}
%    \begin{aligned}
%        &\Psi \coloneqq \hV^\top\widehat{F} - H \overset{\eqref{eq:vel},\eqref{eq:sg-eflux}}{=\joinrel=\joinrel=} \frac{1}{2}g\hu^\top\mcp(\hh)\hh.\\
%    \end{aligned}
%\end{equation}

We now make particular definitions for energy conservative and energy stable schemes for one-dimensional systems of balance laws. To provide context, with no source terms (i.e. $\bs{S}_i = 0$) then spatial discretizations of the form \eqref{eq:fvsemi-discrete} are called \textit{conservative} schemes since they imply,
\begin{align}\label{eq:ec-global}
  \frac{d}{dt} \sum_{i \in [M]} \Delta x \bU_i(t) = \left[ \mathcal{F}_{1/2} - \mathcal{F}_{M+1/2} \right], \hskip 20pt \textrm{(vanishing source, $\bs{S}_i = 0$)},
\end{align}
and in particular with periodic boundary conditions, then this implies that the cumulative amount of $\hU$ in the system is constant in time.\footnote{For non-periodic boundary conditions, the energy would increase/decrease depending on the boundary conditions and their corresponding impact on the boundary fluxes.}

To translate this concept to the notion of an energy conservative scheme, note that an entropy-entropy flux pair $(E,H)$ introduced in \cref{sec3} is explicitly a function of the state $\hU$ and inputs in the source term (here, $\hB$). Hence, the semi-discrete form \eqref{eq:fvsemi-discrete} can be transformed into a semi-discrete form for the companion balance law \eqref{eq:efun-general}. Then we call \eqref{eq:fvsemi-discrete} energy conservative if it implies a conservative scheme for the companion balance law that describes the evolution of the entropy (energy).
\begin{definition}[Energy conservative and energy stable schemes]
  Suppose that the system of balance laws \eqref{eq:swesg41d} has an entropy-entropy flux pair $(E,H)$ where $E(\hU)$ can be interpreted as energy for the system. Then the semi-discrete FV scheme \eqref{eq:fvsemi-discrete} is an Energy Conservative (EC) scheme if it can be rewritten as the following semi-discrete form for the evolution of the numerical cell averages $\bs{E}_i$ of $E$:
  \begin{align}\label{eq:ec-scheme-def}
    \frac{d}{dt} \bs{E}_i(t) &= -\frac{1}{\Delta x} \left( \mathcal{H}_{i+1/2} - \mathcal{H}_{i-1/2}\right), & i &\in [M],
  \end{align}
  where $\mathcal{H}_{i+1/2}$ is some numerical entropy flux at the interface location $x = x_{i+1/2}$. The scheme \eqref{eq:fvsemi-discrete} is called an Energy Stable (ES) scheme if
  \begin{align}\label{eq:es-condition}
    \frac{d}{dt} \bs{E}_i(t) &\leq -\frac{1}{\Delta x} \left( \mathcal{H}_{i+1/2} - \mathcal{H}_{i-1/2}\right), & i &\in [M].
  \end{align}
\end{definition}
Note that the definitions above are cell-wise conditions that are stronger than a global condition such as \eqref{eq:ec-global}.

\subsection{An EC scheme for the SGSWE}
In this section we present an EC scheme for the one-dimensional SGSWE system \eqref{eq:swesg41d}. We use the conservative scheme \eqref{eq:fvsemi-discrete}, with the following choices of flux and source terms:
\begin{subequations}\label{eq:sgswe-ec}
  \begin{align}\label{eq:ec-flux}
  \mathcal{F}_{i+1/2} = \mathcal{F}^{\mathrm{EC}}_{i+1/2} &\coloneqq \begin{pmatrix}
    \mcp(\bbh_{i+\frac{1}{2}})\bbu_{i+\frac{1}{2}}\vspace{0.5em}\\
    \frac{g}{2}\left(\overline{\mcp(\bh)\bh}\right)_{i+\frac{1}{2}} + \mcp(\bbu_{i+\frac{1}{2}})\mcp(\bbh_{i+\frac{1}{2}})\bbu_{i+\frac{1}{2}}
  \end{pmatrix}, \\\label{eq:ec-source}
    \boldsymbol{S}_{i} &= \begin{pmatrix}
    \boldsymbol{0}\vspace{0.5em}\\
    -\frac{g}{2\Delta x}\left(\mcp(\bbh_{i+\frac{1}{2}})\dbracket{\boldsymbol{B}}_{i+\frac{1}{2}}+\mcp(\bbh_{i-\frac{1}{2}})\dbracket{\boldsymbol{B}}_{i-\frac{1}{2}}\right).
    \end{pmatrix}
\end{align}
\end{subequations}
Above, the interfacial averages $\bbu_{i+1/2}$ are computed as defined in \eqref{eq:jump-and-average}. 
Our main result for this scheme is as follows.
\begin{theorem}[EC Scheme]\label{thm:ec-scheme}
  Suppose the bottom topography function $B$ is independent of
  time. Consider the semi-discrete scheme \eqref{eq:fvsemi-discrete}
  for the SGSWE system \eqref{eq:swesg41d}. Suppose that the flux and
  source terms are selected as in \eqref{eq:sgswe-ec}. Then, this is a well-balanced EC scheme with local truncation error $\mathcal{O}(\Delta x^2)$.
\end{theorem}
The remainder of this section is devoted to the proof, which requires some intermediate steps. First, we show that $\bs{S}_i$ is a well-balanced choice for the source term discretization.
\begin{lemma}\label{lem:ec-well-balanced}
  Suppose $\bs{S}_i$ is chosen as in \eqref{eq:ec-source}. If the bottom topography $B$ is independent of time, then \eqref{eq:fvsemi-discrete} is a well-balanced scheme in the sense of \cref{def:stochastic-lake-at-rest-1d}.
\end{lemma}
\begin{proof}
Given initial data 
\begin{align}\label{eq:wbinit}
    &\bu_i \equiv \boldsymbol{0},&&\bh_i + \boldsymbol{B}_i = \text{const vector},\;\;\forall i,
\end{align}
the well-balanced property with time-independent bottom topography (see \cref{def:stochastic-lake-at-rest-1d}) requires that, for every $i$,
\begin{align}\label{eq:wb-goal}
  \frac{d}{dt}\boldsymbol{\bh}_i &\equiv 0, &\frac{d}{dt}\boldsymbol{\bq}_i &\equiv 0.
\end{align}
We first notice that, 
\begin{equation}\label{eq:relation3}
\begin{aligned}
    \left(\overline{\mcp(\bh)\bh}\right)_{i+\frac{1}{2}} - \left(\overline{\mcp(\bh)\bh}\right)_{i-\frac{1}{2}}
    &=\frac{1}{2}\left(\dbracket{\mcp(\bh)\bh}_{i+\frac{1}{2}} + \dbracket{\mcp(\bh)\bh}_{i-\frac{1}{2}}\right) \\
    &\overset{\eqref{eq:relation1}}{=}\mcp(\bbh_{i+\frac{1}{2}})\dbracket{\bh}_{i+\frac{1}{2}} + \mcp(\bbh_{i-\frac{1}{2}})\dbracket{\bh}_{i-\frac{1}{2}}.
\end{aligned}
\end{equation}
  Note that with initialization \eqref{eq:wbinit}, then $\bs{u}_i = 0$, and hence $\bbu_{i+1/2} = 0$.
  %Assume that the solution is well-balanced at current time level, i.e.,
%equation \eqref{eq:wbinit} is satisfied.  Then
%$\bu_{i+\frac{1}{2}}\equiv 0$ (due to a second-order
  %piecewise linear reconstruction procedure)\footnote{TODO: We haven't defined reconstruction procedures.}, 
  Therefore the semi-discrete scheme \eqref{eq:fvsemi-discrete} with the flux and source terms in \eqref{eq:sgswe-ec} yields,
\begin{align}
    \frac{d}{dt}\bh_i = &-\frac{1}{\Delta x}\left(\mcp(\bbh_{i+\frac{1}{2}})\bbu_{i+\frac{1}{2}} - \mcp(\bbh_{i-\frac{1}{2}})\bbu_{i-\frac{1}{2}}\right) = \boldsymbol{0},\label{eq:wbh}\\
    \frac{d}{dt}\bq_i = &-\frac{g}{2\Delta x}\left(\left(\overline{\mcp(\bh)\bh}\right)_{i+\frac{1}{2}} - \left(\overline{\mcp(\bh)\bh}\right)_{i-\frac{1}{2}}\right)\nonumber\\
    &-\frac{g}{2\Delta x}\left(\mcp(\bbh_{i+\frac{1}{2}})\dbracket{\boldsymbol{B}}_{i+\frac{1}{2}}+\mcp(\bbh_{i-\frac{1}{2}})\dbracket{\boldsymbol{B}}_{i-\frac{1}{2}}\right)\nonumber\\\nonumber
    \overset{\eqref{eq:relation3}}{=}&-\frac{g}{2\Delta x}\left(\mcp(\bbh_{i+\frac{1}{2}})\dbracket{\bh+\boldsymbol{B}}_{i+\frac{1}{2}}+\mcp(\bbh_{i-\frac{1}{2}})\dbracket{\bh+\boldsymbol{B}}_{i-\frac{1}{2}}\right) \stackrel{\eqref{eq:wbinit}}{=} \bs{0},%\label{eq:wbq}.
\end{align}
which establishes \eqref{eq:wb-goal}.
\end{proof}

\begin{lemma}\label{lem:ec-lte-2}
  The flux and source terms in \eqref{eq:sgswe-ec} commit a local truncation error of $\mathcal{O}(\Delta x^2)$.
\end{lemma}
The proof is direct, by assuming $(\bs{U}_i, \bs{B}_i)$ are exact cell averages of spatially smooth functions $(\hU, \hB)$ and then comparing $\mathcal{F}_{i+1/2}$ and $\bs{S}_i$ to $\hF(\hU)\big|_{x=x_{i+1/2}}$ and $\hS(\hU)\big|_{x=x_{i}}$, respectively, where $\hF$ and $\hS$ are the exact flux and source functions in \eqref{eq:sgfluxessource1d}. Therefore we omit most details, pointing out only the following quantitative approximations in space (ignoring the time variable $t$):
\begin{align*}
  \overline{\bs{U}}_{i+1/2} &= \hU(x_{i+1/2}) + \mathcal{O}(\Delta x^2), & 
  \dbracket{\bs{U}}_{i+1/2} &= \Delta x\, \hU_x\left(x_{i+1/2}\right) + \mathcal{O}(\Delta x^2) \\ 
  \mcp(\overline{\bs{h}}_{i+1/2}) &= \mcp(\hh(x_{i+1/2})) + \mathcal{O}(\Delta x^2), & 
  \overline{\bs{u}}_{i+1/2} &= \hu(x_{i+1/2}) + \mathcal{O}(\Delta x^2).
\end{align*}
Note that the implicit constants hidden in the asymptotic notation above depend on the maximum singular value of $\mcp(\hh_{xx}(x))$ and the minimum singular value of $\mcp(\hh(x_{i+1/2}))$.

The final result we need is a sufficient condition for a numerical flux to result in an EC scheme.
\begin{lemma}\label{lem:e-conserved}
  Let $\bs{S}_i$ be chosen as in \eqref{eq:ec-source}. Suppose that $\mathcal{F}_{i+1/2}$ satisfies
\begin{align}\label{eq:e-conserved-cond}
  \dbracket{\bV}_{i+\frac{1}{2}}^\top \mathcal{F}_{i+\frac{1}{2}} = \dbracket{\bs{\Psi}}_{i+\frac{1}{2}} + g\dbracket{\boldsymbol{B}}_{i+\frac{1}{2}}^\top\mcp(\bbh_{i+\frac{1}{2}})\bbu_{i+\frac{1}{2}}.
\end{align}
%where $\langle\cdot,\rangle$ denotes the dot product.
%  , and the source term is selected to be 
%\begin{equation}\label{eq:e-conserved-source}
%    \boldsymbol{S}_{i} = \begin{pmatrix}
%    \boldsymbol{0}\vspace{0.5em}\\
%    -\frac{g}{2\Delta x}\left(\mcp(\bbh_{i+\frac{1}{2}})\dbracket{\boldsymbol{B}}_{i+\frac{1}{2}}+\mcp(\bbh_{i-\frac{1}{2}})\dbracket{\boldsymbol{B}}_{i-\frac{1}{2}}\right).
%    \end{pmatrix}
%\end{equation}
  Then the corresponding FV scheme \eqref{eq:fvsemi-discrete} is an EC scheme, i.e., satisfies \eqref{eq:ec-scheme-def}, where the numerical energy flux is given by,
%\begin{equation}
%    \frac{d}{dt}E_{i} = -\frac{1}{\Delta x}\left( \mathcal{H}_{i+\frac{1}{2}} - \mathcal{H}_{i-\frac{1}{2}}\right),
%\end{equation}
%where the numerical energy flux is given by 
\begin{align}\label{eq:e-flux}
  \mathcal{H}_{i+\frac{1}{2}} &\coloneqq \overline{\bV}_{i+\frac{1}{2}}^\top \mathcal{F}_{i+\frac{1}{2}} - \overline{\bs{\Psi}}_{i+\frac{1}{2}} - \frac{g}{4}\dbracket{\boldsymbol{B}}_{i+\frac{1}{2}}^\top\mcp(\bbh_{i+\frac{1}{2}})\dbracket{\bu}_{i+\frac{1}{2}}.
\end{align}
%and assume the discrete periodic flux boundary conditions,
%$\mathcal{H}_{M+\frac{1}{2}}=\mathcal{H}_{-M-\frac{1}{2}}$. 
%Then, the discrete total energy is preserved, i.e., 
%\begin{align}
%    \sum_{i}E_{i}(t)\Delta x \equiv \sum_{i}E_{i}(0)\Delta x.
%\end{align}
\end{lemma}
\begin{proof}
  Multiplying \eqref{eq:fvsemi-discrete} by $\bV_{i}^\top$ and using the definition of $\bs{V}_i$ in \eqref{eq:Vi-def}, we obtain,
  \begin{align}\label{eq:ec-temp-comp}
    \frac{d}{dt}\bs{E}_{i} = &-\frac{1}{\Delta x}\big(\underbrace{\bV_{i}^\top \mathcal{F}_{i+\frac{1}{2}}}_{(A1)}-\underbrace{\bV_{i}^\top \mathcal{F}_{i-\frac{1}{2}}}_{(A2)} - \underbrace{\Delta x\bV_{i}^\top \overline{\boldsymbol{S}}_{i}}_{(B)}\big)
%            + \langle \bV_{i}, \overline{\boldsymbol{S}}_{i}\rangle
  \end{align}
  The first term, labeled (A1), can be expanded to,
  \begin{align*}
    \mathrm{(A1)} &\stackrel{\eqref{eq:jump-average-property}}{=} \overline{\bs{V}}_{i+1/2}^\top \mathcal{F}_{i+1/2} - \frac{1}{2} \dbracket{\bs{V}}_{i+1/2}^\top \mathcal{F}_{i+1/2} \\
    &\stackrel{\eqref{eq:e-conserved-cond}\eqref{eq:e-flux}}{=}
    \mathcal{H}_{i+\frac{1}{2}} + \overline{\bs{\Psi}}_{i+\frac{1}{2}} + \frac{g}{4}\dbracket{\boldsymbol{B}}_{i+\frac{1}{2}}^\top\mcp(\bbh_{i+\frac{1}{2}})\dbracket{\bu}_{i+\frac{1}{2}}-\frac{1}{2}\dbracket{\bs{\Psi}}_{i+\frac{1}{2}} - \frac{g}{2}\dbracket{\boldsymbol{B}}_{i+\frac{1}{2}}^\top\mcp(\bbh_{i+\frac{1}{2}})\bbu_{i+\frac{1}{2}} \\
    &\stackrel{\eqref{eq:jump-average-property}}{=} \mathcal{H}_{i+\frac{1}{2}} + \bs{\Psi}_i - \frac{g}{2}\dbracket{\boldsymbol{B}}_{i+\frac{1}{2}}^\top\mcp(\bbh_{i+\frac{1}{2}}) \bs{u}_i
  \end{align*}
  In an analogous computation, the term labeled (A2) is given by,
  \begin{align*}
    \mathrm{(A2)}
 = \mathcal{H}_{i-\frac{1}{2}} + \bs{\Psi}_i + \frac{g}{2}\dbracket{\boldsymbol{B}}_{i-\frac{1}{2}}^\top\mcp(\bbh_{i-\frac{1}{2}}) \bs{u}_i
  \end{align*}
  Finally, a direct computation shows that term (B) is,
  \begin{align*}
    \mathrm{(B)} \stackrel{\eqref{eq:ec-source},\eqref{eq:Vi-def}} = -\frac{g}{2} \bs{u}_i^\top \mcp(\bbh_{i+\frac{1}{2}}) \dbracket{\boldsymbol{B}}_{i+\frac{1}{2}} - \frac{g}{2} \bs{u}_i^\top \mcp(\bbh_{i-\frac{1}{2}}) \dbracket{\boldsymbol{B}}_{i-\frac{1}{2}} 
  \end{align*}
  Using the expressions for terms (A1), (A2), and (B) derived above in \eqref{eq:ec-temp-comp} establishes that the scheme satisfies \eqref{eq:ec-scheme-def}, i.e., is an EC scheme.
\end{proof}

We now have all the ingredients necessary to prove \cref{thm:ec-scheme}.
\begin{proof}[Proof of \cref{thm:ec-scheme}]
  Lemmas \ref{lem:ec-well-balanced} and \ref{lem:ec-lte-2} verify that the scheme is well-balanced and second-order. We therefore need only show that it is EC. To do this, we must verify the condition in \cref{lem:e-conserved}.
  We accomplish this with direct computation:
  \begin{align*}
    \dbracket{\bV}_{i+\frac{1}{2}}^\top \mathcal{F}^{\text{EC}}_{i+\frac{1}{2}} & \stackrel{\eqref{eq:Vi-def}, \eqref{eq:ec-flux}}{=} \left(g\left(\dbracket{\bh}_{i+\frac{1}{2}}+\dbracket{\boldsymbol{B}}_{i+\frac{1}{2}}\right) - \frac{1}{2}\dbracket{\mcp(\bu)\bu}_{i+\frac{1}{2}}\right)^\top\mcp(\bbh_{i+\frac{1}{2}})\bbu_{i+\frac{1}{2}}\\
        &\;\;\; +\dbracket{\bu}_{i+\frac{1}{2}}^\top\left(\frac{g}{2}\left(\overline{\mcp(\bh)\bh}\right)_{i+\frac{1}{2}} + \mcp(\bbu_{i+\frac{1}{2}})\mcp(\bbh_{i+\frac{1}{2}})\bbu_{i+\frac{1}{2}}\right)\\
        &\stackrel{\eqref{eq:relation1}}{=} g\left( \dbracket{\bh}_{i+\frac{1}{2}} + \dbracket{\bs{B}}_{i+1/2}\right)^\top\mcp(\bbh_{i+\frac{1}{2}})\bbu_{i+\frac{1}{2}}+\frac{g}{2}\dbracket{\bu}_{i+\frac{1}{2}}^\top\left(\overline{\mcp(\bh)\bh}\right)_{i+\frac{1}{2}}\\
%        &+g\dbracket{\boldsymbol{B}}_{i+\frac{1}{2}}^\top\mcp(\bbh_{i+\frac{1}{2}})\bbu_{i+\frac{1}{2}}\\
        &\stackrel{\eqref{eq:relation1}}{=} \frac{g}{2}\dbracket{\mcp(\bh)\bh}_{i+\frac{1}{2}}^\top\overline{\bu}_{i+\frac{1}{2}}
        +g\dbracket{\boldsymbol{B}}_{i+\frac{1}{2}}^\top\mcp(\bbh_{i+\frac{1}{2}})\bbu_{i+\frac{1}{2}}
        +\frac{g}{2}\dbracket{\bu}_{i+\frac{1}{2}}^\top\left(\overline{\mcp(\bh)\bh}\right)_{i+\frac{1}{2}}\\
%        &+g\dbracket{\boldsymbol{B}}_{i+\frac{1}{2}}^\top\mcp(\bbh_{i+\frac{1}{2}})\bbu_{i+\frac{1}{2}}\\
        &\overset{\eqref{eq:relation2}}{=}\frac{g}{2}\dbracket{\bu^\top\mcp(\bh)\bh}_{i+\frac{1}{2}}+ g\dbracket{\boldsymbol{B}}_{i+\frac{1}{2}}^\top\mcp(\bbh_{i+\frac{1}{2}})\bbu_{i+\frac{1}{2}}\\
        &=\dbracket{\bs{\Psi}}_{i+\frac{1}{2}}+g\dbracket{\boldsymbol{B}}_{i+\frac{1}{2}}^\top\mcp(\bbh_{i+\frac{1}{2}})\bbu_{i+\frac{1}{2}},
  \end{align*}
  which verifies \eqref{eq:e-conserved-cond}, and hence \cref{lem:e-conserved} is applicable, showing that this is an EC scheme.
\end{proof}

\subsection{A first-order ES scheme}
%\par  (\textcolor{red} {Continue below, April 5 2023})\\  
The scheme determined by \eqref{eq:sgswe-ec} 
%e-conserved-source} and \eqref{eq:e-conserved-flux} can 
numerically preserves the energy of the PDE system \eqref{eq:swesg1-1d}. However, it may lead to spurious oscillations since the energy should dissipate in the presence of shocks. The issue can be resolved by introducing appropriate numerical viscosity \cite{tadmor1987numerical, tadmor2003entropy, fjordholm_mishra_tadmor_2009, fjordholm2011well, fjordholm2012arbitrarily}. Our numerical diffusion operators are a straightforward stochastic extension of the energy-stable diffusion operators proposed in \cite{fjordholm_mishra_tadmor_2009,fjordholm2011well}. 

For context of the approach, the introduction of a traditional Roe-type diffusion for a conservation law involves augmenting an EC flux as
follows:
\begin{align*}
  \mathcal{F}_{i+1/2}^{\mathrm{RD}} \coloneqq \mathcal{F}_{i+1/2}^{EC} - \frac{1}{2} \bs{Q}^{Roe}_{i+1/2} \dbracket{\bs{U}}_{i+1/2},
\end{align*}
where $\bs{Q}^{Roe}$ is a positive semi-definite matrix defined through a diagonalization of the interfacial flux Jacobian at a Roe-averaged state:
\begin{align}\label{eq:roe-diffusion}
  \boldsymbol{Q}^{\text{Roe}}_{i+\frac{1}{2}} &\coloneqq \boldsymbol{T}^{\text{Roe}} \vert\bs{\Lambda}^{\text{Roe}}\vert\left(\boldsymbol{T}^{\text{Roe}}\right)^{-1}, &
  \frac{\partial \widehat{F}}{\partial \hU}(\overline{\bU}_{i+1/2}) &=\boldsymbol{T}^{\text{Roe}}\bs{\Lambda}^{\text{Roe}}\left(\boldsymbol{T}^{\text{Roe}}\right)^{-1}.
\end{align}
Then the semi-discrete scheme \eqref{eq:fvsemi-discrete} using the numerical flux $\mathcal{F}_{i+1/2} = \mathcal{F}^{\mathrm{RD}}_{i+1/2}$ would behave like,
\begin{align*}
  \frac{d}{dt} \bs{U}_i(t) &= -\frac{1}{\Delta x} \left(
                             \mathcal{F}^{EC}_{i+1/2} -
                             \mathcal{F}^{EC}_{i-1/2} \right) +
                             \frac{1}{2\Delta x} \left( \bs{Q}^{Roe}_{i+1/2} \dbracket{\bs{U}}_{i+1/2} - \bs{Q}^{Roe}_{i-1/2} \dbracket{\bs{U}}_{i-1/2} \right) +\boldsymbol{S}_{i}\\ 
  &\approx -\frac{1}{\Delta x} \left( \hF(\hU)\big|_{x = x_{i+1/2}} -
    \hF(\hU)\big|_{x = x_{i-1/2}}  \right) + \Delta x\,\bs{Q} \hU_{xx}\big|_{x=x_i}+S(\hU) \big|_{x=x_i},
\end{align*}
where $\bs{Q}$ is a positive-definite matrix, and hence this introduces diffusion into an EC scheme. While the above approach works in terms of adding a diffusion-like term, a convenient way to ensure energy stability is to employ a numerical diffusion term that operates on the entropic variables $\bs{V}$ instead of the conservative variables $\bs{U}$:
\begin{align}\label{eq:FES}
  \mathcal{F}^{ES}_{i+\frac{1}{2}}&\coloneqq \mathcal{F}^{EC}_{i+\frac{1}{2}} - \frac{1}{2}\boldsymbol{Q}_{i+\frac{1}{2}}^{ES}\dbracket{\bV}_{i+\frac{1}{2}},  
\end{align}
where $\bs{Q}_{i+\frac{1}{2}}^{ES}$ is a positive definite matrix that will be identified in a Roe-type way from the 
%which is the image of a diffusion operator $\mathcal{Q}^{ES}$
%applied to the 
two adjacent states $\bs{U}_i$ and $\bs{U}_{i+1}$ at the cell interface $x =
x_{i+\frac{1}{2}}$. The term $\bV_i$ is as given in \eqref{eq:Vi-def}, and is a second-order approximation to the cell-average of the entropy variable $\hV$. %(\ref{entrvar}) (for smooth solutions), namely,
%\begin{equation}\label{entrvarc1}
%  \bV:= \left(-\frac{1}{2}\bu^\top\mcp(\bu) +
%    g(\bh+\boldsymbol{B})^\top,\bu^\top\right)^\top.
%\end{equation}
We are interested in the Roe-type energy-stable operator defined as,
\begin{align}\label{eq:roe-diff-operator}
  \mathcal{Q}_{i+1/2}(\bs{U}_{i}, \bs{U}_{i+1}) \coloneqq %:(\bU_{l}, \bU_{r}) \coloneqq 
  \bs{T}_{}\vert\bs{\Lambda}_{}\vert\boldsymbol{T}_{}^{\top} \geq 0,
  %\to \boldsymbol{T}_{c}\vert\bs{\Lambda}_{c}\vert\boldsymbol{T}_{c}^{\top},
\end{align}
%where $\bU_{l}$, $\bU_{r}$ are the left- and the right- states at some cell interface, 
where the matrices $\boldsymbol{T}_{}$ and $\bs{\Lambda}_{}$ are matrices from the eigendecomposition of the flux Jacobian \eqref{eq:x-jacobian1d} evaluated at a Roe-type average state:
%satisfy the eigendecomposition of the Jacobian matrix at an average state $\bU_{c}$ determined by $\bU_{l}$ and  $\bU_{r}$:
\begin{align}\label{eq:roe-diffusion-1d}
  \frac{\partial \widehat{F}}{\partial \hU}(\widetilde{\bU}_{i+1/2}) =\boldsymbol{T}_{}\bs{\Lambda}_{}\boldsymbol{T}_{}^{-1},&&\widetilde{\bU}_{i+1/2} \coloneqq \begin{pmatrix}
    \overline{\bh}_{i+1/2}\\
    \mcp(\overline{\bh}_{i+1/2}) \overline{\bu}_{i+1/2}%\bu_{c}
    \end{pmatrix}.
\end{align}
Note in particular that $\overline{\bs{q}}_{i+1/2} \neq \mcp(\overline{\bh}_{i+1/2}) \overline{\bu}_{i+1/2}$, so that $\widetilde{\bs{U}}_{i+1/2} \neq \overline{\bs{U}}_{i+1/2}$.
%Here, 
%\begin{align*}
%    &\bh_c = \frac{1}{2}(\bh_l + \bh_r), &\bu_c = \frac{1}{2}(\bu_l+\bu_r),
%\end{align*}
%and 
%\begin{align*}
%    &\bU_{l} = \begin{pmatrix}
%    \bh_{l}\\
%    \mcp(\bh_{l})\bu_{l}
%    \end{pmatrix},&
%    \bU_{r} = \begin{pmatrix}
%    \bh_{r}\\
%    \mcp(\bh_{r})\bu_{r}
%    \end{pmatrix}.
%\end{align*}
%Note that $\bU_{c} \ne \frac{1}{2}(\bU_l+\bU_r)$ since we are taking the average of the PCE of the velocity $\bu$. 
%We will refer to the energy stable scheme introduced below as ES1, as we will show that it is first-order accurate.
%In what follows we refer to our first- and second-order energy stable schemes as ES1 and ES2, respectively.
%\subsubsection{First-Order Diffusion}
The focal scheme of this section uses the numerical flux \eqref{eq:FES}, where $\bs{Q}$ is given by the Roe-type diffusion matrix introduced above,
%matrix at $x = x_{i+\frac{1}{2}}$ is simply defined to be
\begin{align}\label{eq:entropic-roe-diffusion-1d}
    \bs{Q}^{ES1}_{i+\frac{1}{2}}\coloneqq \mathcal{Q}_{i+1/2}(\bU_{i}, \bU_{i+1}) = \bs{T}_{}\vert\bs{\Lambda}_{}\vert\boldsymbol{T}_{}^{\top},
\end{align}
where we refer to this scheme as ``ES1'' because we will show it is first-order accurate.  Our main result for this scheme is as follows.
\begin{theorem}[ES1 scheme]\label{thm:es1}
  Consider the finite volume scheme \eqref{eq:fvsemi-discrete} with source term \eqref{eq:ec-source} and diffusive numerical flux \eqref{eq:FES}, selecting the diffusion matrix as,
  \begin{align}\label{eq:ES-ES1}
    \bs{Q}_{i+1/2}^{ES} = \bs{Q}_{i+1/2}^{ES1},
  \end{align}
  The resulting scheme is a first-order, well-balanced ES scheme.
%    With $\mathcal{F}^{EC}_{i+\frac{1}{2}}$ given in \eqref{eq:e-conserved-flux} and $\boldsymbol{V}$ defined via \eqref{eq:Vi-def}, consider the flux,
%    \begin{align}\label{eq:e-stable-flux}
%        \mathcal{F}^{ES1}_{i+\frac{1}{2}} = \mathcal{F}^{EC}_{i+\frac{1}{2}} - \frac{1}{2}\boldsymbol{Q}_{i+\frac{1}{2}}^{ES1}\dbracket{\boldsymbol{V}}_{i+\frac{1}{2}}.
%    \end{align}
%    Then the corresponding ES1 scheme is
%    \begin{enumerate}[(i)]
%        \item first-order accurate, 
%        \item energy stable (energy dissipative), in the sense that 
%        \begin{align}
%            \frac{d}{dt}\sum_{i}E_i(t)\Delta x \le 0,
%        \end{align}
%        and,
%        \item well-balanced.
%    \end{enumerate}
\end{theorem}
\begin{proof}
  We omit some details that are similar to the proof of \cref{thm:ec-scheme}. We have already established in \Cref{thm:ec-scheme} that $\mathcal{F}_{i+1/2}^{EC}$ is second-order accurate. That this ES1 scheme is first-order is direct from the definition of $\bs{V}_i$ in \eqref{eq:Vi-def}, resulting in the approximation 
  \begin{align*}
 \dbracket{\bs{V}}_{i+1/2} \approx \Delta x \hV_x(x_{i+1/2}).
  \end{align*}
  which implies that the diffusive augmentation in \eqref{eq:FES} commits a first-order local truncation error.

  To establish that this scheme is well-balanced, we assume the stochastic lake-at-rest initial data \eqref{eq:wbinit}, and this coupled with the definition of $\bs{V}_i$ in \eqref{eq:Vi-def} implies $\dbracket{\bs{V}}_{i+1/2} = 0$. Since the EC flux and source are well-balanced (\cref{lem:ec-well-balanced}), then this implies that this ES1 scheme is also well-balanced.

  Finally, we seek to show the ES property. We define the ES1 energy flux,
  \begin{align*}
      \mathcal{H}^{ES1}_{i+\frac{1}{2}} = \mathcal{H}_{i+\frac{1}{2}} - \frac{1}{2} \bbV_{i+\frac{1}{2}}^\top \boldsymbol{Q}_{i+\frac{1}{2}}^{ES1}\dbracket{\boldsymbol{V}}_{i+\frac{1}{2}}.
  \end{align*}
  with $\mathcal{H}_{i+1/2}$ as defined in \eqref{eq:e-flux}. As in \cref{lem:e-conserved}, we multiply \eqref{eq:fvsemi-discrete} by $\bs{V}_i^\top$; after manipulations that are similar to those in the proof of \cref{lem:e-conserved}, we have,
        \begin{align*}
          \frac{d}{dt} \bs{E}_i(t) = &-\frac{1}{\Delta x}\left(\mathcal{H}^{ES1}_{i+\frac{1}{2}} 
            - \mathcal{H}^{ES1}_{i-\frac{1}{2}}\right)\\
            &-\frac{1}{4\Delta x}\left(\dbracket{\boldsymbol{V}}_{i+\frac{1}{2}}^\top \boldsymbol{Q}_{i+\frac{1}{2}}^{ES1} \dbracket{\boldsymbol{V}}_{i+\frac{1}{2}} +\dbracket{\boldsymbol{V}}_{i-\frac{1}{2}}^\top \boldsymbol{Q}_{i-\frac{1}{2}}^{ES1} \dbracket{\boldsymbol{V}}_{i-\frac{1}{2}} \right). 
        \end{align*}
        Since $\bs{Q}_{i+1/2}^{ES1}$ is positive semi-definite, then this scheme satisfies \eqref{eq:es-condition}, and hence is an ES scheme.
\end{proof}

\subsection{ES1 diffusion vs Roe diffusion}
We provide in this section a result that motivates and justifies our particular form of the ES1 diffusion modification defined in \eqref{eq:entropic-roe-diffusion-1d} and \eqref{eq:ES-ES1}. This result states that if the bottom topography function vanishes (i.e., we are in the specialized case of a conservation law), then our chosen Roe-type ES1 diffusion in \eqref{eq:roe-diff-operator} and \eqref{eq:roe-diffusion-1d} coincides with a standard Roe-type diffusion term. Hence, in specialized scenarios our diffusive augmentations using entropic variables are equivalent to more standard Roe-type diffusion.
%a more standard Roe averaging of the conservative variables produces a diffusion operator that mimics the one defined with the Roe-type averaging considered above.
\begin{proposition}\label{prop:flatRoe}
  Define the Roe diffusion matrix as in \eqref{eq:roe-diffusion}, but using the flux Jacobian evaluated at $\widetilde{\bU}_{i+\frac{1}{2}}$,
  \begin{align}\label{eq:modified-roe-flux}
    \frac{\partial \widehat{F}}{\partial \hU}(\widetilde{\bU}_{i+\frac{1}{2}}) &=\boldsymbol{T}^{\text{Roe}}\bs{\Lambda}^{\text{Roe}}\left(\boldsymbol{T}^{\text{Roe}}\right)^{-1}.
  \end{align}
  where we have evaluated the flux jacobian at $\widetilde{\bU}_{i+1/2}$ instead of at $\overline{\bU}_{i+1/2}$. 
%  \begin{align*}
%    \boldsymbol{Q}^{\text{Roe}}_{i+\frac{1}{2}} &\coloneqq \boldsymbol{T}^{\text{Roe}} \vert\bs{\Lambda}^{\text{Roe}}\vert\left(\boldsymbol{T}^{\text{Roe}}\right)^{-1}, &
%    \frac{\partial \widehat{F}}{\partial \hU}(\overline{\bU}_{i+1/2}) &=\boldsymbol{T}^{\text{Roe}}\bs{\Lambda}^{\text{Roe}}\left(\boldsymbol{T}^{\text{Roe}}\right)^{-1}.
%  \end{align*}
  Assume $\bs{B}_i = 0$ for all $i \in [M]$. Then,\\
  \begin{align}\label{eq:flatRoe}
    \boldsymbol{Q}^{\text{Roe}}_{i+\frac{1}{2}}\dbracket{\bU}_{i+\frac{1}{2}} = \boldsymbol{Q}^{ES1}_{i+\frac{1}{2}}\dbracket{\bV}_{i+\frac{1}{2}}.
  \end{align}
% {\color{red} Akil, I commented part about ES1 and Roe fluxes since they are
%   not the same, only the diffusion correction part of the flux becomes the same}
 %and in particular $\mathcal{F}^{ES1}_{i+1/2} = \mathcal{F}^{Roe}_{i+1/2}$.
\end{proposition}
Proving this result requires some setup: 
%We seek to show \eqref{eq:flatRoe}, for which we provide a direct and self-contained presentation. 
Under the assumptions of \cref{prop:flatRoe} we consider the SGSWE \eqref{eq:swesg41d} with flat bottom, i.e., $\hB = \boldsymbol{0}$, together with entropy $E^{\text{flat}}(\hU) = \frac{1}{2}(\hq)^\top\hu+\frac{g}{2}\Vert\hh\Vert^2$ and the entropy variables, 
    \begin{align}\label{eq:entropicvarflat}
        \hV^{\text{flat}} = \partial_{\hU}E = \begin{pmatrix}
        -\frac{1}{2}\mcp(\hu)\hu+g\hh\vspace{0.5em}\\ \hu
        \end{pmatrix}.
    \end{align}
Our main tool will be some results of the proof of Theorem 3.1 in \cite{doi:10.1137/20M1360736}; in particular, while we have provided the flux Jacobian for this system in \eqref{eq:x-jacobian1d}, we will need the explicit similarity transform that accomplishes its symmetrization.
\begin{lemma}[\cite{doi:10.1137/20M1360736}, Theorem 3.1]
Assume $\mcp(\hh) > 0$. Define $G = \sqrt{g\mcp(\hh)}$ as the positive definite square root matrix of $g\mcp(\hh)$. Then,
\begin{align*}
    \frac{\partial \widehat{F}}{\partial \hU}(\hU) = \bR\boldsymbol{D}\bR^{-1},
\end{align*}
where $\boldsymbol{D}$ is the symmetric matrix, 
\begin{align}\label{eq:matrixD}
  \boldsymbol{D}(\hU) = \frac{1}{2} \left( \begin{array}{cc} 2 G + \mcp(\hu) + gG^{-1}\mcp(\hq)G^{-1} & 
    \mcp(\hu) - gG^{-1}\mcp(\hq)G^{-1} \vspace{.5em}\\ 
    \mcp(\hu) - gG^{-1}\mcp(\hq)G^{-1} & \mcp(\hu) +  g G^{-1}\mcp(\hq)G^{-1} - 2 G\end{array}\right),
\end{align}
and 
\begin{align}\label{eq:scaledeigenmatrix}
  \bR(\hU) = \frac{1}{\sqrt{2g}}\begin{pmatrix}
    I&I\\
    \mcp(\hu)+\sqrt{g\mcp(\hh)}&
    \mcp(\hu)-\sqrt{g\mcp(\hh)}
    \end{pmatrix}.
\end{align}
\end{lemma}
The second lemma reveals the relation between the cell interface jump of
$\bV^{\text{flat}}$ (the spatial approximation corresponding to the
cell-averaged entropy variable $\hV^{\text{flat}}$ in \eqref{eq:entropicvarflat}) and $\bU$.% at the same cell interface.
\begin{lemma}\label{eq:jumprelation}
  Recall the definition of $\widetilde{\bU}_{i+\frac{1}{2}}$ in \eqref{eq:roe-diffusion-1d}:
  \begin{align*}%\label{eqapp:intermediate}
    \widetilde{\bU}_{i+\frac{1}{2}} = \begin{pmatrix}
    \bbh_{i+\frac{1}{2}}\\
    \mcp(\bbh_{i+\frac{1}{2}})\bbu_{i+\frac{1}{2}}
    \end{pmatrix}
\end{align*}
which is an intermediate state defined by the arithmetic average of $\bh$
and $\bu$ across the cell interface $x = x_{i+\frac{1}{2}}$. Denote,
$\bV^{\text{flat}}$ to be the corresponding spatial approximation of the
cell-averaged entropy variable defined in \eqref{eq:entropicvarflat}. Then,
\begin{enumerate}%[(i)]
    \item The jump $\dbracket{\bU}_{i+\frac{1}{2}}$ is a rescaling of the jump $\dbracket{\bV^{\text{flat}}}_{i+\frac{1}{2}}$, i.e.,
    \begin{align}\label{eq:e-jump-to-c-jump}
        \dbracket{\bU}_{i+\frac{1}{2}} = (\bV^{\text{flat}}_{\bU})_{i+\frac{1}{2}}\dbracket{\bV^{\text{flat}}}_{i+\frac{1}{2}},
    \end{align}
    where 
    \begin{align*}
      (\bV^{\text{flat}}_{\bU})_{i+\frac{1}{2}} &\coloneqq \frac{1}{g}\begin{pmatrix}
        I&\mcp(\bbu_{i+\frac{1}{2}})\\
        \mcp(\bbu_{i+\frac{1}{2}})& \mcp^2(\bbu_{i+\frac{1}{2}})+g\mcp(\bbh_{i+\frac{1}{2}})
      \end{pmatrix}, \\
      \dbracket{\bV^{\text{flat}}}_{i+1/2} &\stackrel{\eqref{eq:entropicvarflat}}{=} \left(\begin{array}{c} -\frac{1}{2} \dbracket{\mcp(\bu) \bu}_{i+\frac{1}{2}} + g \dbracket{\bh}_{i+\frac{1}{2}}\\ \dbracket{\bu}_{i+\frac{1}{2}}\end{array}\right)
    \end{align*}
  \item Let $\bR_{i+\frac{1}{2}}$ denote the matrix that symmetrizes the flux Jacobian at the state $\widetilde{\bU}_{i+\frac{1}{2}}$,%Define the scaled matrix that symmetrizes the Jacobian
%      $\partial \widehat{F}/\partial \widehat{U}$ at
%      $\bU_{i+\frac{1}{2}}$ (see \cite{doi:10.1137/20M1360736},
%      Theorem 3.1) as,
    \begin{align*}%\label{eq:scaledeigenmatrix}
      \bR_{i+\frac{1}{2}} \coloneqq \bs{R}\left(\widetilde{\bU}_{i+\frac{1}{2}}\right) = \frac{1}{\sqrt{2g}}\begin{pmatrix}
        I&I\\
        \mcp(\bbu_{i+\frac{1}{2}})+\sqrt{g\mcp(\bbh_{i+\frac{1}{2}})}&
        \mcp(\bbu_{i+\frac{1}{2}})-\sqrt{g\mcp(\bbh_{i+\frac{1}{2}})}
        \end{pmatrix},
    \end{align*}
    cf. \eqref{eq:scaledeigenmatrix}.
%    where $\sqrt{g\mcp(\bbh_{i+\frac{1}{2}})}$ is the positive definite square root matrix of $g\mcp(\bbh_{i+\frac{1}{2}})$. 
    Then, 
    \begin{align}\label{eq:u-v-decomp}
        \bR_{i+\frac{1}{2}}\bR_{i+\frac{1}{2}}^\top = (\bV^{\text{flat}}_{\bU})_{i+\frac{1}{2}}.
    \end{align}
\end{enumerate}
\end{lemma}
\begin{proof}
    Part (2), i.e., \eqref{eq:u-v-decomp}, is a straightforward matrix algebra calculation that we omit. For part (1), we first recall that \eqref{eq:relation1} implies, 
    \begin{align}\label{eq:relation4}
        \frac{1}{2}\dbracket{\mcp(\bu)\bu}_{i+\frac{1}{2}} = \mcp(\bbu_{i+\frac{1}{2}})\dbracket{\bu}_{i+\frac{1}{2}}.
    \end{align}    
    Second, we use the linearity of $\mcp(\cdot)$, the property %Meanwhile, using the linearity of $\mcp$-matrices,
    \eqref{eq:jump-and-average} for arithmetic averages, and the commutation property \eqref{eq:commute}, to conclude,
    \begin{equation}\label{eq:relation5}
        \mcp(\bbu_{i+\frac{1}{2}})\dbracket{\bh}_{i+\frac{1}{2}}+\mcp(\bbh_{i+\frac{1}{2}})\dbracket{\bu}_{i+\frac{1}{2}} = \dbracket{\mcp(\bh)\bu}_{i+\frac{1}{2}} = \dbracket{\bq}_{i+\frac{1}{2}}.
    \end{equation}
    Therefore,
    {\footnotesize
    \begin{equation*}
    \begin{aligned}\label{eq:process1}
        (\bV^{\text{flat}}_{\bU})_{i+\frac{1}{2}} \dbracket{\bV^{\text{flat}}}_{i+\frac{1}{2}} %\\
        &= \frac{1}{g}\begin{pmatrix}
        -\frac{1}{2}\dbracket{\mcp(\bu)\bu}_{i+\frac{1}{2}} + g\dbracket{\bh}_{i+\frac{1}{2}}+\mcp(\bbu_{i+\frac{1}{2}})\dbracket{\bu}_{i+\frac{1}{2}}\\
        -\frac{1}{2}\mcp(\bbu_{i+\frac{1}{2}})\dbracket{\mcp(\bu)\bu}_{i+\frac{1}{2}} + g\mcp(\bbu_{i+\frac{1}{2}})\dbracket{\bh}_{i+\frac{1}{2}}+\left(\mcp^2(\bbu_{i+\frac{1}{2}})+g\mcp(\bbh_{i+\frac{1}{2}})\right)\dbracket{\bu}_{i+\frac{1}{2}}
        \end{pmatrix}\\
        \overset{\eqref{eq:relation4}\eqref{eq:relation5}}{=\joinrel=\joinrel=}&\begin{pmatrix}
        \dbracket{\bh}_{i+\frac{1}{2}}\\
        \dbracket{\bq}_{i+\frac{1}{2}}\\
        \end{pmatrix} = \dbracket{\bU}_{i+\frac{1}{2}}.
    \end{aligned}        
    \end{equation*}
    }
\end{proof}
Now we are in position to show \eqref{eq:flatRoe} in \cref{prop:flatRoe}.
\begin{proof}[Proof of \cref{prop:flatRoe}]
 Let $\bs{D}_{i+\frac{1}{2}}$ be the \textit{symmetric} matrix defined in \eqref{eq:matrixD} evaluated at $\bU_{i+\frac{1}{2}}$, and 
$\bs{D}_{i+\frac{1}{2}} = \bs{L}_{i+\frac{1}{2}}\boldsymbol{\Lambda}_{i+\frac{1}{2}}\boldsymbol{L}_{i+\frac{1}{2}}^{\top}$ be its eigenvalue decomposition. Then,
\begin{subequations}\label{eq:T-def}
\begin{align}
  \frac{\partial \widehat{F}}{\partial \hU}(\widetilde{\bU}_{i+\frac{1}{2}}) &=  \bs{R}_{i+\frac{1}{2}}\left(\bs{L}_{i+\frac{1}{2}}\boldsymbol{\Lambda}_{i+\frac{1}{2}}\boldsymbol{L}_{i+\frac{1}{2}}^{\top}\right)\bs{R}^{-1}_{i+\frac{1}{2}} \\
  &\eqqcolon \bs{T}_{i+\frac{1}{2}}\bs{\Lambda}_{i+\frac{1}{2}} \bs{T}_{i+\frac{1}{2}}^{-1},
\end{align}
\end{subequations}
  is an eigendecomposition of the Jacobian matrix $\frac{\partial \widehat{F}}{\partial \hU}(\widetilde{\bU}_{i+\frac{1}{2}})$, where we have used the fact that $\bs{L}_{i+\frac{1}{2}}^{-1} = \bs{L}_{i+\frac{1}{2}}^{\top}$ due to the symmetry of $\bs{D}_{i+\frac{1}{2}}$.  The Roe-diffusion operator evaluated at the location $\widetilde{\bU}_{i+\frac{1}{2}}$ as indicated in \eqref{eq:modified-roe-flux} is then given by, 
  \begin{align*}
    \bs{Q}^{Roe}_{i+\frac{1}{2}} \stackrel{\eqref{eq:roe-diffusion}}{=} 
    \bs{T}_{i+\frac{1}{2}}\vert\boldsymbol{\Lambda}_{i+\frac{1}{2}}\vert\bs{T}_{i+\frac{1}{2}}^{-1}
    %\underbrace{(\bs{R}_{i+\frac{1}{2}}\bs{L}_{i+\frac{1}{2}})}_{\coloneqq\bs{T}_{i+\frac{1}{2}}}\vert\boldsymbol{\Lambda}_{i+\frac{1}{2}}\vert(\bs{R}_{i+\frac{1}{2}}\bs{L}_{i+\frac{1}{2}})^{-1},
  \end{align*}
Therefore,
\begin{equation}\label{eq:roe-entropy-jump}
\begin{aligned}
    \boldsymbol{Q}^{\text{Roe}}_{i+\frac{1}{2}}\dbracket{\bU}_{i+\frac{1}{2}} &= 
    \bs{T}_{i+\frac{1}{2}}\vert\boldsymbol{\Lambda}_{i+\frac{1}{2}}\vert\bs{T}_{i+\frac{1}{2}}^{-1}\dbracket{\bU}_{i+\frac{1}{2}},\\
    %\left(\bR_{i+\frac{1}{2}}\boldsymbol{L}_{i+\frac{1}{2}}\right)\vert\boldsymbol{\Lambda}_{i+\frac{1}{2}}\vert\left(\bR_{i+\frac{1}{2}}\boldsymbol{L}_{i+\frac{1}{2}}\right)^{-1}\dbracket{\bU}_{i+\frac{1}{2}},\\
    &\stackrel{\eqref{eq:T-def},\eqref{eq:e-jump-to-c-jump}}{=}
    \left(\bR_{i+\frac{1}{2}}\boldsymbol{L}_{i+\frac{1}{2}}\right)\vert\boldsymbol{\Lambda}_{i+\frac{1}{2}}\vert\left(\bR_{i+\frac{1}{2}}\boldsymbol{L}_{i+\frac{1}{2}}\right)^{-1}(\bV^{\text{flat}}_{\bU})_{i+\frac{1}{2}}\dbracket{\bV^{\text{flat}}}_{i+\frac{1}{2}},\\
    &\overset{\eqref{eq:u-v-decomp}}{=}\left(\bR_{i+\frac{1}{2}}\boldsymbol{L}_{i+\frac{1}{2}}\right)\vert\boldsymbol{\Lambda}_{i+\frac{1}{2}}\vert\left(\bR_{i+\frac{1}{2}}\boldsymbol{L}_{i+\frac{1}{2}}\right)^{-1}\bR_{i+\frac{1}{2}}\bR_{i+\frac{1}{2}}^\top\dbracket{\bV^{\text{flat}}}_{i+\frac{1}{2}},\\
    &=\bR_{i+\frac{1}{2}}\left(\boldsymbol{L}_{i+\frac{1}{2}}\vert\boldsymbol{\Lambda}_{i+\frac{1}{2}}\vert\boldsymbol{L}^{\top}_{i+\frac{1}{2}}\right)\bR_{i+\frac{1}{2}}^\top\dbracket{\bV^{\text{flat}}}_{i+\frac{1}{2}},\\
    &= \bs{T}_{i+\frac{1}{2}} \vert\boldsymbol{\Lambda}_{i+\frac{1}{2}}\vert \bs{T}_{i+\frac{1}{2}}^\top \dbracket{\bV^{\text{flat}}}_{i+\frac{1}{2}} \\
    &\stackrel{\eqref{eq:entropic-roe-diffusion-1d}}{=} \boldsymbol{Q}^{ES1}_{i+\frac{1}{2}}\dbracket{\bV^{\text{flat}}}_{i+\frac{1}{2}}.
\end{aligned}
\end{equation}
\end{proof}

%\begin{remark}
%In the case of flat bottom, i.e., $\hB\equiv\bs{0}$, one can verify that 
%\begin{align}\label{eq:flatRoe}
%\boldsymbol{Q}^{\text{Roe}}_{i+\frac{1}{2}}\dbracket{\bU}_{i+\frac{1}{2}} = \boldsymbol{Q}^{ES1}_{i+\frac{1}{2}}\dbracket{\bV^{\text{flat}}}_{i+\frac{1}{2}},  
%\end{align}
%where $\boldsymbol{Q}^{\text{Roe}}_{i+\frac{1}{2}}$ is the usual Roe diffusion operator evaluated at the same average state $\bU_{i+\frac{1}{2}}$ defined by \eqref{eq:roe-diffusion-1d} with $\bU_{i}$ and $\bU_{i+1}$ as the left and right states, respectively,
%\begin{align}\label{eq:roe-remark}
%    \boldsymbol{Q}^{\text{Roe}}_{i+\frac{1}{2}} = \boldsymbol{T}_{i+\frac{1}{2}}\vert\bs{\Lambda}_{i+\frac{1}{2}}\vert\boldsymbol{T}_{i+\frac{1}{2}}^{-1},
%\end{align}
%and $\bV^{\text{flat}}$ is the first-order approximation of the
%entropy variable \eqref{entrvar} with the flat bottom $\bs{B}\equiv \bs{0}$. A proof is provided in \Cref{sect:flatRoe}.
%\end{remark}

\subsection{A second-order ES scheme}
To develop a second-order accurate energy-stable scheme, we use jump operators with $O(\Delta x^2)$ accuracy. A natural choice is to use the jumps obtained by non-oscillatory second-order reconstructions of the entropy variable. However, attaining a provable energy-stable scheme requires the more subtle reconstruction procedure in \cite{fjordholm2012arbitrarily} that we follow. The new idea for second-order diffusions is to use reconstructions in order to compute jumps. To that end, we let $\bs{V}^+_i$ and $\bs{V}^-_{i+1}$ be second-order reconstructions from the right and left, respectively, of the entropy variable $\bs{V}(x)$ at location $x = x_{i+1/2}$. We will describe later in this section how these reconstructions are computed. 

Assuming we have these reconstructions in hand, we can compute second-order accurate jumps of the entropy variables:
\begin{align}\label{eq:Vtilde-pw}
    \dabracket{\bV}_{i+\frac{1}{2}} = \bV_{i+1}^- - \bV_{i}^+,
\end{align}
The overall scheme is similar as the previous section, but uses a second-order diffusive augmentation of a conservative flux,
\begin{align}\label{eq:es2-flux}
  \mathcal{F}^{ES2}_{i+\frac{1}{2}} \coloneqq \mathcal{F}^{EC}_{i+\frac{1}{2}} - \frac{1}{2}\bs{Q}^{ES2}_{i+\frac{1}{2}}\dabracket{\boldsymbol{V}}_{i+\frac{1}{2}}.
\end{align}
We choose the matrix $\bs{Q}^{ES2}$ as for the ES1 scheme, 
\begin{align}\label{eq:Q-es2}
  \bs{Q}^{ES2}_{i+1/2} = \bs{Q}^{ES1}_{i+1/2} = \mathcal{Q}_{i+1/2}(\bs{U}_i, \bs{U}_{i+1}) = \bs{T}_{i+1/2} \vert\bs{\Lambda}_{i+1/2} \vert\bs{T}^\top_{i+1/2},
  %\mathcal{Q}_{i+1/2}\left(\bs{U}^+_i, \bs{U}^-_{i+1}\right),
\end{align}
where we recall that the eigendecomposition matrices $\bs{T}$, $\bs{\Lambda}$ are computed from the Roe-type average of the flux Jacobian, cf. \eqref{eq:roe-diff-operator}, \eqref{eq:roe-diffusion-1d}.  One could alternatively select $\bs{Q}^{ES2}$ by using second-order reconstructions of $\bs{U}$ as input to $\mathcal{Q}$, e.g., 
\begin{align*}
  \bs{Q}^{ES2}_{i+1/2} = \mathcal{Q}_{i+1/2}(\bs{U}_i^-, \bs{U}^+_i),
\end{align*}
for some second-order reconstructions $\bs{U}_i^{\pm}$. 

What remain is to describe how $\bs{V}_i^{\pm}$ are computed in a way that ensures the energy stable property. The main idea is to design $\bV_i^{\pm}$ through a second-order reconstruction of \textit{scaled} (transformed) versions of the entropy variables:
\begin{align}\label{eq:wi-def}
  \bs{w}_i^{\pm} \coloneqq \bs{T}^\top_{i\pm 1/2} \bs{V}_i,
\end{align}
%\an{I also don't know how $\bs{T}_{i+1/2}$ is computed; is it through the Roe-type procedure \eqref{eq:roe-diffusion-1d}, or is it done with the standard Roe average as in \eqref{eq:roe-diffusion}? I don't think it too much, but it's not clear to me. The source of how this is computed appears to be the discussion around equation (3.4) in \cite{fjordholm2012arbitrarily}, but that is not very specific, either.}
where the matrices $\bs{T}_{i\pm1/2}$ are as in \eqref{eq:Q-es2}. 
Once these have been computed, we perform a second-order total variation-diminishing (TVD) reconstruction on the $\bs{w}$ variable at the interfaces:
\begin{align}\label{eq:wtilde-def}
  \widetilde{\bs{w}}_i^{\pm} \coloneqq \bs{w}_i^{\pm} \pm \frac{1}{2} \phi\left( \bs{\theta}_i^{\pm}\right) \circ \dabracket{\bs{w}}_{i\pm 1/2},
\end{align}
where $\circ$ is the Hadamard (elementwise) product on vectors, and $\bs{\theta}_i^\pm$ are difference quotients,
\begin{align*}
  \bs{\theta}_i^\pm \coloneqq \dabracket{\bs{w}}_{i\mp 1/2} \oslash \dabracket{\bs{w}}_{i\pm 1/2},
\end{align*}
where $\oslash$ is the Hadamard (elementwise) division between vectors. We select the function $\phi$ to be the minmod limiter, 
\begin{align}\label{eq:phi-def}
\phi(\theta) = \left\{
    \begin{aligned}
        &0, &&\textit{if }\theta < 0,\\
        &\theta. &&\textit{if }0\le \theta \le 1,\\
        &1, &&\textit{otherwise}.
    \end{aligned}\right.
\end{align}
which operates elementwise on vector inputs. Note that other slope limiter functions $\phi$ may be selected, but minmod is the only valid limiter in this context that also satisfies the TVD property \cite[Section 3.4]{fjordholm2012arbitrarily}. Finally, the desired reconstructions for $\bs{V}_i^{\pm}$ are defined by inverting the $\bs{w}$-to-$\bs{V}$ map,
\begin{align}\label{eq:wtilde-to-Vi}
  \bs{T}^\top_{i\pm1/2} \bs{V}_i^{\pm} \coloneqq \widetilde{\bs{w}}_i^\pm
\end{align}
The full scheme has now been described, and satisfies the following properties.
\begin{theorem}[ES2 scheme]\label{thm:es2}
  The FV scheme \eqref{eq:fvsemi-discrete} choosing the flux $\mathcal{F}_{i+1/2} = \mathcal{F}^{ES2}_{i+1/2}$ defined in \eqref{eq:es2-flux} is a second-order, well-balanced, ES scheme.
\end{theorem}
We focus the remaining discussion in this section on sketching the proof of the above result. The second-order property results from the fact that the jumps are computed using second-order accurate reconstructions; the well-balanced property can be proven in exactly the same way as is done for the ES1 scheme in the proof of \cref{thm:es1}. To show the ES property, we exercise one of the major results in \cite{fjordholm2012arbitrarily} that we reproduce below.
\begin{lemma}[\cite{fjordholm2012arbitrarily}, Lemma 3.2]
    For each $i$, if there exists a positive diagonal matrix $\bs{\Pi}_{i+1/2}\ge 0$ such that the second-order jump satisfies,
    \begin{align}\label{eq:cond-e-stable}
        \dabracket{\bV}_{i+\frac{1}{2}} = (\boldsymbol{T}_{i+\frac{1}{2}}^\top)^{-1}\bs{\Pi}_{i+\frac{1}{2}}\boldsymbol{T}_{i+\frac{1}{2}}^{\top}\dbracket{\bV}_{i+\frac{1}{2}},
    \end{align}
    then the scheme \eqref{eq:fvsemi-discrete} with flux term $\mathcal{F}_{i+1/2} = \mathcal{F}_{i+1/2}^{ES2}$ is an ES scheme.
\end{lemma}
Hence, showing the ES property for our scheme only requires us to establish \eqref{eq:cond-e-stable}. To accomplish this, note that the definition \eqref{eq:wtilde-def} implies,
\begin{align}\label{eq:w-wtilde-jump}
    \dabracket{\widetilde{\bw}}^\ell_{i+\frac{1}{2}} = \left(1-\frac{1}{2}\phi((\bs{\theta}_{i+1}^{-})^\ell) - \frac{1}{2}\phi((\bs{\theta}_i^{+})^\ell)\right)\dabracket{\bw}^\ell_{i+\frac{1}{2}},
\end{align}
I.e., we have,
\begin{align}\label{eq:w-wtilde-jump-Pi}
  \dabracket{\widetilde{\bw}}_{i+\frac{1}{2}} &= \bs{\Pi}_{i+\frac{1}{2}}\dabracket{\bw}_{i+\frac{1}{2}}, & \left(\bs{\Pi}_{i+1/2}\right)_{\ell,\ell} &\coloneqq \left(1-\frac{1}{2}\phi((\bs{\theta}_{i+1}^{-})_\ell) - \frac{1}{2}\phi((\bs{\theta}_i^{+})_\ell)\right)
\end{align}
and in particular $\bs{\Pi}_{i+1/2}$ is a diagonal matrix and positive semi-definite since $0 \leq \phi(\theta) \leq 1$. Since the jump operators $\dabracket{\cdot}$ and $\dbracket{\cdot}$ are linear in their arguments, then combining \eqref{eq:w-wtilde-jump} with the relations \eqref{eq:wi-def} and \eqref{eq:wtilde-to-Vi} that connect $\bs{w}_i$ and $\widetilde{\bs{w}}_i$ to $\bs{V}_i$ and $\bs{V}^{\pm}_i$ yields the relation \eqref{eq:cond-e-stable} with a positive-definite diagonal matrix $\bs{\Pi}_{i+1/2}$. Hence, this is an ES scheme, and completes the proof of \cref{thm:es2}.

Finally, we remark that the implementation of the diffusion term in the ES2 flux \eqref{eq:es2-flux} does not require explicit construction of $\bs{V}^{\pm}_i$. I.e., we have,
\begin{align*}
  \frac{1}{2}\bs{Q}^{ES2}_{i+\frac{1}{2}}\dabracket{\boldsymbol{V}}_{i+\frac{1}{2}} &\stackrel{\eqref{eq:roe-diff-operator},\eqref{eq:cond-e-stable}}{=} \frac{1}{2}\boldsymbol{T}_{i+\frac{1}{2}}\vert\bs{\Lambda}_{i+\frac{1}{2}}\vert\boldsymbol{T}_{i+\frac{1}{2}}^{\top}(\boldsymbol{T}_{i+\frac{1}{2}}^\top)^{-1}\bs{\Pi}_{i+\frac{1}{2}}\boldsymbol{T}_{i+\frac{1}{2}}^{\top}\dbracket{\bV}_{i+\frac{1}{2}}\\
            &= \frac{1}{2}\boldsymbol{T}_{i+\frac{1}{2}}\vert\bs{\Lambda}_{i+\frac{1}{2}}\vert\bs{\Pi}_{i+\frac{1}{2}}\boldsymbol{T}_{i+\frac{1}{2}}^{\top}\dbracket{\bV}_{i+\frac{1}{2}}\\             
             & \stackrel{\eqref{eq:wi-def}}{=}\frac{1}{2}\bs{T}_{i+\frac{1}{2}}\vert\bs{\Lambda}_{i+\frac{1}{2}}\vert\bs{\Pi}_{i+\frac{1}{2}}\dabracket{\bw}_{i+\frac{1}{2}},\\
             &\overset{\eqref{eq:w-wtilde-jump-Pi}}{=}\frac{1}{2}\bs{T}_{i+\frac{1}{2}}\vert\bs{\Lambda}_{i+\frac{1}{2}}\vert\dabracket{\widetilde{\bw}}_{i+\frac{1}{2}},
\end{align*}
and hence one need only compute $\widetilde{\bs{w}}^{\pm}_i$ in order to directly evaluate the diffusion part of the ES2 flux.

\subsection{Algorithmic details}
Our overall scheme is the semi-discrete form \eqref{eq:fvsemi-discrete}, which we pair with a numerical time-stepping scheme. We provide pseudocode in this section that describes a fully discrete SGSWE time-stepping algorithm. This full pseudocode introduces some additional details for the scheme that were devised in \cite{doi:10.1137/20M1360736}, many of which are based on standard procedures used in schemes for deterministic SWE models \cite{kurganov_finite-volume_2018}. We very briefly describe these additional details in the coming sections; more comprehensive discussion can be found in \cite{doi:10.1137/20M1360736}. The full algorithmic pseudocode is given in \cref{alg:swesg}.

\subsubsection{Velocity desingularization}
Computing $\bs{u}_i$ requires inversion of the matrix $\mathcal{P}(\bs{h}_i)$, which is assumed (and enforced in the scheme) to be symmetric and positive-definite. However, this matrix may be ill-conditioned. To ameliorate numerical artifacts associated with this ill-conditioned operation, we employ a \textit{desingularization} procedure, introduced for the deterministic SWE in \cite{kurganov_secondorder_2007-2}. We describe here the stochastic variant of the desingularization procedure, proposed in \cite{doi:10.1137/20M1360736}. If $\mcp(\bs{h}_i)$ has the eigenvalue decomposition,
\begin{align*}
  \mcp(\bs{h}_i) &= \bs{Q} \bs{\Pi} \bs{Q}^\top, & \bs{\Pi} &= \mathrm{diag}(\pi_1, \ldots, \pi_K),
\end{align*}
where $\pi_k > 0$ are the eigenvalues of $\mcp(\bs{h}_i)$, then the desingularization process approximates $\mcp(\bs{h}_i)^{-1} \bs{q}_i$ by regularizing the matrix inverse procedure:
\begin{align}\label{eq:ui-desingularized}
  \bs{u}_i &= \bs{Q} \widetilde{\bs{\Pi}}^{-1} \bs{Q}^T \bs{q}_i, & 
  \bs{\Pi} &= \mathrm{diag}(\widetilde{\pi}_1, \ldots, \widetilde{\pi}_K), & 
  \widetilde{\pi}_k &= \frac{\sqrt{\pi_k^4 + \max\{ \pi_k^4, \epsilon^4\}}}{\sqrt{2} \pi_k},
\end{align}
where $\epsilon > 0$ is a small constant; we choose it to be $\epsilon = \Delta x$. Note that if $\pi_k \geq \epsilon^{1/4}$, then $\widetilde{\pi}_k = \pi_k$, and hence regularization is performed only in the presence of small eigenvalues. Compared to \eqref{eq:ui-def}. This procedure to compute $\bs{u}_i$ is a stabilized way to compute velocities.

For scheme consistency, if the desingularization above is activated, then we recompute the discharge variable:
\begin{align*}
  \bs{q}_i \gets \mcp(\bs{h}_i) \bs{u}_i.
\end{align*}

\subsubsection{Hyperbolicity preservation}
The SGSWE PDE \eqref{eq:swesg41d} is hyperbolic and has an entropy pair if $\mcp(\hh) > 0$, i.e., \Cref{thm:hyperbolicity,thm:eef-sgswe-1d}, respectively. To ensure this holds at the discrete level, we require the condition $\mathcal{P}(\bs{h}_i) > 0$ for every cell $i$. To enforce this, we employ \cite[Theorem 3.4, Corollary 3.5]{doi:10.1137/20M1360736}, which state that a sufficient condition for $\mcp(\bs{h}_i) > 0$ is that for every $m =1, \ldots, M$,
\begin{align}\label{eq:h-pos}
  \hh_i(\xi_m) &> 0, & 
  \hh_i(\xi) &\coloneqq \sum_{k=1}^K \bs{h}_{i,k} \phi_k(\xi_m), & 
  \bs{h}_i &= (\bs{h}_{i,1}, \ldots, \bs{h}_{i,K})^\top,
\end{align}
where $\{\xi_m\}_{m=1}^M$ is a nodal set in $\R^d$ for a positive-weight quadrature rule having sufficient accuracy relative to the $\xi$-polynomial space $P$ defined in \eqref{eq:P-def}. The functions $\phi_k$ are the basis of $P$ in \eqref{eq:P-def} for which $\bs{h}_i$ are coordinates. The function $\hh_i(\xi)$ is the SGSWE approximation to the $\mathcal{I}_i$-cell average of $\hh(x,t,\xi)$ at the current time. Hence, the computational vehicle we use to enforce hyperbolicity of the underlying PDE in our scheme is to enforce the above positivity-type condition on the $\bs{h}_i$ variable.

\subsubsection{Positivity-preservation}
We enforce the positivity condition \eqref{eq:h-pos} by restricting the timestep size. We assume that the current time value of $\bs{h}_i$ satisfies \eqref{eq:h-pos}. If Forward Euler with a stepsize $\Delta t$ is used to discretize \eqref{eq:fvsemi-discrete}, then \eqref{eq:h-pos} is true at the next time step if,
\begin{align}\label{eq:lambda-def}
  \Delta t < \lambda \coloneqq \min_{i} \min_{m=1,\ldots, M} \left| \frac{\Delta x\; \hh_i(\xi_m)}{\widehat{\mathcal{F}}^h_{i+1/2}(\xi_m) - \widehat{\mathcal{F}}^h_{i-1/2}(\xi_m)} \right|,
\end{align}
where $\widehat{\mathcal{F}}^h_{i+1/2}(\cdot)$ is the SG approximation of the $\hh$-variable flux:
\begin{align*}
  \widehat{\mathcal{F}}^h_{i+1/2}(\xi) &\coloneqq \sum_{k=1}^K \mathcal{F}^h_{i+1/2,k} \phi_k(\xi), & 
  \mathcal{F}_{i+1/2} &= \left( \left(\mathcal{F}^h_{i+1/2}\right)^\top, \; \left(\mathcal{F}^q_{i+1/2}\right)^\top \right)^\top \in \R^{2 K}.
\end{align*}
Hence, we enforce positivity preservation by ensuring a small enough timestep so that the positivity condition \eqref{eq:h-pos} is respected globally over all spatial cells. We must also restrict $\Delta t$ to satisfy the wave speed CFL condition; see \cite[Equation (4.16)]{doi:10.1137/20M1360736}.

\subsubsection{Adaptive time-stepping}\label{sssec:adaptive-ts}
The time step restriction \eqref{eq:lambda-def} works for Forward Euler time-stepping. To extend this to a higher-order temporal scheme, we employ a third-order strong stability-preserving scheme, which is a convex combination of Forward Euler steps \cite{gottlieb_strong_2001}. However, the intermediate stages of a(ny) time-stepping scheme need not obey the positivity-preserving property, even if $\Delta t$ is chosen to obey the condition \eqref{eq:lambda-def} determined at the initial step. 

To address this issue, we employ the \textit{adaptive} time-stepping strategy proposed in \cite[Remark 3.6]{chertock_well-balanced_2015}. We refer the reader to that reference for details, and present here only a high-level description of the procedure: $\lambda$ is initialized as the initial stage value of $\lambda$, as shown in \eqref{eq:lambda-def}. At intermediate stages, new intermediate values of $\lambda$ are computed. If an intermediate-stage value of $\lambda$ is smaller than the current value of $\lambda$, then we restart the entire time-step using the new, smaller-$\lambda$ restriction on $\Delta t$.

\begin{algorithm}
  \begin{algorithmic}
    \State Input scheme type: \texttt{scheme} = \texttt{EC}, \texttt{ES1}, or \texttt{ES2}
    \State Input: Bottom topography $B$, initial data $U(t=0)$, polynomial index set $\Lambda$
    \State Input: Terminal time $T$
    \State Initialize: $\bs{U}_i$, $t=0$
    \Repeat
      \State Compute $\bs{B}_i$ from $B$ for all $i$
      \State Compute $\bs{u}_i$ in \eqref{eq:ui-desingularized} for all $i$
      \State Compute $\mathcal{F}^{EC}_{i+1/2}$ for all $i$, given by \eqref{eq:sgswe-ec}
      \If{\texttt{scheme} is \texttt{EC}} for all $i$:
        \State Set $\mathcal{F}_{i+1/2} \gets \mathcal{F}^{EC}_{i+1/2}$.
      \Else{} for all $i$:
        \State Compute entropy variable $\bs{V}_i$ using \eqref{eq:Vi-def}.
        \State Compute $\bs{T}_{i+1/2}$, $\bs{\Lambda}_{i+1/2}$ through \eqref{eq:roe-diffusion-1d}.
        \If{\texttt{scheme} is \texttt{ES1}}:
          \State Compute $\bs{Q}_{i+1/2}^{ES1}$ using \eqref{eq:roe-diffusion-1d},\eqref{eq:entropic-roe-diffusion-1d} with $\bs{T}_{i+1/2}$, $\bs{\Lambda}_{i+1/2}$.
          \State Compute $\mathcal{F}_{i+1/2} \gets \mathcal{F}^{ES}_{i+1/2}$ in \eqref{eq:FES} using $\mathcal{F}^{EC}_{i+1/2}$, $\bs{V}_i$, and $\bs{Q}_{i+1/2}^{ES} \gets \bs{Q}_{i+1/2}^{ES1}$.
        \ElsIf{\texttt{scheme} is \texttt{ES2}}:
          \State Construct $\bs{Q}^{ES2}_{i+1/2}$ as in \eqref{eq:Q-es2} with $\bs{T}_{i+1/2}$, $\bs{\Lambda}_{i+1/2}$.
          \State Construct $\bs{V}_i^{\pm}$ through \eqref{eq:wi-def}, \eqref{eq:wtilde-def}, and \eqref{eq:wtilde-to-Vi}.
          \State Compute $\mathcal{F}_{i+1/2} \gets \mathcal{F}^{ES2}_{i+1/2}$ in \eqref{eq:es2-flux} and \eqref{eq:Vtilde-pw} using $\bs{Q}^{ES2}_{i+1/2}$, $\bs{V}_i^{\pm}$, and $\mathcal{F}^{EC}_{i+1/2}$.
        \EndIf
      \EndIf
      \State Initialize $\lambda$ and $\Delta t$ as shown in \eqref{eq:lambda-def}.
      \State Adaptively determine $\Delta t$ using the procedure discussed in \cref{sssec:adaptive-ts}.
      \State Use a third-order SSP method to take a time step of size $\Delta t$, updating $\bs{h}_i$ and $\bs{q}_i$.
      \State Set $t \gets t + \Delta t$.
    \Until{$t \geq T$}
  \end{algorithmic}
  \caption{The fully discrete SGSWE schemes proposed in this paper; we ignore specifying the handling of boundary conditions.}\label{alg:swesg}
\end{algorithm}

\section{Numerical Experiments}\label{sec:results}
\par Below we present several numerical examples to illustrate properties
of the developed schemes. We refer to the second order energy-conservative scheme,
the first order energy-stable scheme, and the second order energy-stable scheme as
the EC, ES1, and ES2 schemes, respectively. We introduce the relative change in energy quantity,
\begin{align}
  \textrm{relative energy} = \frac{E(t) - E(0)}{E(t)},
\end{align}
where $E(t)$ is computed as $\sum_i \Delta x \bs{E}_i(t)$. This provides a way to visualize the relative change in the discrete energy for different
numerical schemes, namely for the EC, ES1 and ES2. In all tests
below we consider a single (scalar) random variable $\xi$ that is a uniformly distributed random variable on $[-1,1]$; hence our choice for the functions $\phi_k$ are orthonormal Legendre polynomials on $[-1,1]$. We use $K=9$ for the dimension of the PC space $P$. Instead of visualizing the conservative variable $h$ corresponding to water height, we will plot the \textit{water surface} $w$, defined as $w = h + B$, with the bottom topography $B$ superimposed on the same graph; plots of $(w,B)$ are more physically interpretable than directly plotting the water height $h$.
\subsection{Flat-Bottom Dam Break}\label{ssec:results-flat-bot}
In the first experiment, we consider a stochastic water surface,
\begin{equation*}%\label{eq:IV1}
  h(x,0,\xi) + B(x,0,\xi) = w(x,0,\xi) = \left\{\begin{aligned}&2.0+0.1\xi&&x<0\\ &1.5+0.1\xi&&x>0\end{aligned}\right.,\quad q(x,0,\xi) = 0,
\end{equation*} 
with a flat bottom $B(x,t,\xi)\equiv 0$. This is a stochastic
modification of the deterministic ``dam break test'' problem from
\cite{fjordholm2011well}. In \cref{fig:ex1a-flat-bot-wq}, we
use a uniform grid size $\Delta x=400$ over the physical domain $x \in
[-1,1]$,  and compute up to the time $t =
0.4$. We test the example using different numerical
methods (EC, ES1, ES2), Section~\ref{sec:schemes}.
\par From \cref{fig:ex1a-flat-bot-wq}, similar to the
results presented in \cite[Figs. 1 and 4]{fjordholm2011well}, we
observe that the water surface with uncertainties develops a leftward-going
rarefaction wave and a rightward-going shock. Similar to
\cite{fjordholm2011well}, EC computes such solutions accurately,
but at the expense of large post-shock oscillations as observed on
\cref{fig:ex1a-flat-bot-wq} (right plot). These oscillations are
expected since the EC scheme preserves energy, and hence energy is
not dissipated across the shock as it should. We also demonstrate on
\cref{fig:ex1a-flat-bot-mean} (middle and right plots)  the numerical
energy conservation for the EC
scheme. We note that the energy conservation errors due
to time discretization are reduced significantly by decreasing the time
step/CFL constant (right figure), similar to the results reported in
fig.~1 in \cite{fjordholm2011well}. The presented results in 
\cref{fig:ex1a-flat-bot-wq} and \cref{fig:ex1a-flat-bot-mean} (left figure) also illustrate
that ES2 produces less smearing than the ES1 at both the rarefaction and the shock
waves. The schemes ES1 and ES2 are both designed to dissipate
energy which is also confirmed by the numerical experiments as presented
in \cref{fig:ex1a-flat-bot-mean} (middle plot), with the energy
dissipation in the ES2 being lower than in the ES1
scheme. In addition, the numerical results seem to indicate that the ES2 scheme is better able to capture large variance spikes compared to the ES1 scheme. Finally, the employment of the the numerical diffusion operators in ES1 and ES2
schemes, removes oscillations present in the numerical solution using the EC scheme. The observed results are also in agreement with the results
of the deterministic model reported in \cite{fjordholm2011well}.
\begin{figure}[htbp]
    \centering
    \includegraphics[width = .32\textwidth, trim={0 6cm 0 6cm}, clip]{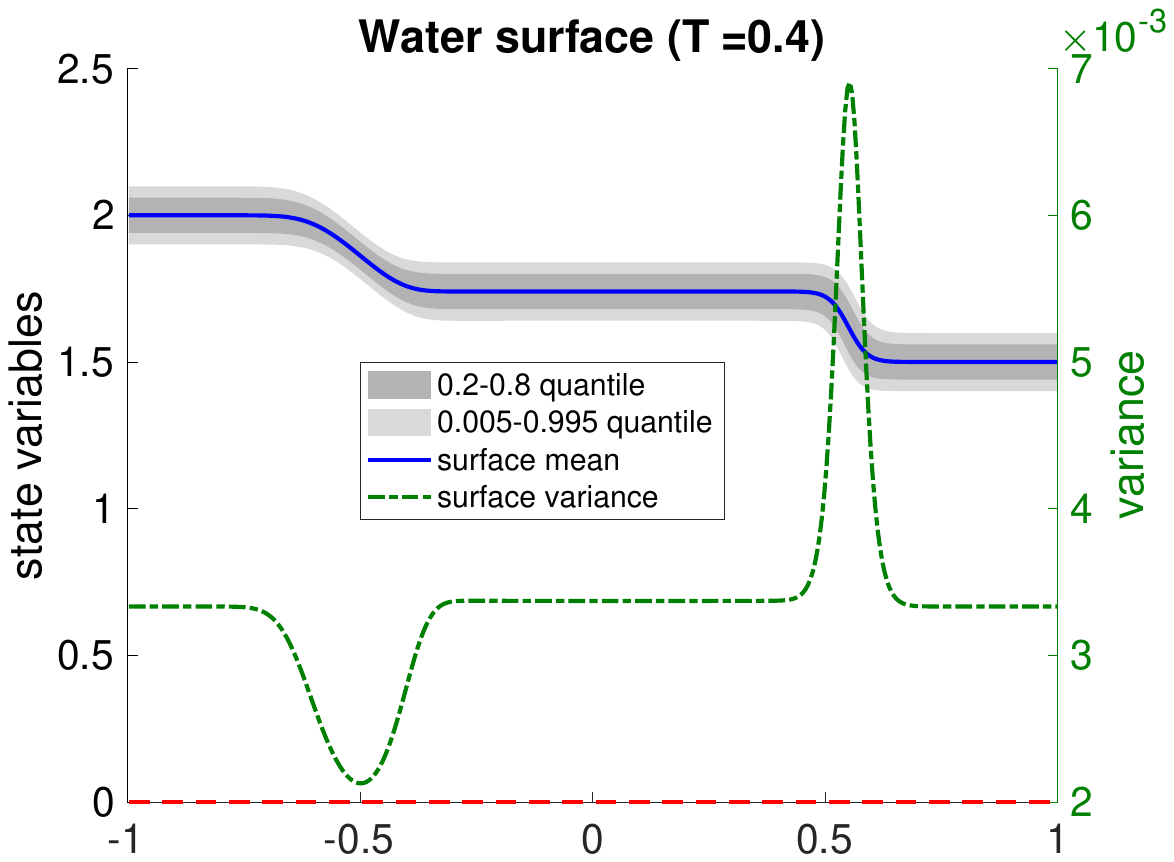}
    \includegraphics[width = .32\textwidth, trim={0 6cm 0 6cm}, clip]{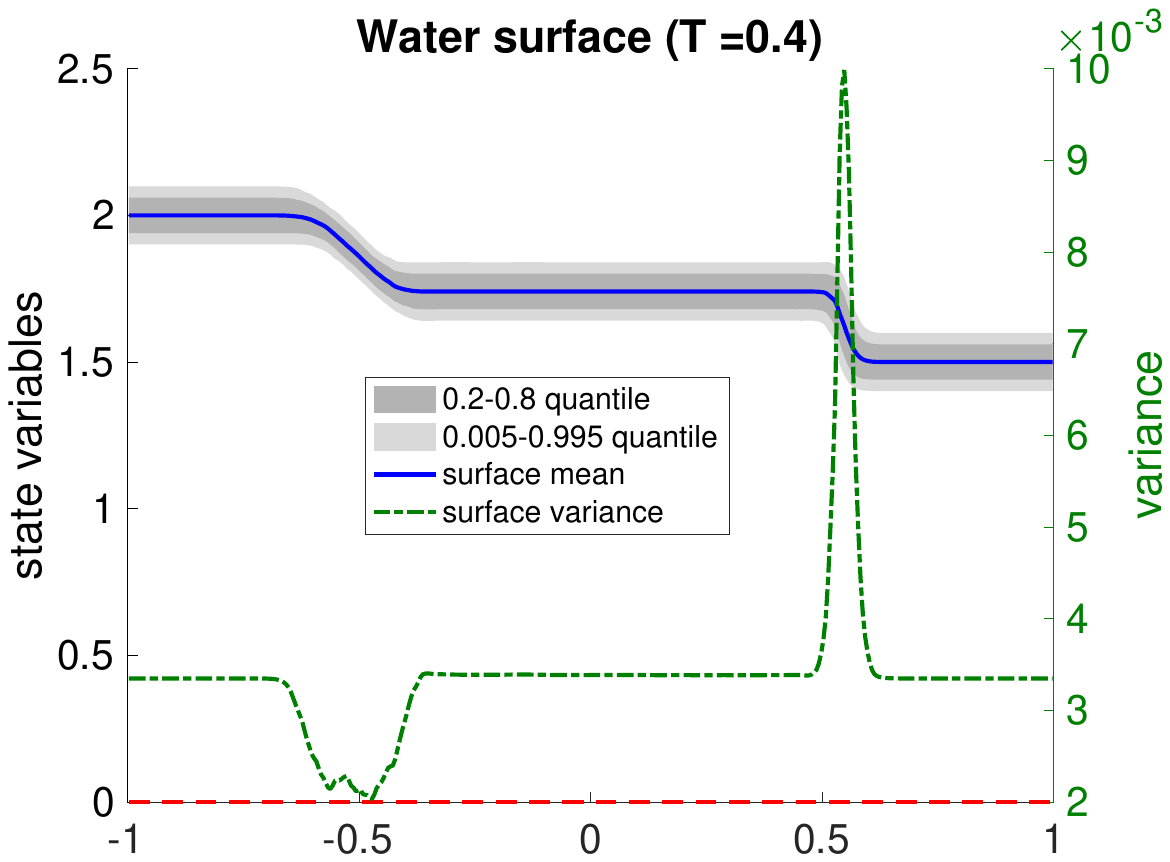}    
    \includegraphics[width = .32\textwidth, trim={0 6cm 0 6cm}, clip]{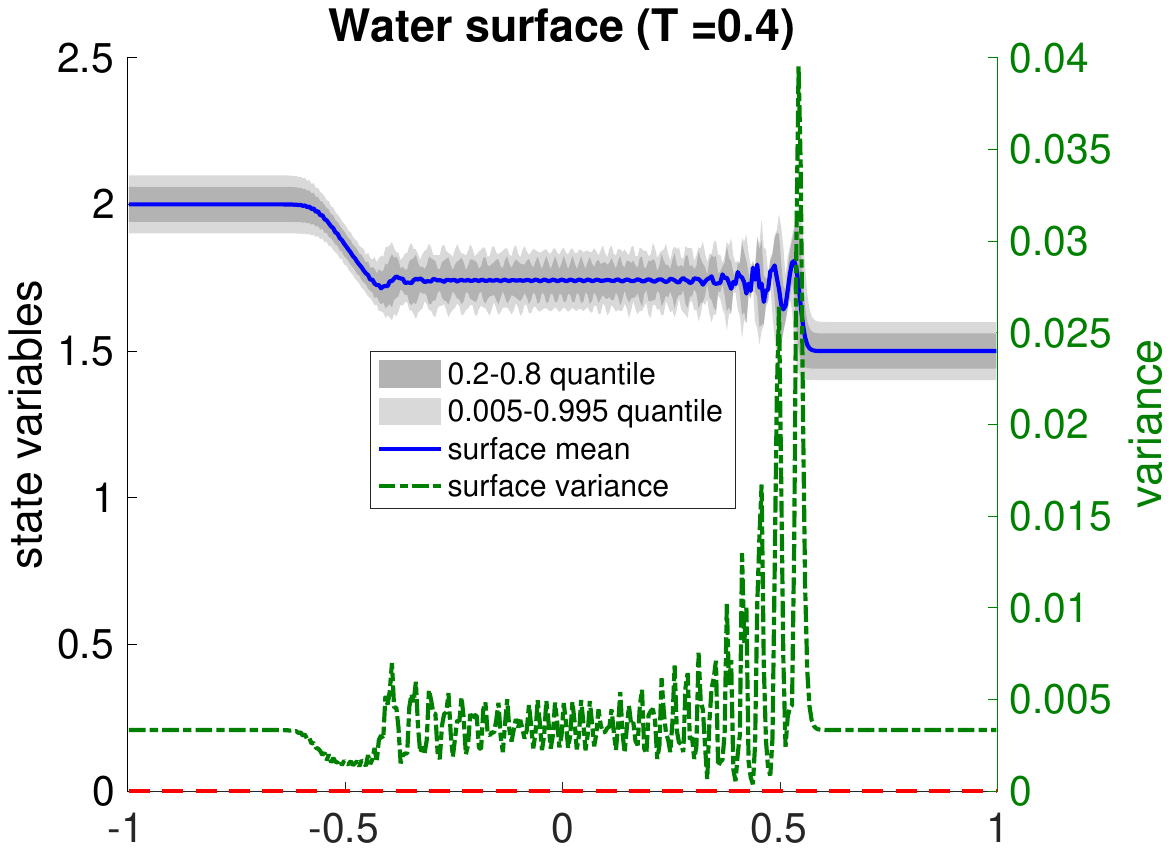}    
    \\
    \includegraphics[width = .32\textwidth, trim={0 6cm 0 6cm}, clip]{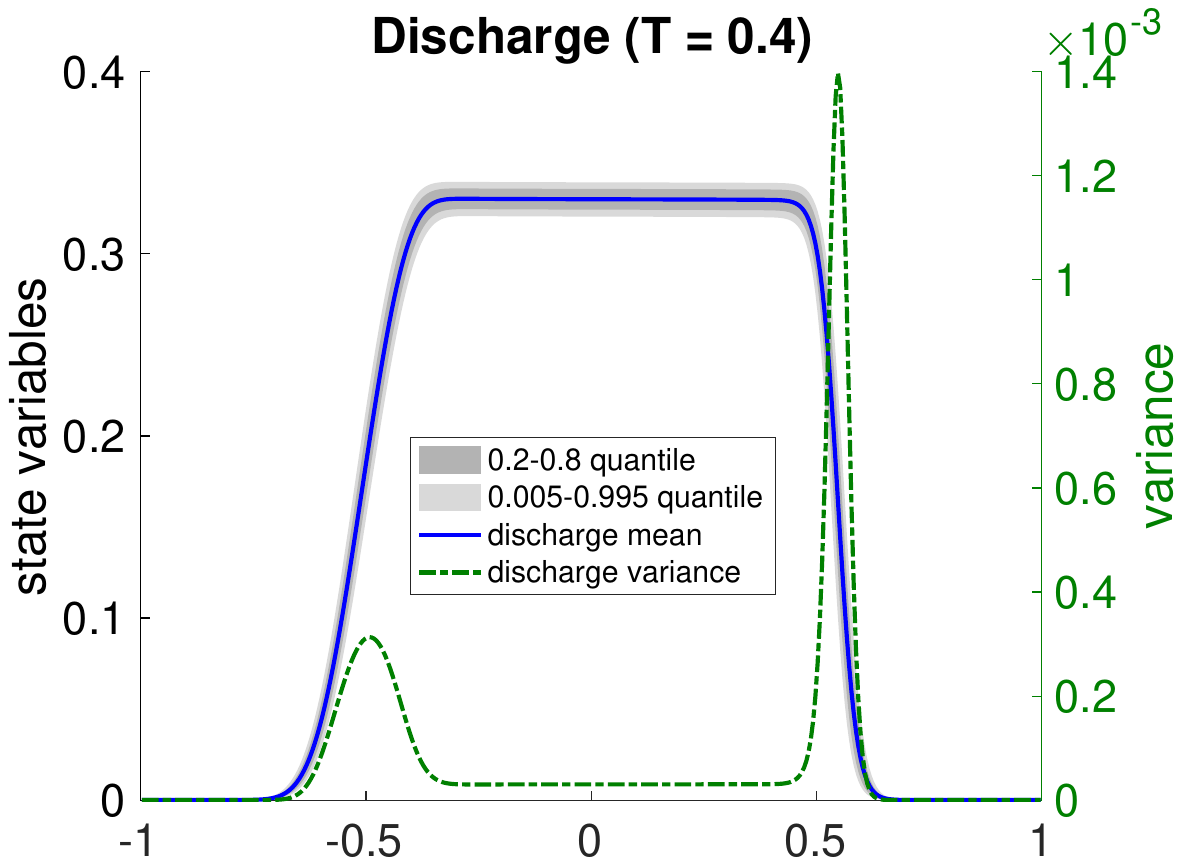}
    \includegraphics[width = .32\textwidth, trim={0 6cm 0 6cm}, clip]{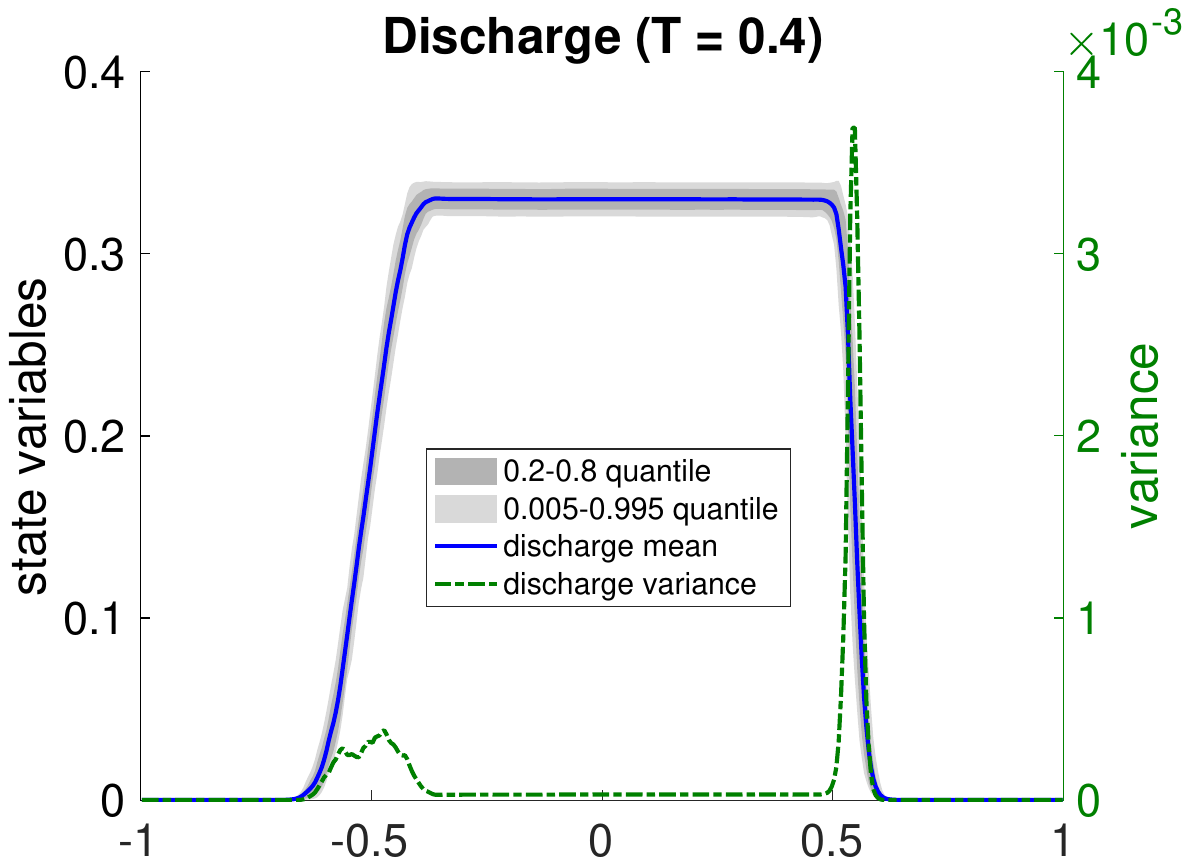}    
    \includegraphics[width = .32\textwidth, trim={0 6cm 0 6cm}, clip]{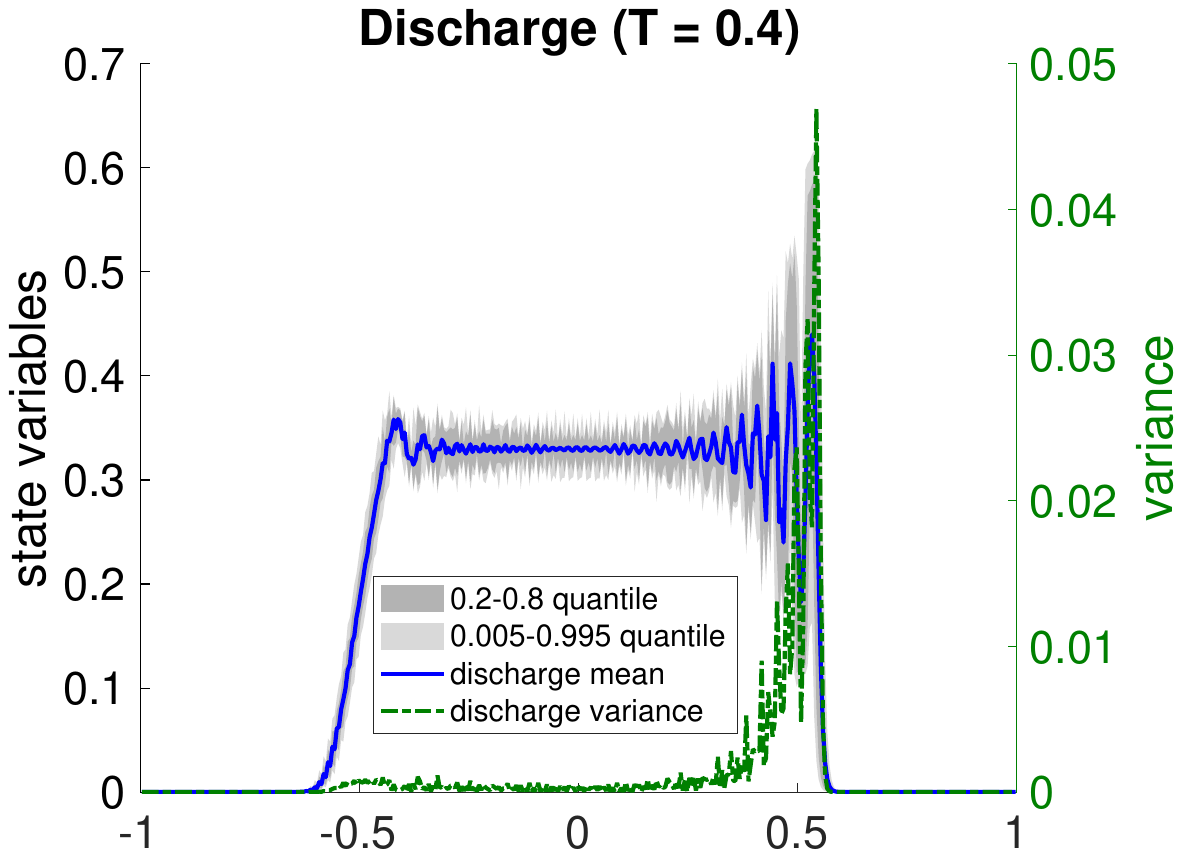}   
   % \includegraphics[width =
  %  .32\textwidth]{figure/1d-ex2-es1-N9-nx400-e.png}
   % \includegraphics[width = .32\textwidth]{figure/1d-ex2-N9-nx400-e-cu1.png}    
  %  \includegraphics[width = .32\textwidth]{figure/1d-ex2-N9-nx400-e-cu2.png}    
    \caption{Results for \cref{ssec:results-flat-bot}. Top: water
      surface, Bottom: discharge: Left: ES1. Middle: ES2. Right: EC. Mesh nx = 400 and PC basis functions K=9.}
    \label{fig:ex1a-flat-bot-wq}
  \end{figure}

  \begin{figure}[htbp]
    \centering
      \includegraphics[width = .32\textwidth, trim={0 6cm 0 6cm}, clip]{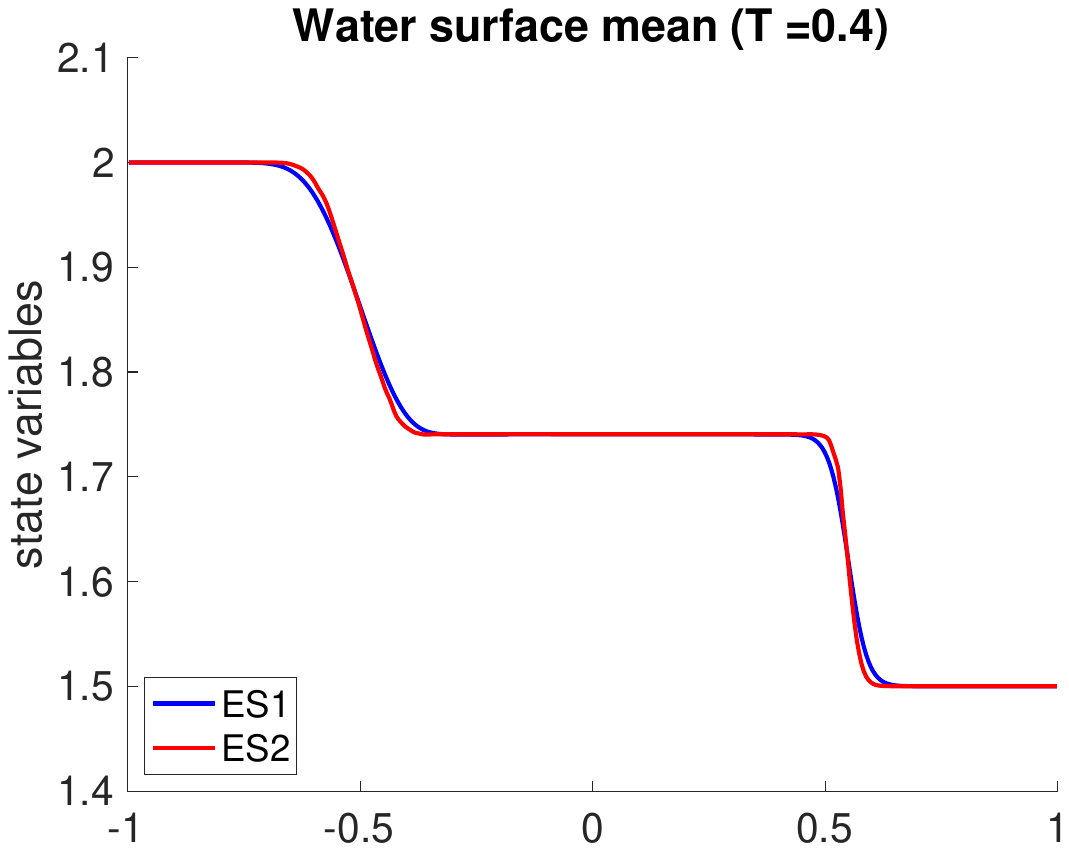}
      \includegraphics[width = .32\textwidth, trim={0 6cm 0 6cm}, clip]{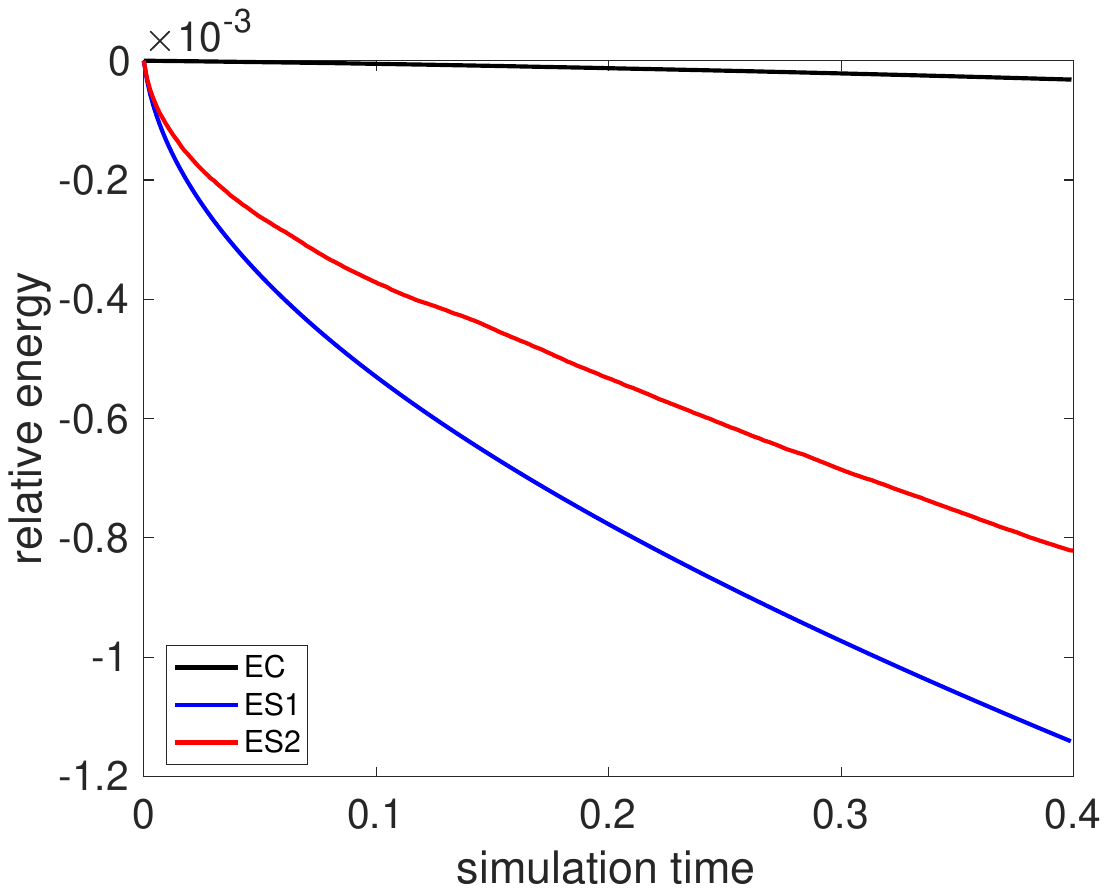}
      \includegraphics[width = .32\textwidth, trim={0 6cm 0 6cm}, clip]{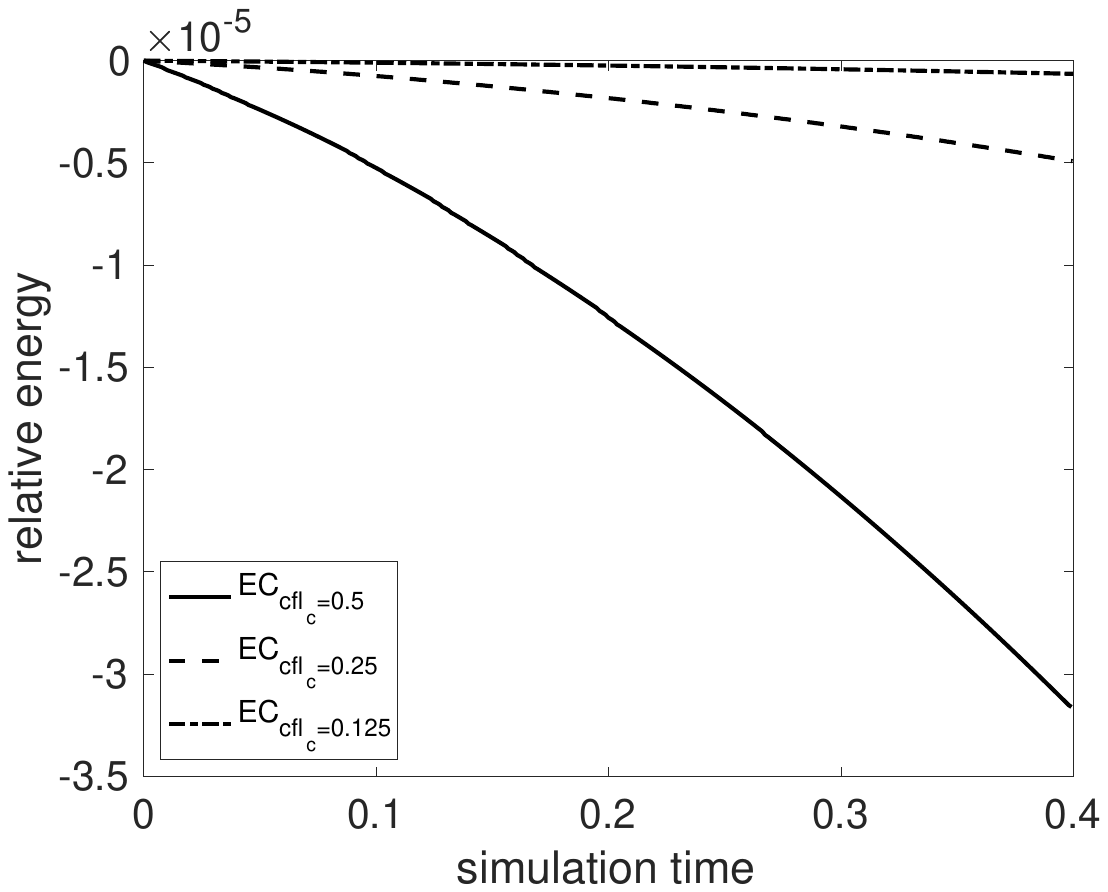}
   % \includegraphics[width =
  %  .32\textwidth]{figure/1d-ex2-es1-N9-nx400-e.png}
   % \includegraphics[width = .32\textwidth]{figure/1d-ex2-N9-nx400-e-cu1.png}    
  %  \includegraphics[width = .32\textwidth]{figure/1d-ex2-N9-nx400-e-cu2.png}    
    \caption{Results for \cref{ssec:results-flat-bot}. Comparison:
      Left - water
      surface mean ES1 vs. ES2. Middle: relative energy change in EC, ES1,
      ES2. Right: relative energy change in EC under different time
      step/CFL constant. Mesh nx = 400 and PC basis functions K=9.}
    \label{fig:ex1a-flat-bot-mean}
  \end{figure}

  %\begin{figure}[htbp]
 %   \centering
 %     \includegraphics[width =
%      .49\textwidth]{new_figure/w_mean_test1a_es1_es2_m400_K9.pdf}
 %       \includegraphics[width =
 %    .49\textwidth]{new_figure/q_mean_test1a_es1_es2_m400_K9.pdf}
   % \includegraphics[width =
  %  .32\textwidth]{figure/1d-ex2-es1-N9-nx400-e.png}
   % \includegraphics[width = .32\textwidth]{figure/1d-ex2-N9-nx400-e-cu1.png}    
  %  \includegraphics[width = .32\textwidth]{figure/1d-ex2-N9-nx400-e-cu2.png}    
%    \caption{Results for \cref{ssec:results-flat-bot}. Top: water
%      surface, Middle: discharge: Left: ES1. Middle: ES2. Right: EC. Bottom: relative change in
%      energy. Mesh nx = 400 and PC basis functions K=9.}
 %   \label{fig:ex-flat-bot}
%\end{figure}

% Continue with next test below: Sept 21 2023
  \subsection{Stochastic Bottom Topography}\label{ssec:results-sbt}
Next, we consider the shallow water system with deterministic initial conditions,
\begin{equation*}%\label{eq:IV1}
w(x,0) = \left\{\begin{aligned}&1&&x<0\\ &0.5&&x>0\end{aligned}\right.,\quad q(x,0) = 0,
\end{equation*} 
and with a stochastic bottom topography,
\begin{equation}\label{eq:bottom 1}
B(x,\xi) = \left\{\begin{aligned}0.125(\cos(5\pi x)+2)+0.125\xi,\quad&|x|<0.2\\0.125+0.125\xi,\quad&\text{otherwise.}\end{aligned}\right.
\end{equation}
This test example was presented previously in \cite{doi:10.1137/20M1360736}. Initially, the highest possible
bottom barely touches the initial water surface at $x=0$, see
\cref{fig:ex1-et-m400}-\cref{fig:ex1-et-m1600}. In
\cref{fig:ex1-et-m400}-\cref{fig:ex1-et-m1600}, we use a uniform grid
size $\Delta x =400, 800, 1600$ over the
physical domain $x \in [-1,1]$, and compute up to time $t=0.0995$
(Immediately after this time, the EC scheme fails for an nx=$400$ due to spurious oscillations near
sharp gradients of the solution). In
\cref{fig:ex1-et-m800}-\cref{fig:ex1-et-m1600} we compare only
performance of ES1 and ES2 at $t=0.0995$ since EC fails on those
meshes even earlier. Again, the numerical results indicate that the ES2 scheme can more easily resolve large, spatially-concentrated variance values compared to the ES1 scheme, but under mesh refinement both schemes converge to
similar numerical solutions. In \cref{fig:ex1-m800}, we show numerical
solution obtained using ES1 and ES2 at the final time $t=0.8$ and on mesh
$nx=800$. For both schemes,  the $99\%$ confidence region of the water
surface stays above the $99\%$ confidence region of the bottom
function in \cref{fig:ex1-m800}, and both methods produce similar numerical
solutions. The presented results are comparable to the results in
\cite[Section 5.1]{doi:10.1137/20M1360736}. In \cref{fig:ex1-rel_E}, we again observe as expected that
the EC scheme numerically conserves energy, while ES1 and ES2
dissipate energy with larger dissipation produced by ES1 method.
\begin{figure}[htbp]
    \centering
    \includegraphics[width = .32\textwidth, trim={0 6cm 0 6cm}, clip]{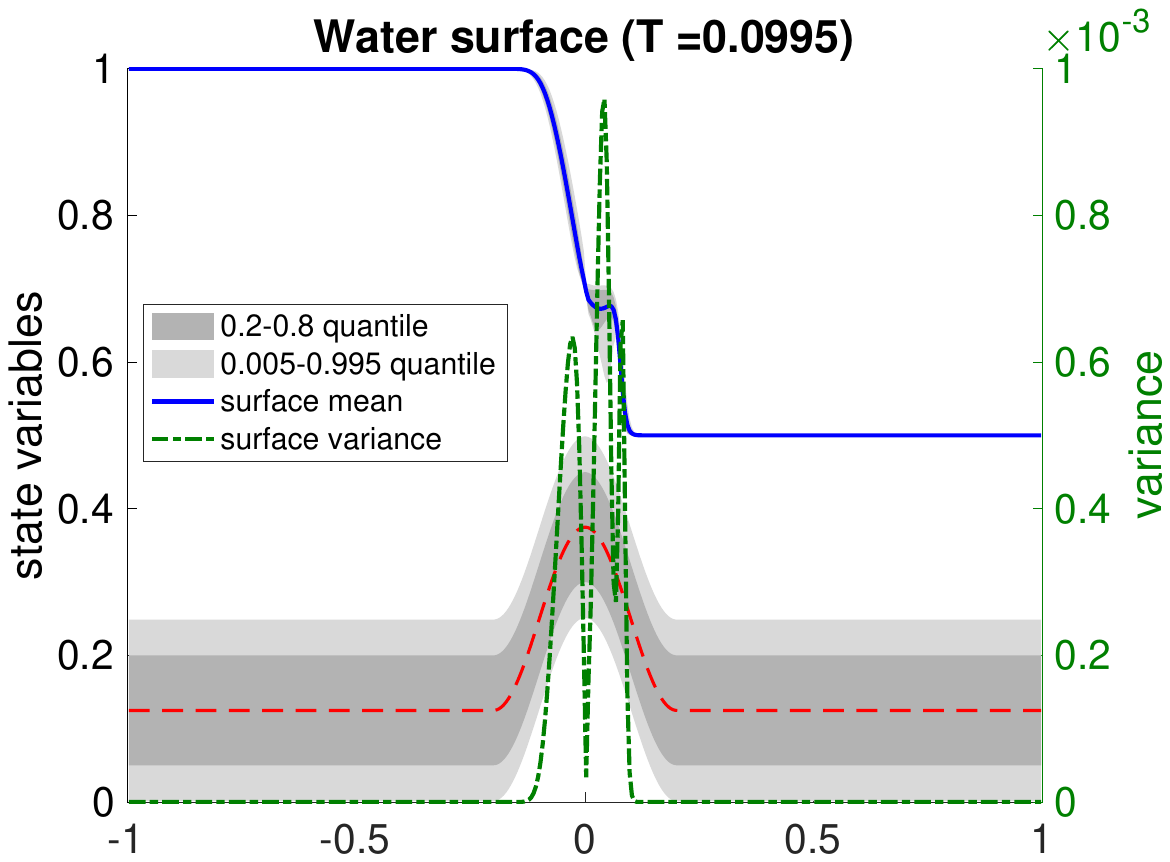}
    \includegraphics[width = .32\textwidth, trim={0 6cm 0 6cm}, clip]{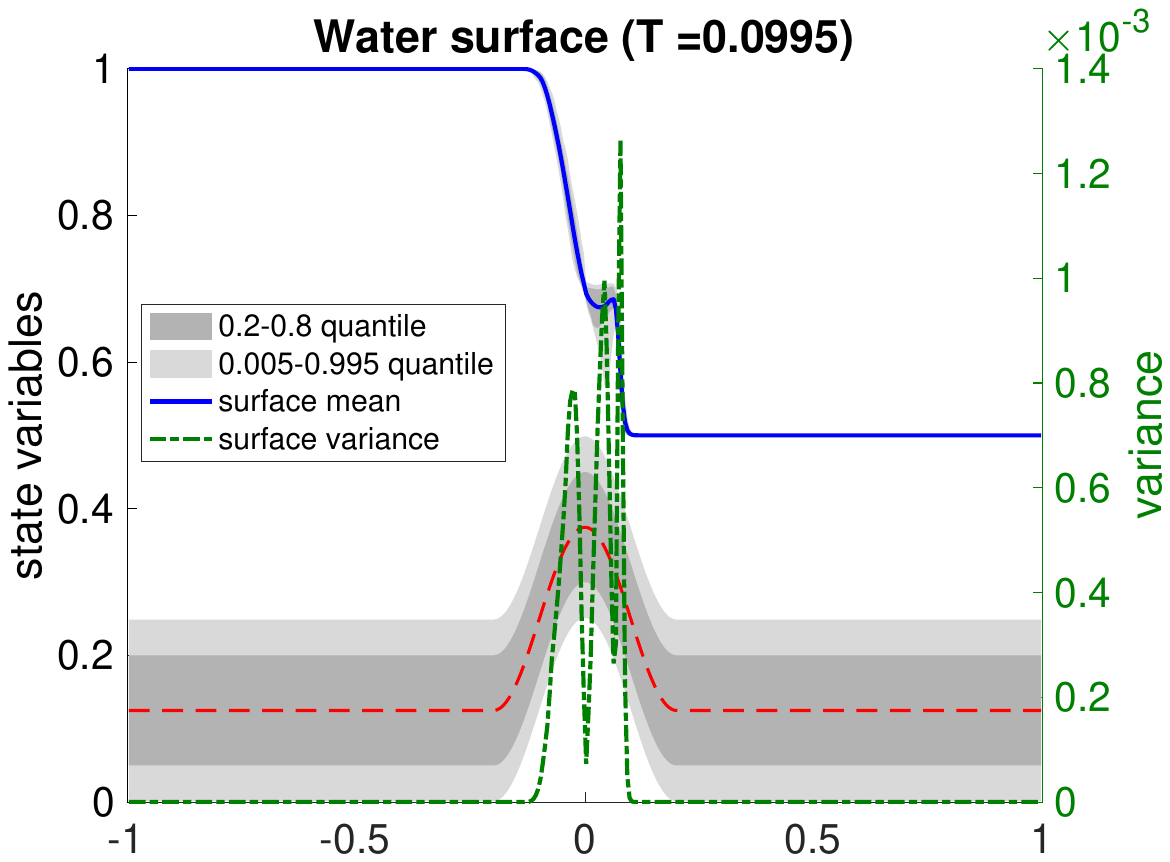}    
    \includegraphics[width = .32\textwidth, trim={0 6cm 0 6cm}, clip]{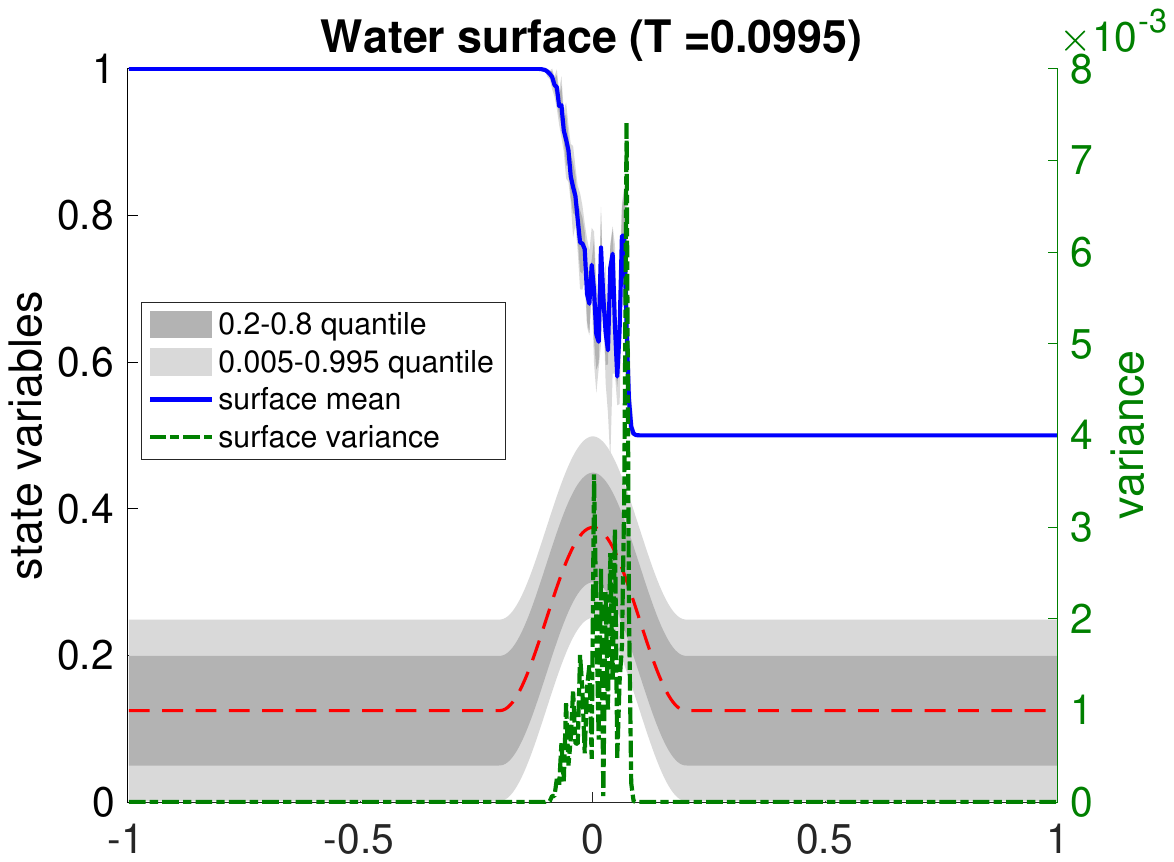}    
    \\
    \includegraphics[width = .32\textwidth, trim={0 6cm 0 6cm}, clip]{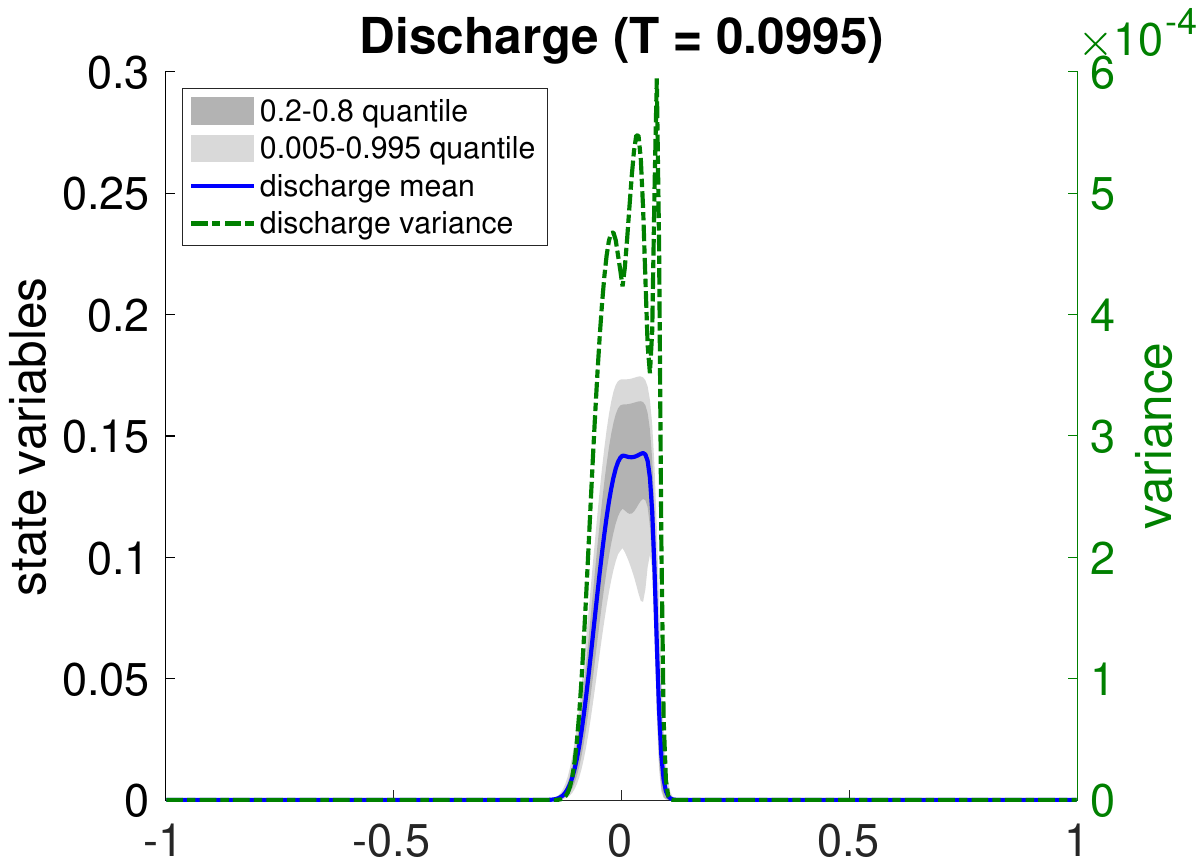}
    \includegraphics[width = .32\textwidth, trim={0 6cm 0 6cm}, clip]{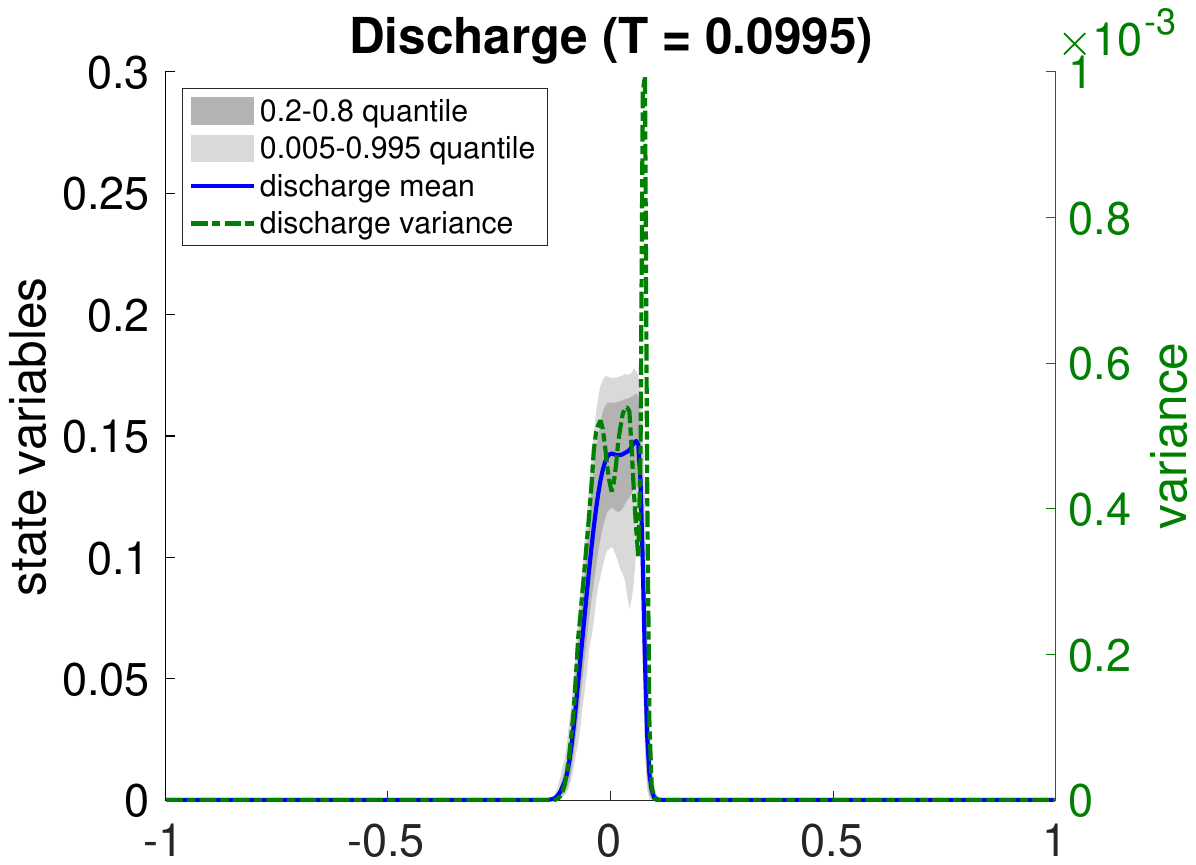}    
    \includegraphics[width = .32\textwidth, trim={0 6cm 0 6cm}, clip]{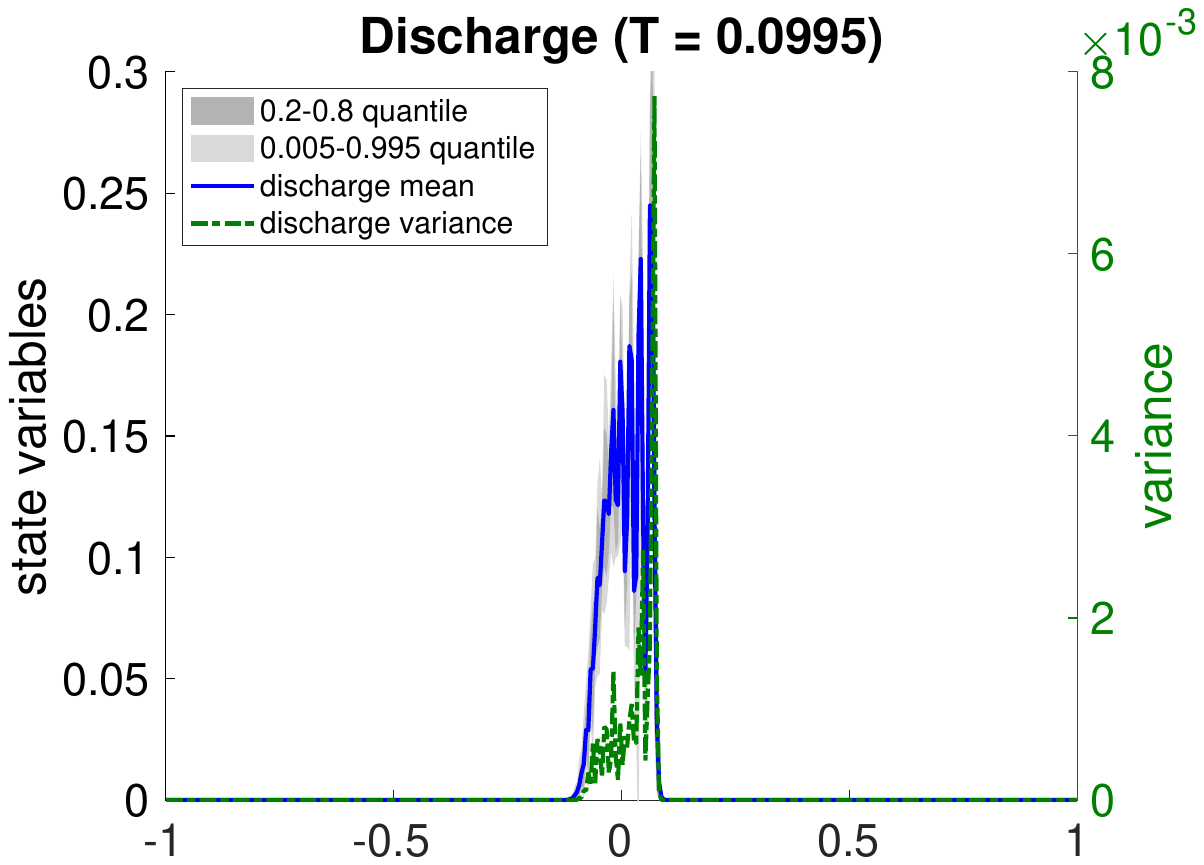}   
    \caption{Comparison of the results for \cref{ssec:results-sbt} using
      different schemes. Top: water surface. Bottom: discharge. Left:
      ES1, Middle: ES2, Right: EC. Mesh $nx=400$ with $K=9$ at earlier
    time $T=0.0995$.}
    \label{fig:ex1-et-m400}
  \end{figure}
  \begin{figure}[htbp]
    \centering
    \includegraphics[width = .49\textwidth, trim={0 6cm 0 6cm}, clip]{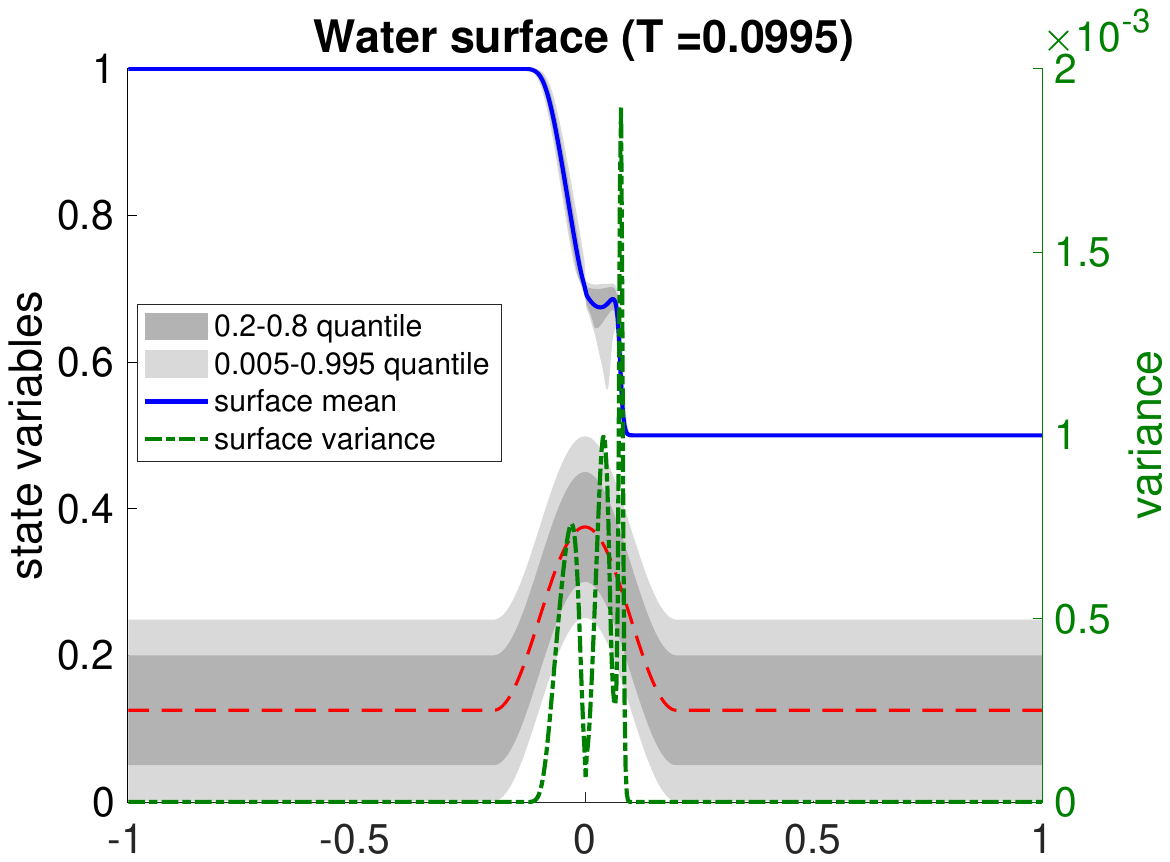}
    \includegraphics[width = .49\textwidth, trim={0 6cm 0 6cm}, clip]{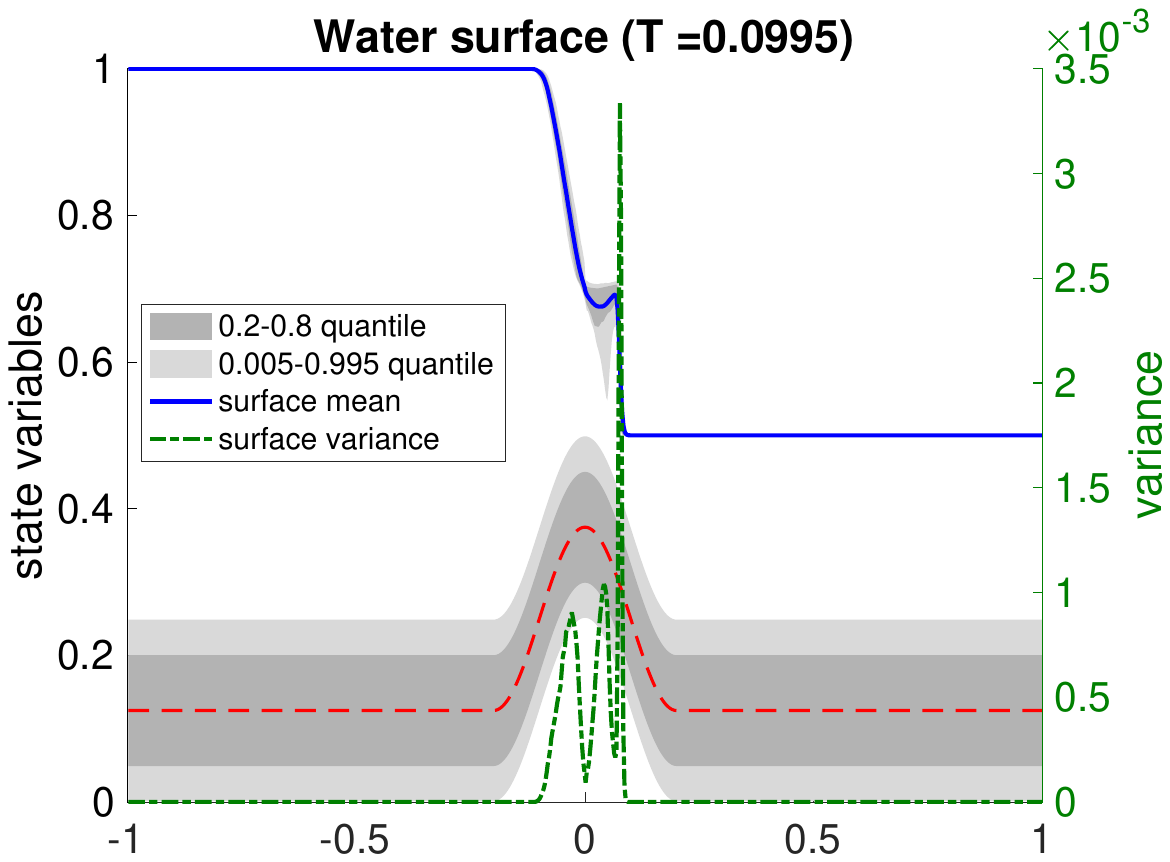}   
    \\
  \includegraphics[width = .49\textwidth, trim={0 6cm 0 6cm}, clip]{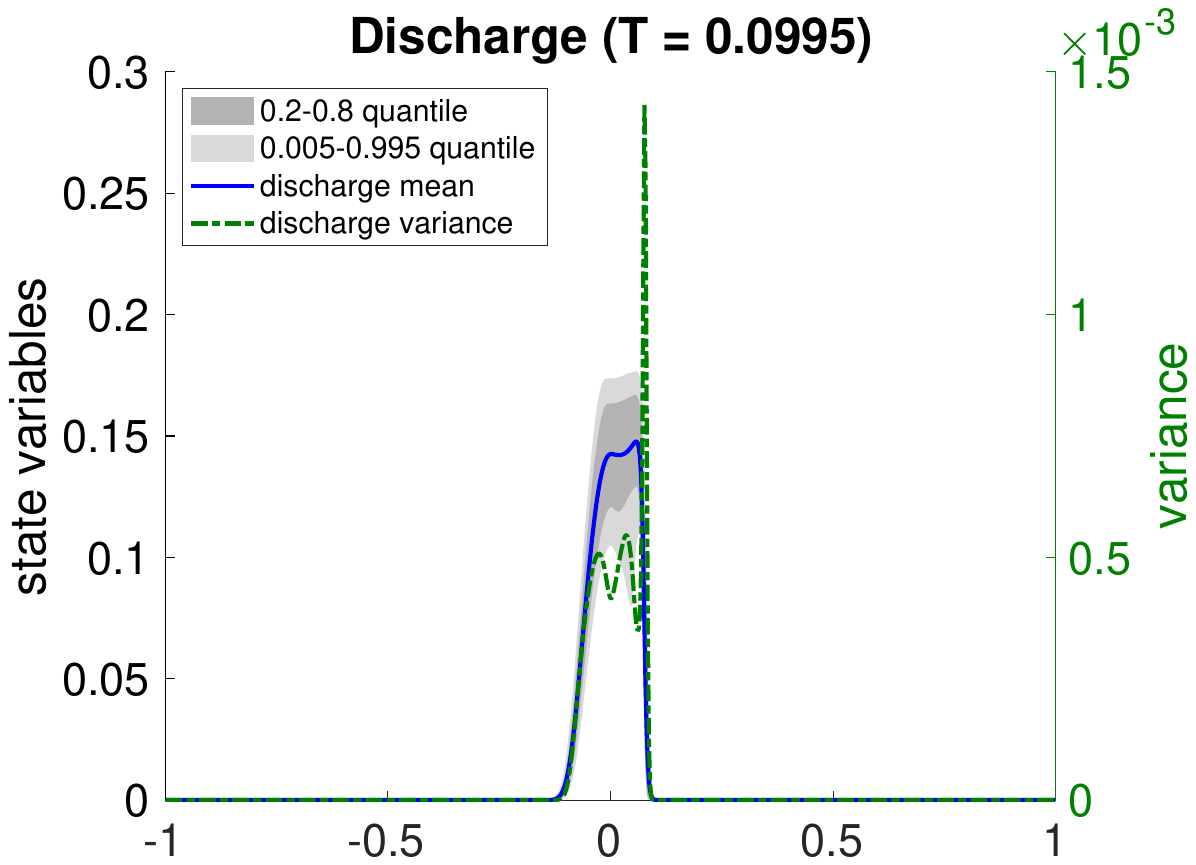}
    \includegraphics[width = .49\textwidth, trim={0 6cm 0 6cm}, clip]{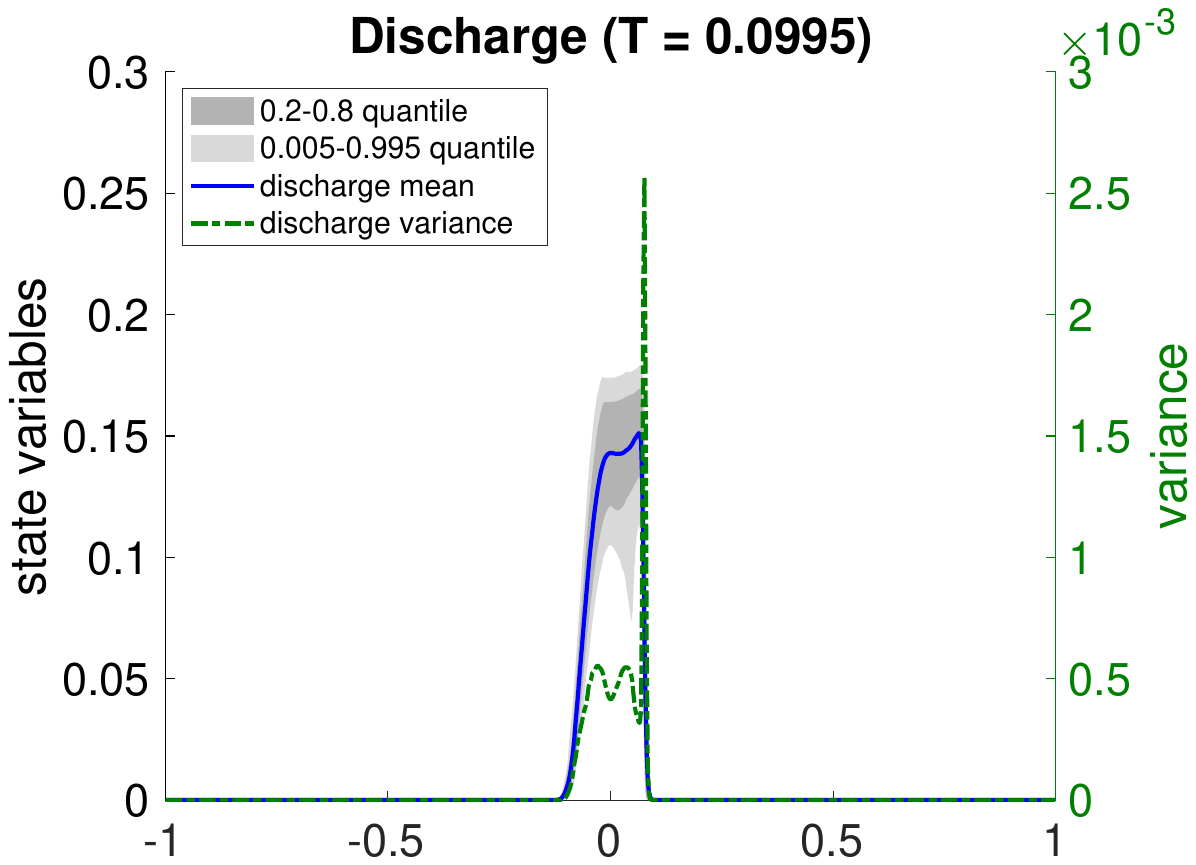}   
    \caption{Comparison of the results for \cref{ssec:results-sbt} using
      different schemes. Top: water surface. Bottom: discharge. Left:
      ES1, Right: ES2. Mesh $nx=800$ with $K=9$ at earlier
    time $T=0.0995$.}
    \label{fig:ex1-et-m800}
  \end{figure}
  \begin{figure}[htbp]
    \centering
    \includegraphics[width = .49\textwidth, trim={0 6cm 0 6cm}, clip]{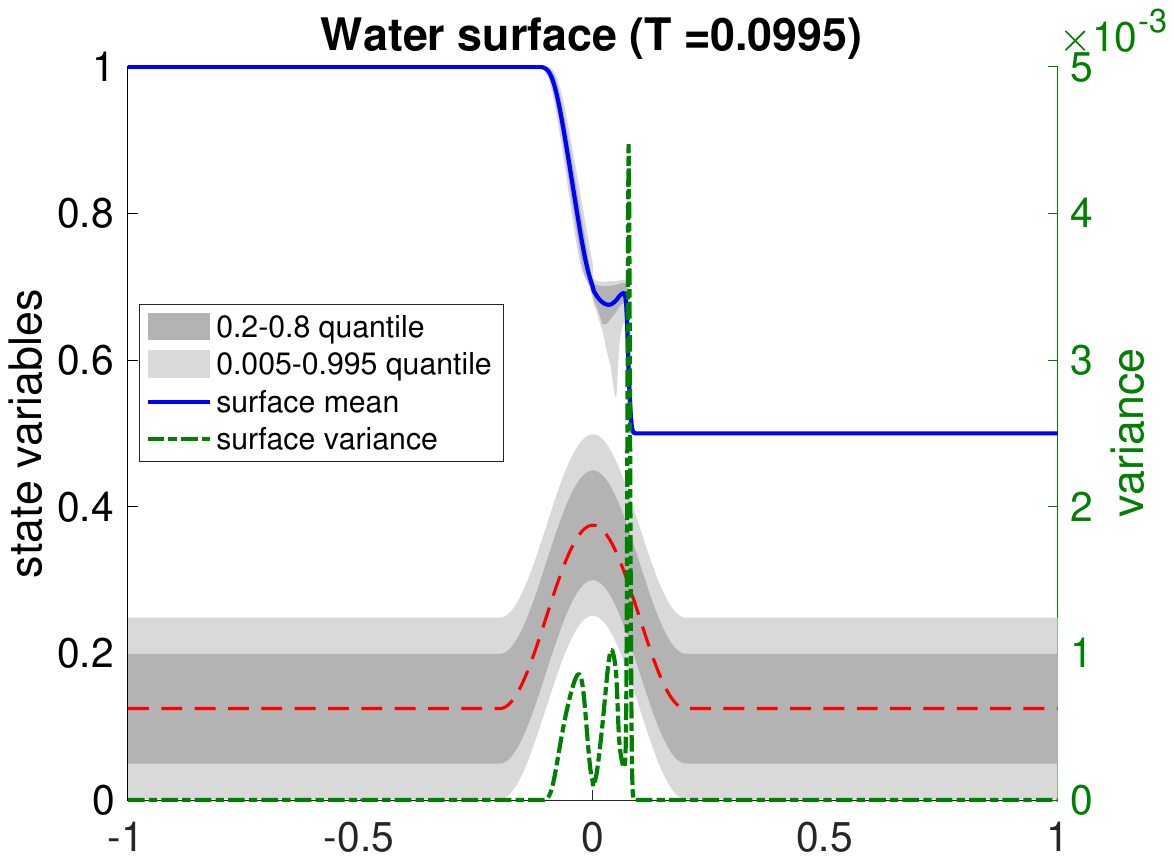}
    \includegraphics[width = .49\textwidth, trim={0 6cm 0 6cm}, clip]{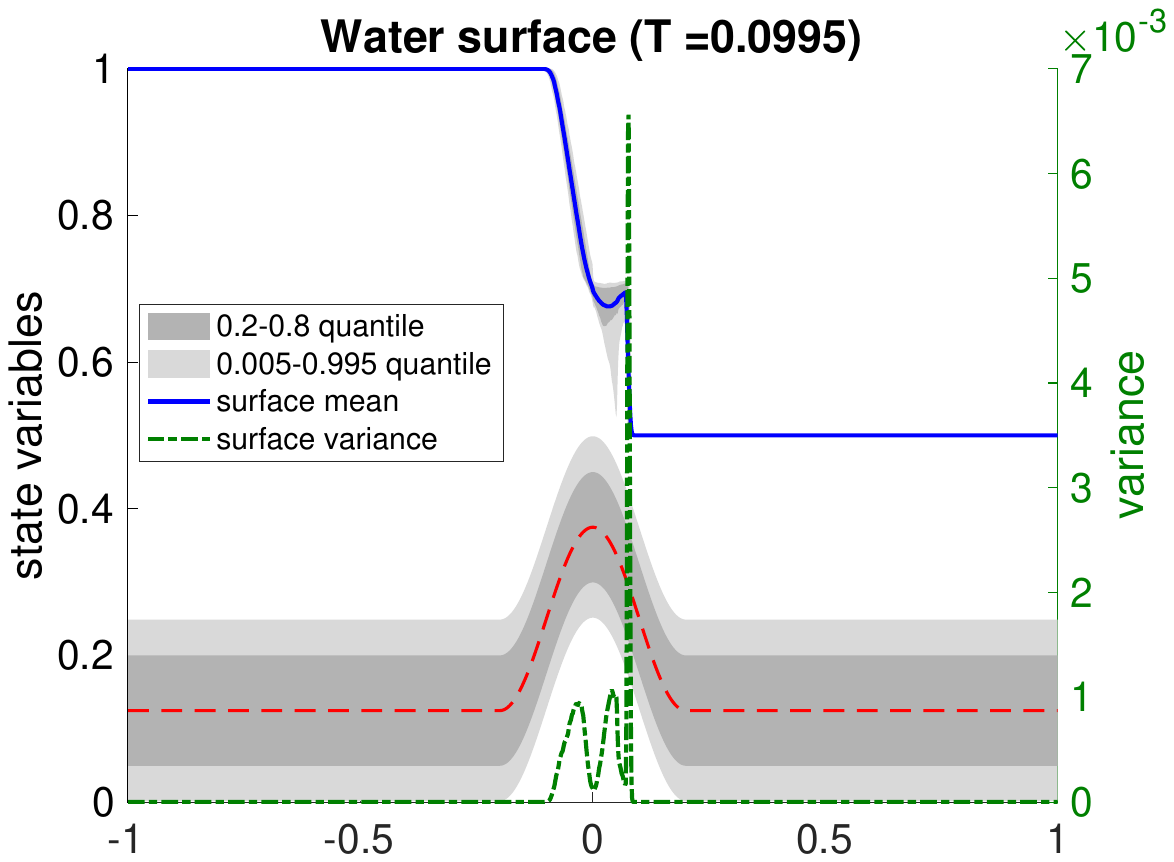}  
    \\
  \includegraphics[width = .49\textwidth, trim={0 6cm 0 6cm}, clip]{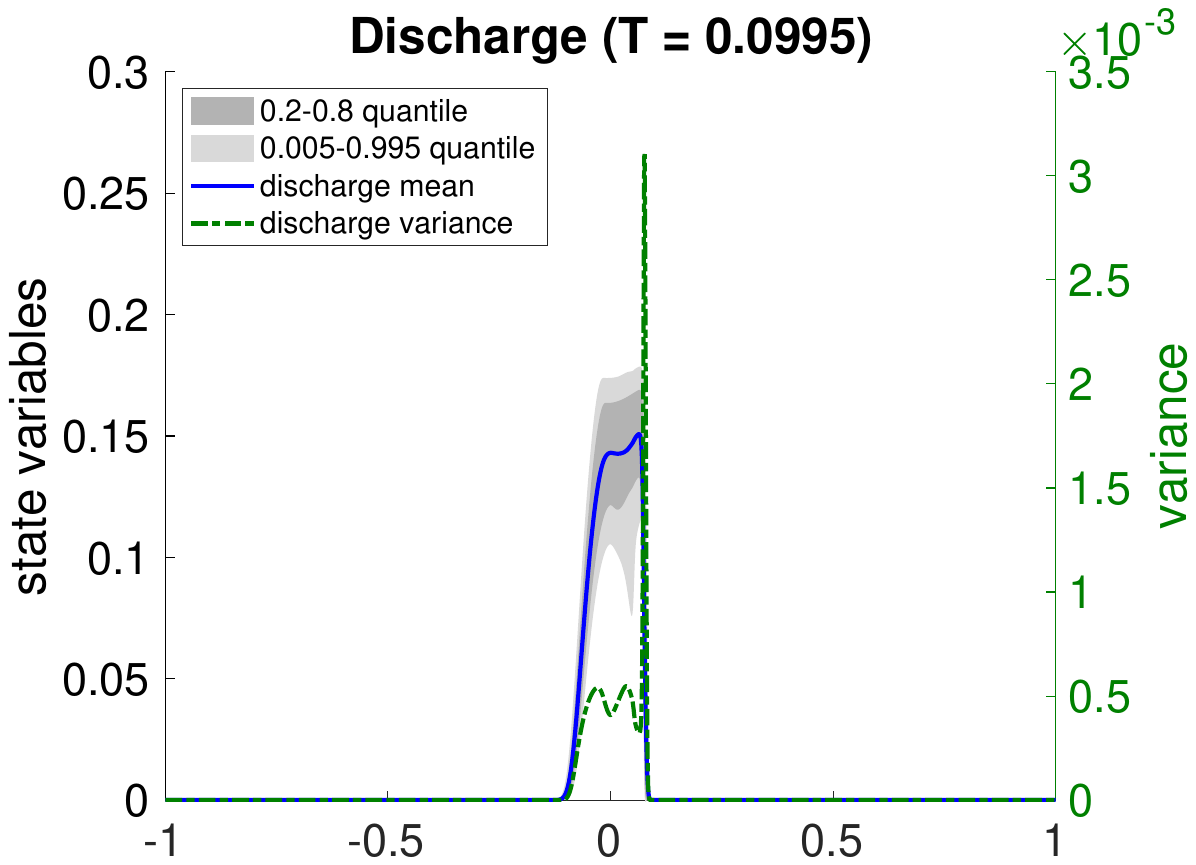}
    \includegraphics[width = .49\textwidth, trim={0 6cm 0 6cm}, clip]{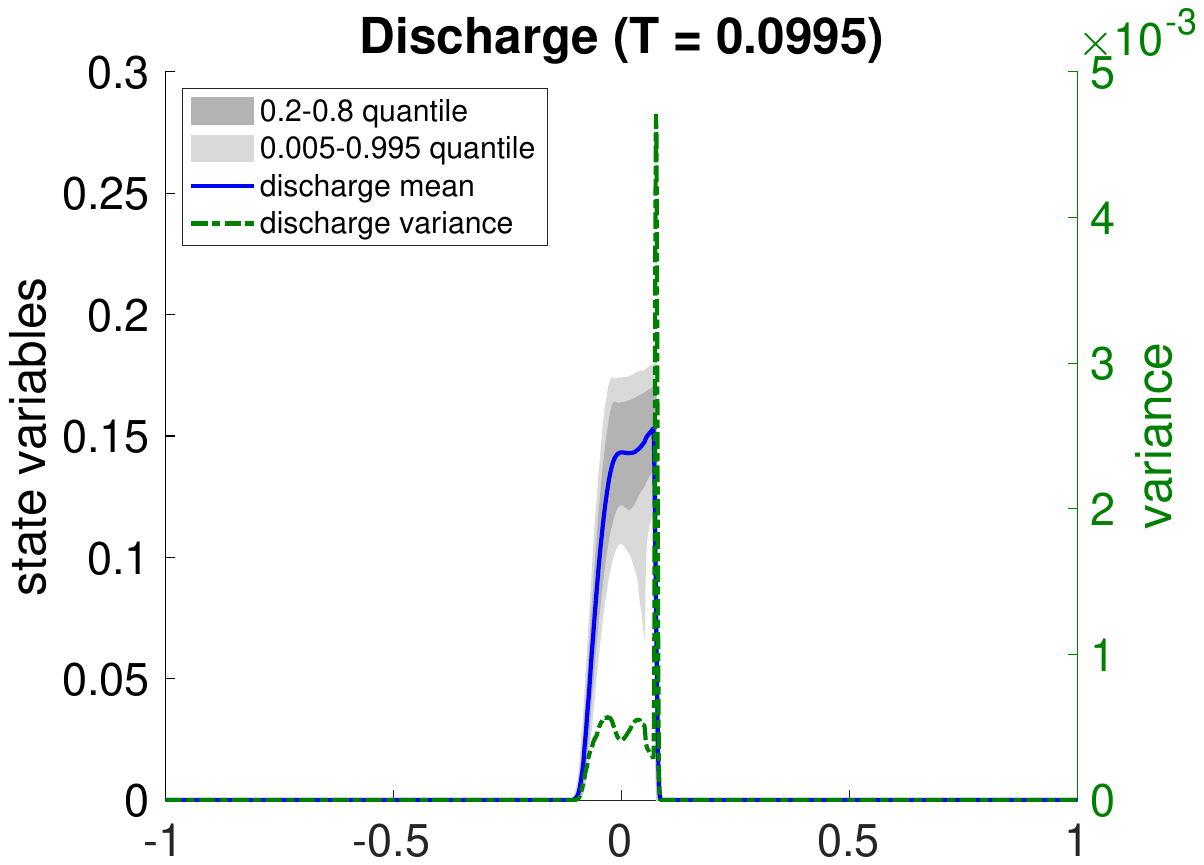}   
    \caption{Comparison of the results for \cref{ssec:results-sbt} using
      different schemes. Top: water surface. Bottom: discharge. Left:
      ES1, Right: ES2. Mesh $nx=1600$ with $K=9$ at earlier
    time $T=0.0995$.}
    \label{fig:ex1-et-m1600}
  \end{figure}
  \begin{figure}[htbp]
    \centering
    \includegraphics[width = .49\textwidth, trim={0 6cm 0 6cm}, clip]{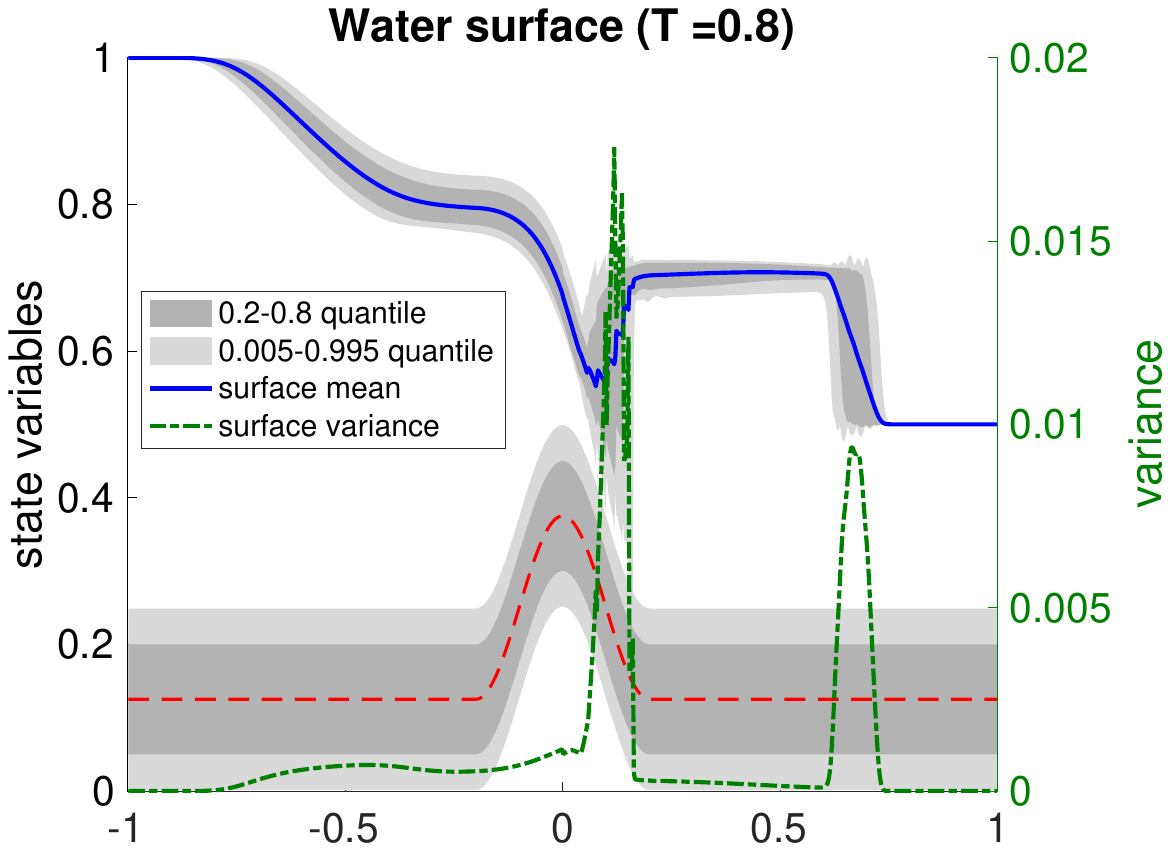}
    \includegraphics[width = .49\textwidth, trim={0 6cm 0 6cm}, clip]{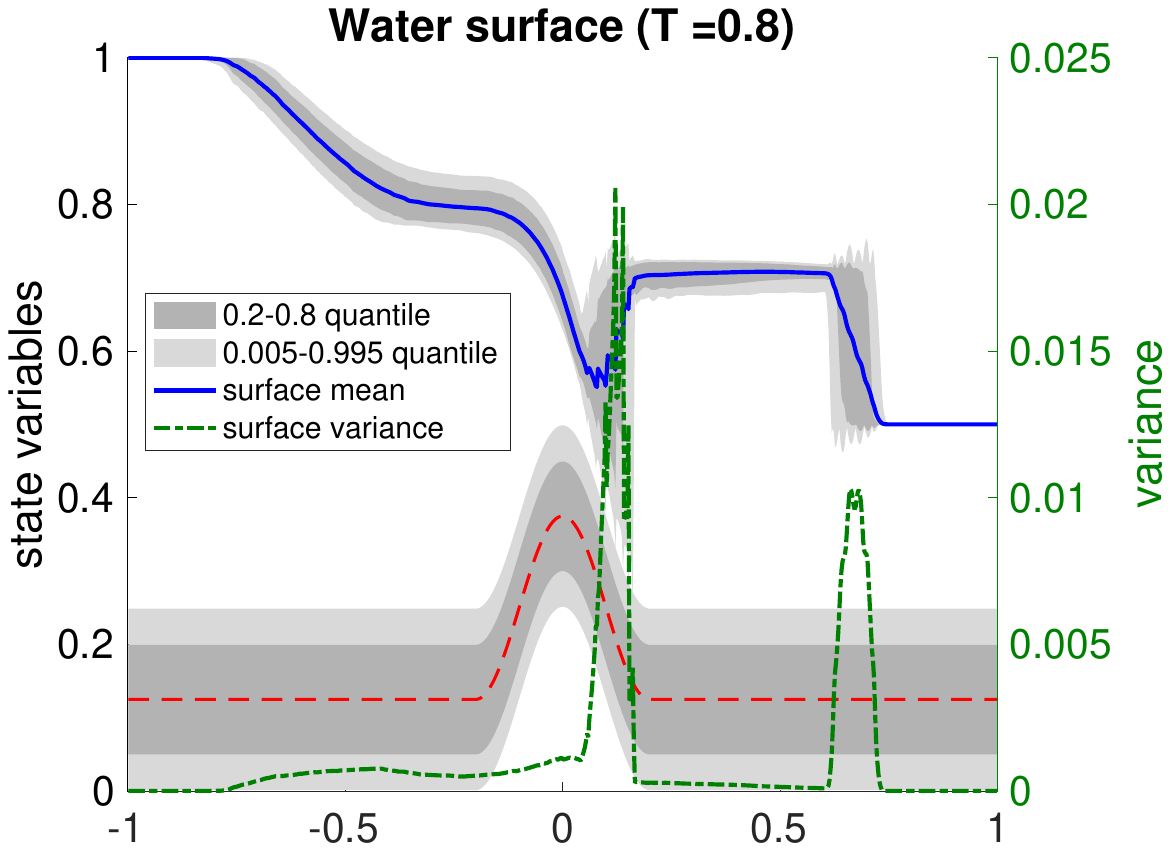}
    \\
  \includegraphics[width = .49\textwidth, trim={0 6cm 0 6cm}, clip]{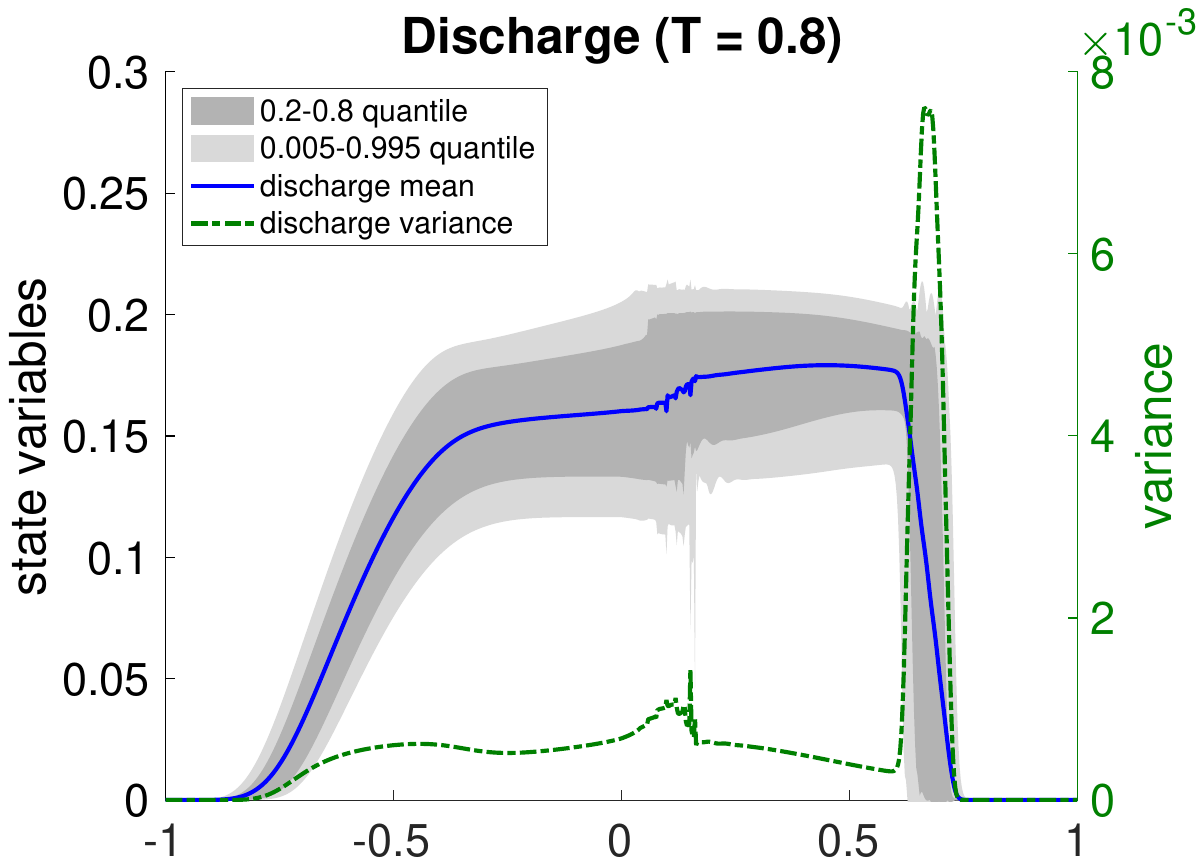}
    \includegraphics[width = .49\textwidth, trim={0 6cm 0 6cm}, clip]{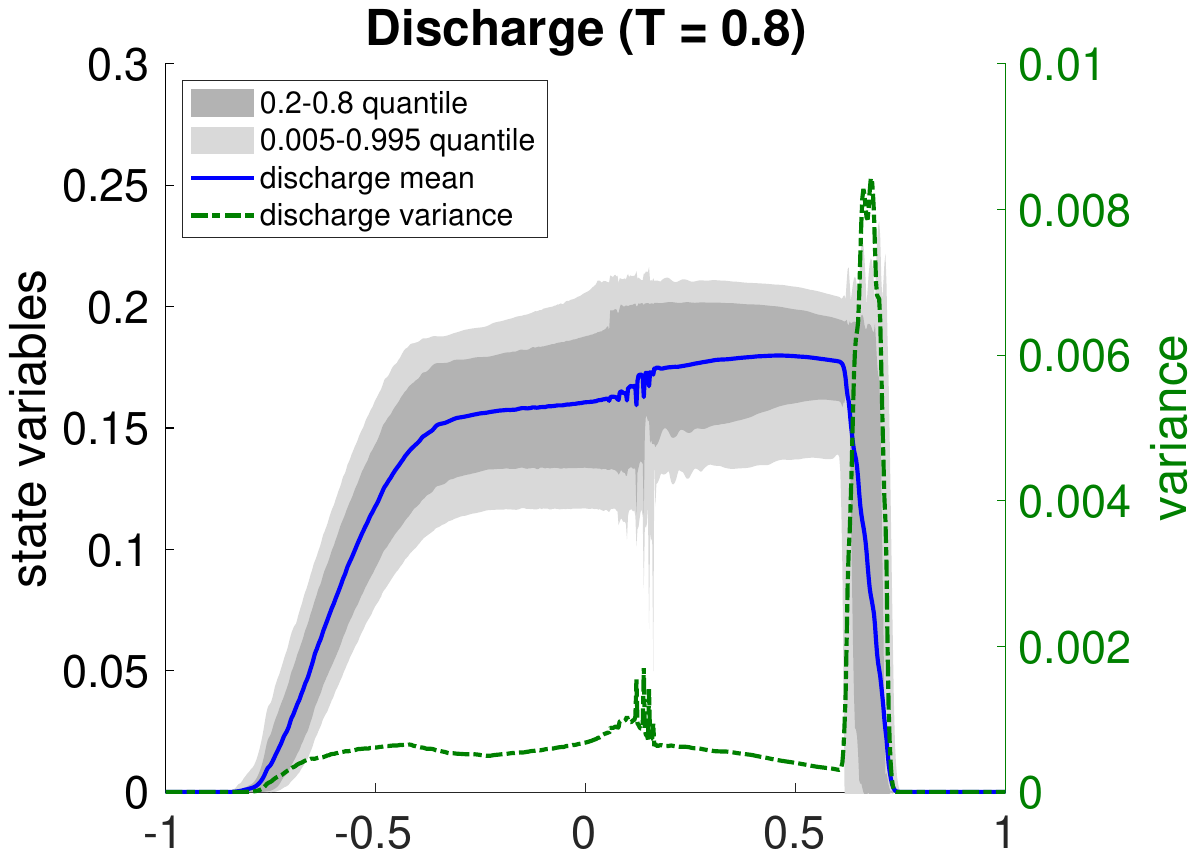}
    \caption{Comparison of the results for \cref{ssec:results-sbt} using
      different schemes. Top: water surface. Bottom: discharge. Left:
      ES1, and Right: ES2.  Mesh $nx=800$ with $K=9$ at the final
    time $T=0.8$.}
    \label{fig:ex1-m800}
  \end{figure}
  \begin{figure}[htbp]
    \centering
     \includegraphics[width = .49\textwidth, trim={0 6cm 0 6cm}, clip]{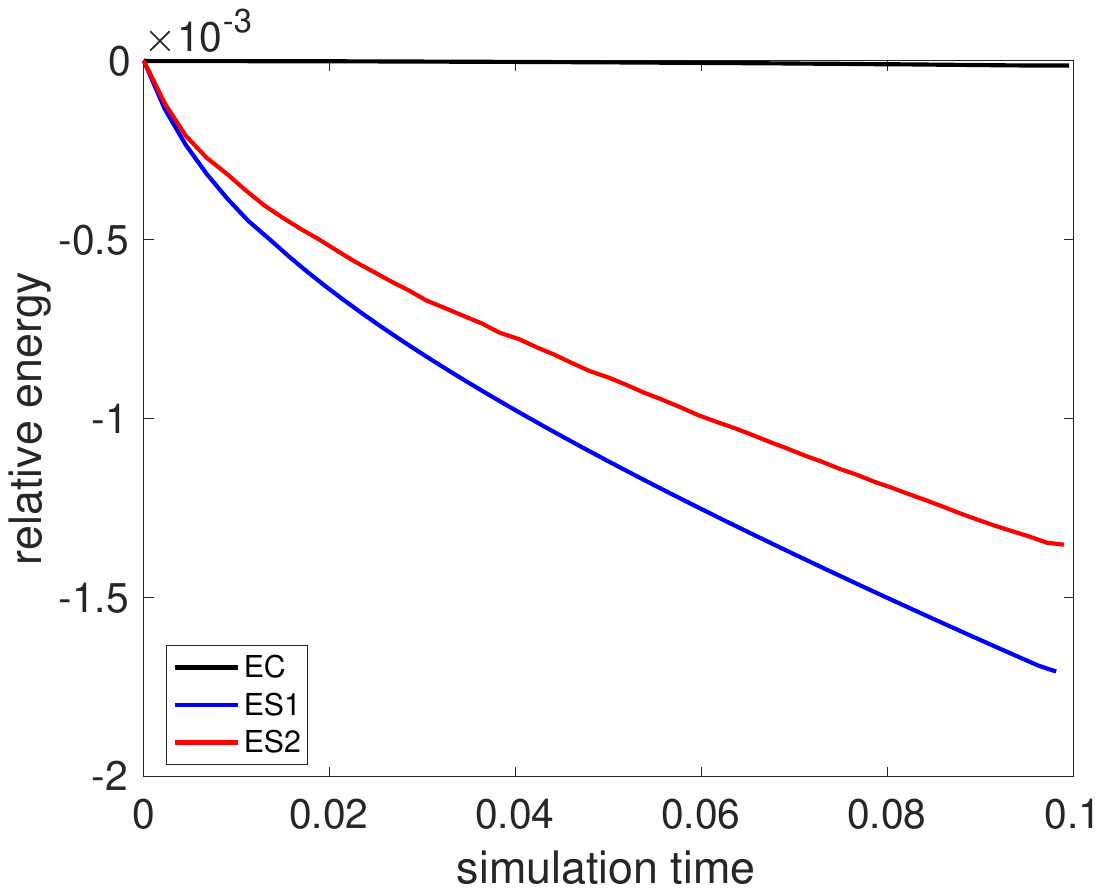}
       \includegraphics[width = .49\textwidth, trim={0 6cm 0 6cm}, clip]{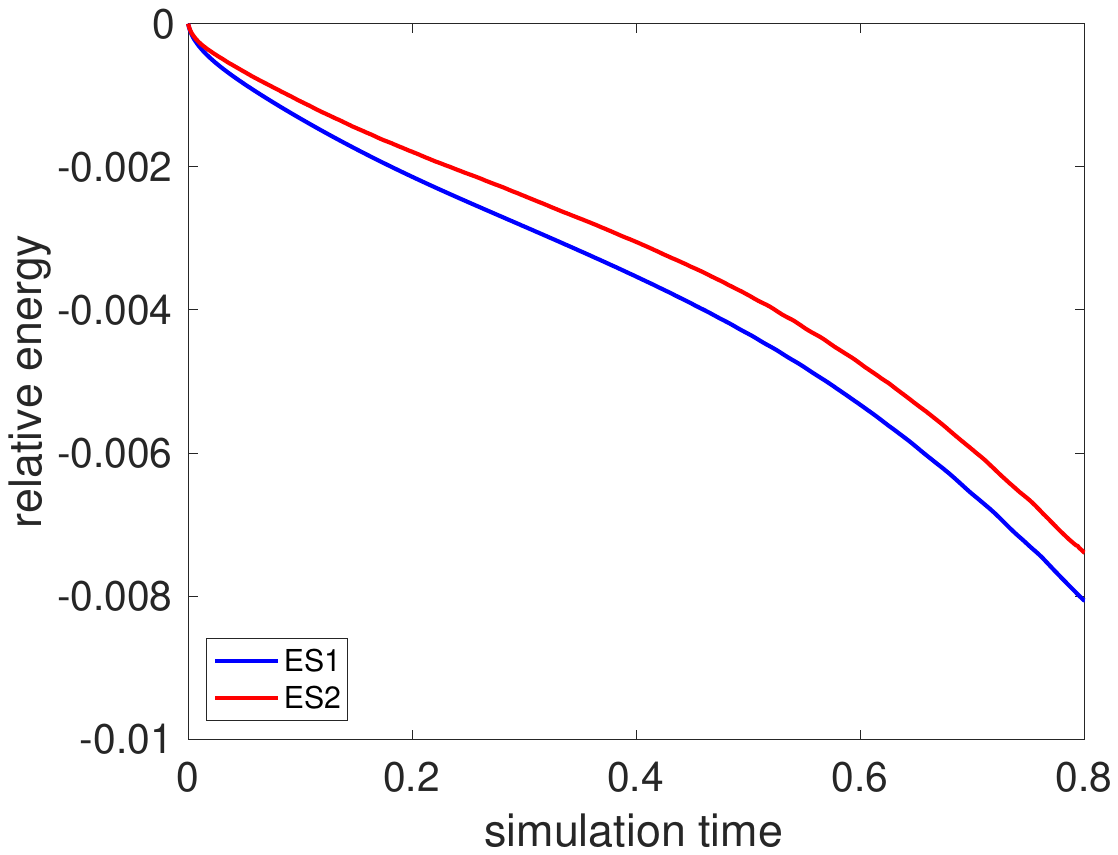}
    \caption{Comparison of the results for \cref{ssec:results-sbt} using
      different schemes. Relative energy change: Left: EC vs. ES1
    vs. ES2 on mesh $nx=400$ on time interval $[0, 0.0995]$. Right:
    ES1 vs. ES2 on mesh $nx=800$ on time interval $[0, 0.8]$.}
    \label{fig:ex1-rel_E}
\end{figure}
\subsection{Perturbation to Lake at Rest}\label{ssec:wbperturb2}
As a final example, we consider the shallow water system with stochastic water surface,
\begin{equation}\label{eq:IV-wbperturb1}
w(x,0,\xi) = \left\{\begin{aligned}&1+0.001(\xi + 1)&&\vert x\vert \le 0.05\\ &1&&\text{otherwise}\end{aligned}\right.,\quad q(x,0,\xi) = 0,
\end{equation} 
and with a deterministic bottom topography
\begin{equation}\label{eq:bt-wbperturb1}
B(x) = \left\{\begin{aligned}0.25(\cos(5\pi(x+0.35))+1),\quad &-0.55 < x < -0.15\\
0.125(\cos(10\pi(x-0.35))+1),\quad &0.25 < x < 0.45\\
0,\quad &\text{otherwise.}\end{aligned}\right.
\end{equation}
The test is from \cite{CKJin1} and is similar to the deterministic
tests of the perturbation of lake at rest solution, for example
to the one presented in \cite{fjordholm2011well}. From
presented results in \cref{fig:ex2-mean}, \cref{fig:ex2-m400} and
\cref{fig:ex2-m1600}, we make conclusions similar to previous sections: Both ES1 and ES2 capture small stochastic perturbations of
the lake at rest solution quite well (with both leftward- and rightward- going
waves present in the numerical solutions). The first order ES1 scheme exhibits
much more dissipation in the left and the right going waves than the ES2
scheme, which produces a more accurate solution, as shown in
\cref{fig:ex2-mean}, \cref{fig:ex2-m400} (left and middle figures),
\cref{fig:ex2-m1600}. The
results of EC scheme is also shown in \cref{fig:ex2-m400} (right
figure). The EC scheme resolves the left and the right going water waves with
heights higher than in both ES1 and ES2 methods, but again there are
oscillations present near both waves in EC numerical solution as
expected since EC does not dissipates energy across shocks. The
relative energy change for this example produced by EC, ES1 and ES2
methods is illustrated in \cref{fig:ex2-r_E}. The
presented results are also comparable to the results reported in
\cite{fjordholm2011well} and in \cite{CKJin1}.
\begin{figure}[htbp]
    \centering
    \includegraphics[width = .32\textwidth, trim={0 6cm 0 6cm}, clip]{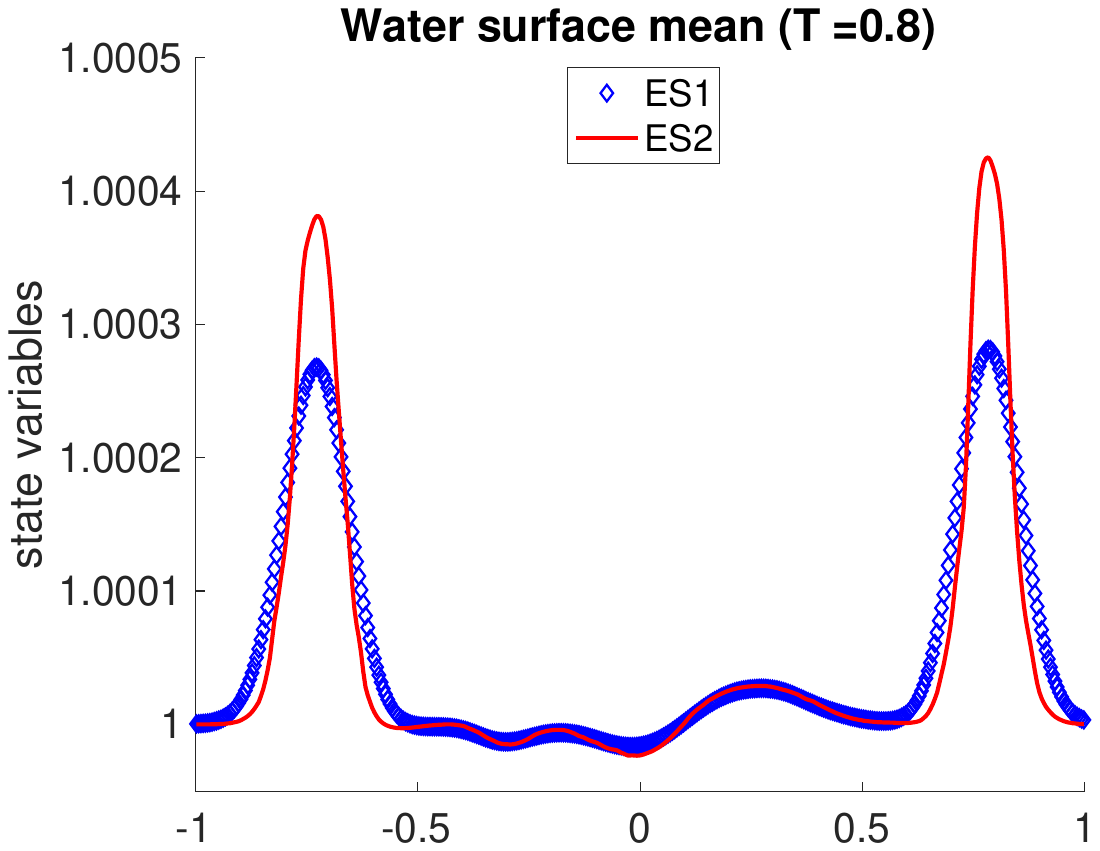}
    \includegraphics[width = .32\textwidth, trim={0 6cm 0 6cm}, clip]{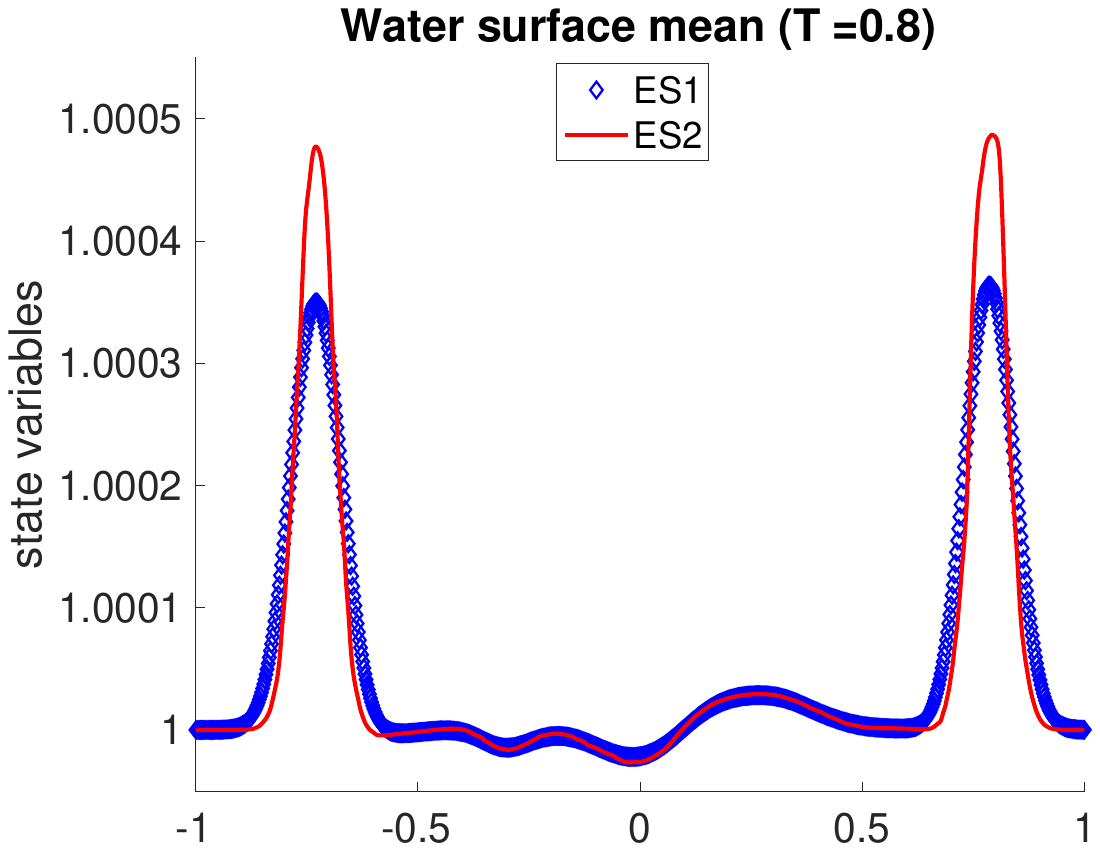}    
    \includegraphics[width = .32\textwidth, trim={0 6cm 0 6cm}, clip]{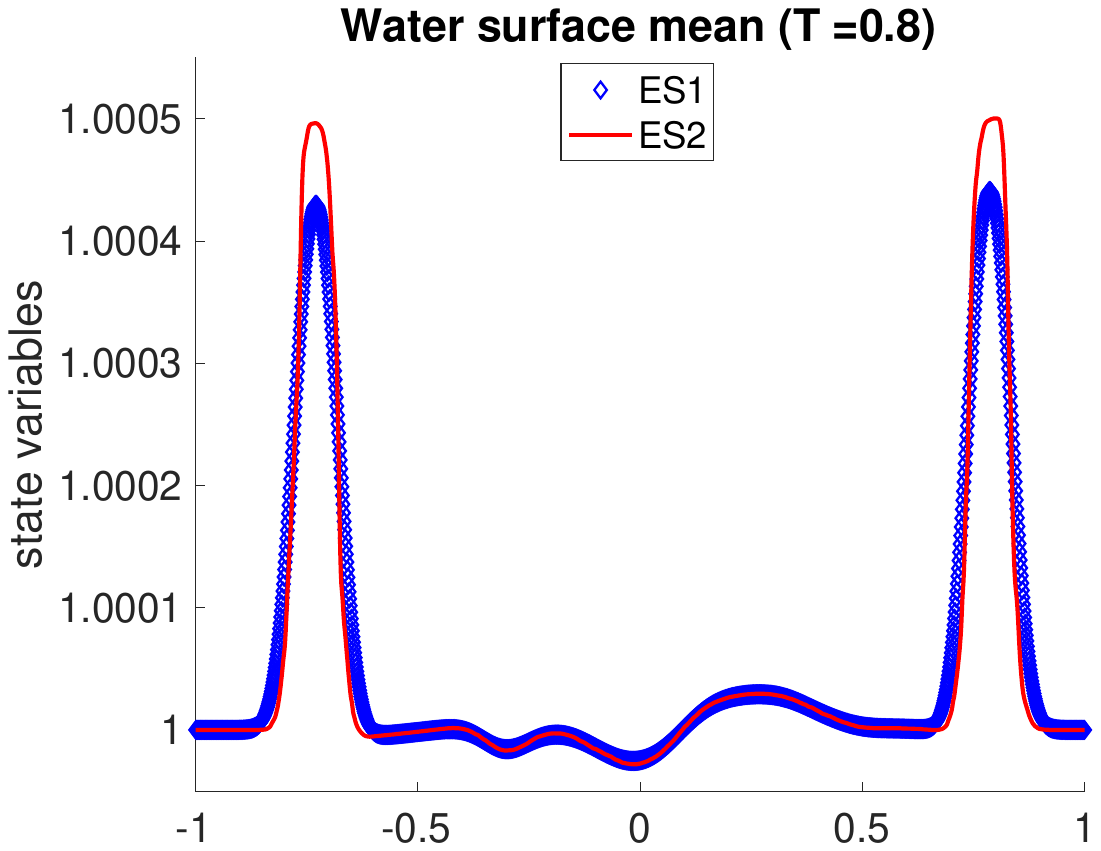}  
    \\
   \includegraphics[width = .32\textwidth, trim={0 6cm 0 6cm}, clip]{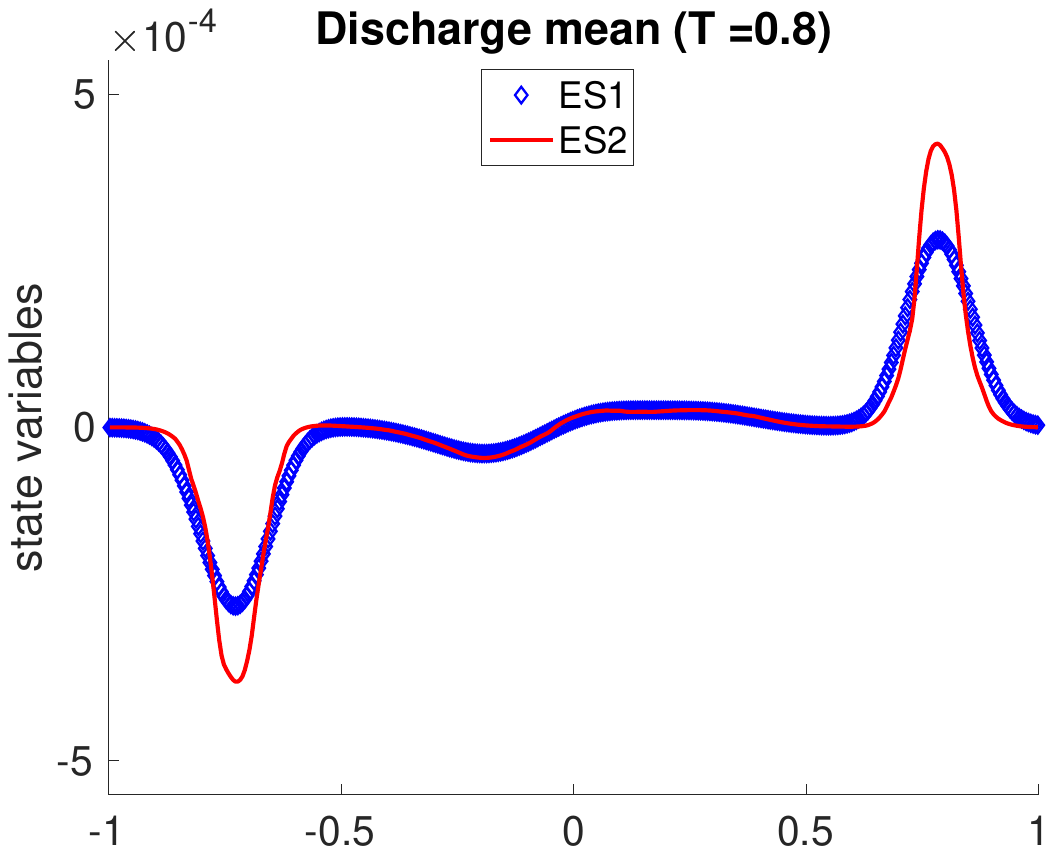}
    \includegraphics[width = .32\textwidth, trim={0 6cm 0 6cm}, clip]{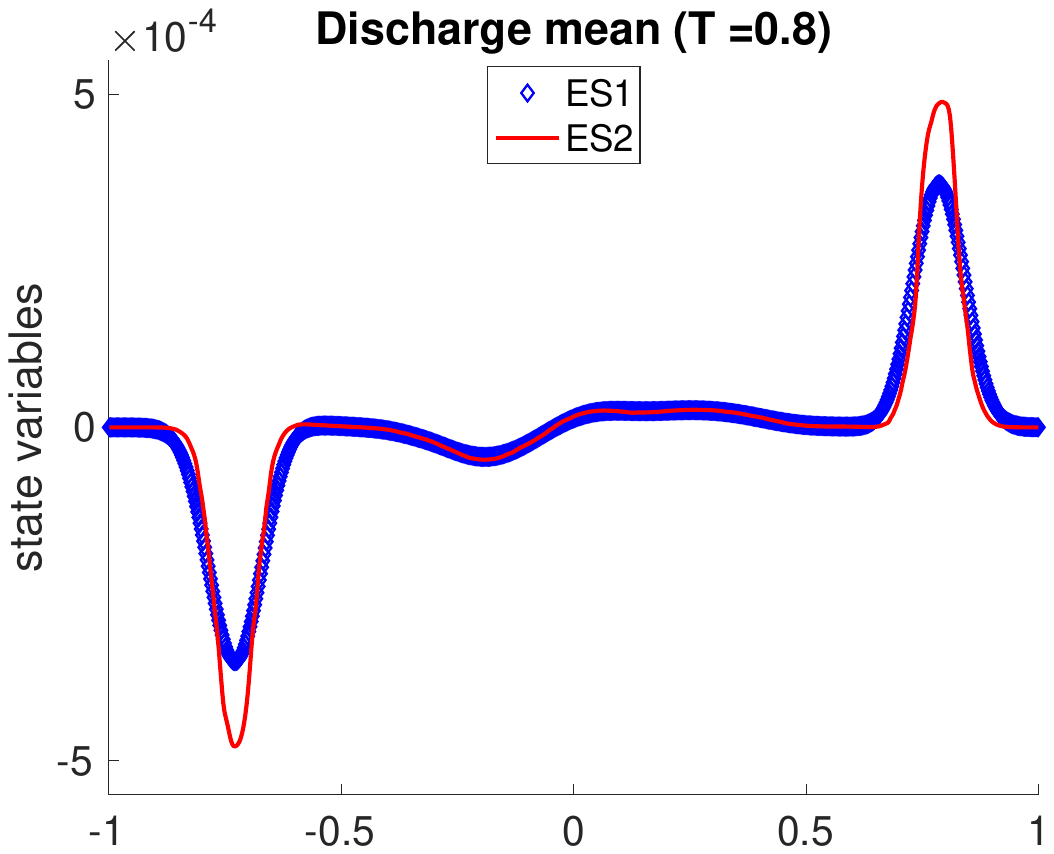}    
    \includegraphics[width = .32\textwidth, trim={0 6cm 0 6cm}, clip]{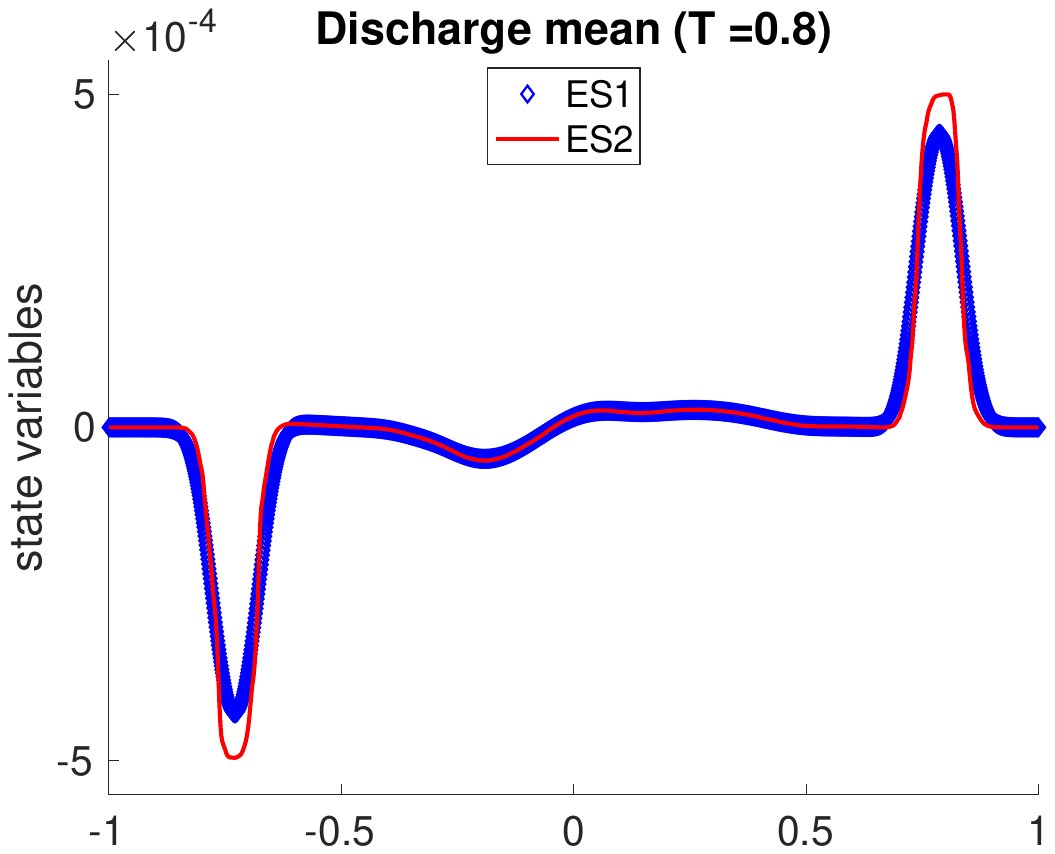}  
   % \includegraphics[width =
  %  .32\textwidth]{figure/1d-ex2-es1-N9-nx400-e.png}
   % \includegraphics[width = .32\textwidth]{figure/1d-ex2-N9-nx400-e-cu1.png}    
  %  \includegraphics[width = .32\textwidth]{figure/1d-ex2-N9-nx400-e-cu2.png}    
    \caption{Results for \cref{ssec:wbperturb2}. Top: water
      surface mean, Bottom: discharge mean: Left: ES1 vs. ES2 on mesh
      $nx=400$. Middle: ES1 vs. ES2 on mesh $nx=800$. Right: ES1 vs. ES2 on mesh $nx=1600$. PC basis functions K=9.}
    \label{fig:ex2-mean}
  \end{figure}

  \begin{figure}[htbp]
    \centering
    \includegraphics[width = .32\textwidth, trim={0 6cm 0 6cm}, clip]{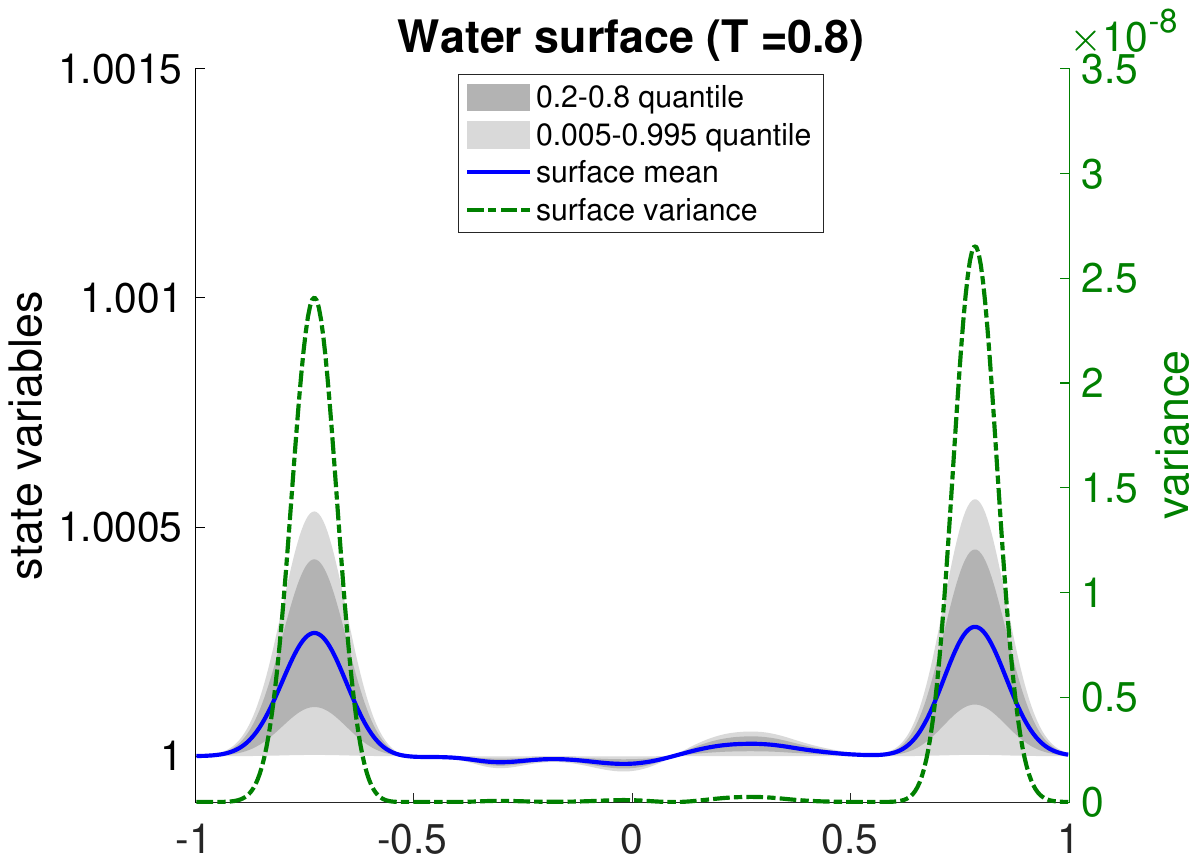}
    \includegraphics[width = .32\textwidth, trim={0 6cm 0 6cm}, clip]{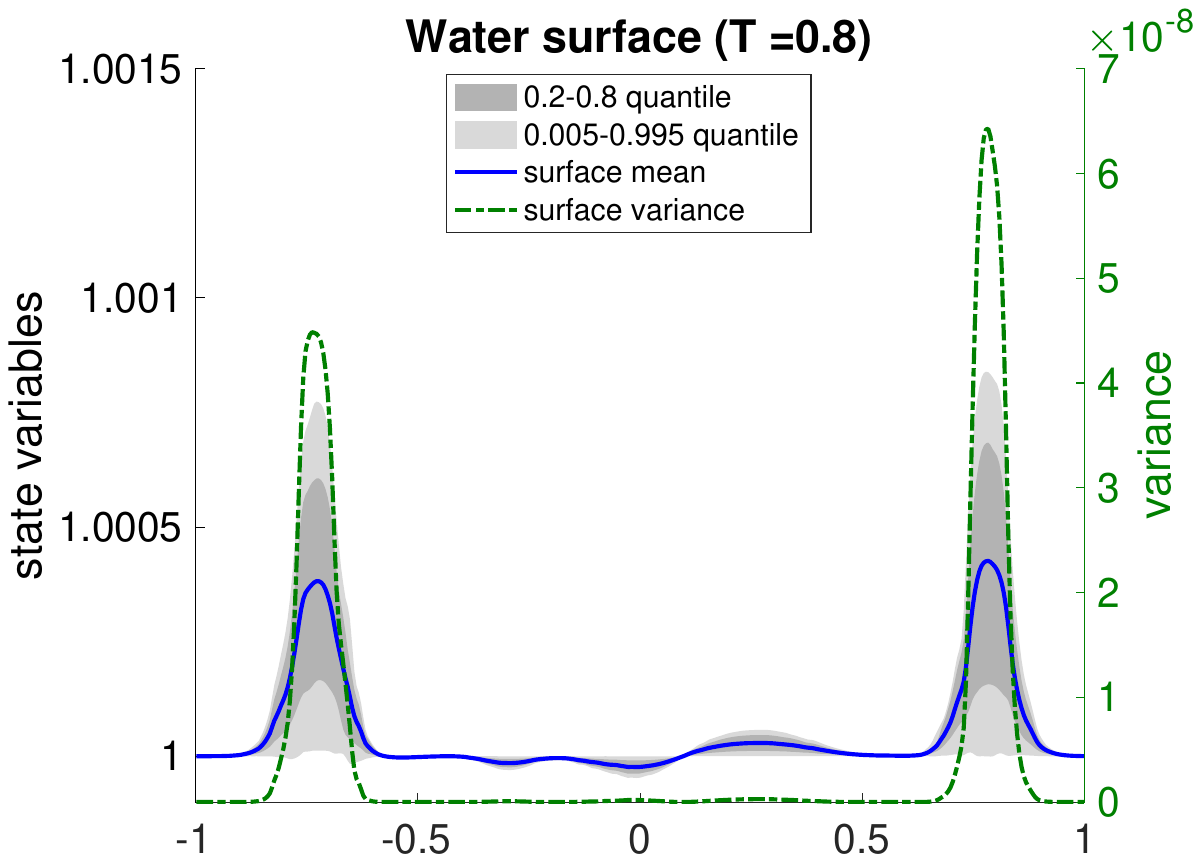}    
    \includegraphics[width = .32\textwidth, trim={0 6cm 0 6cm}, clip]{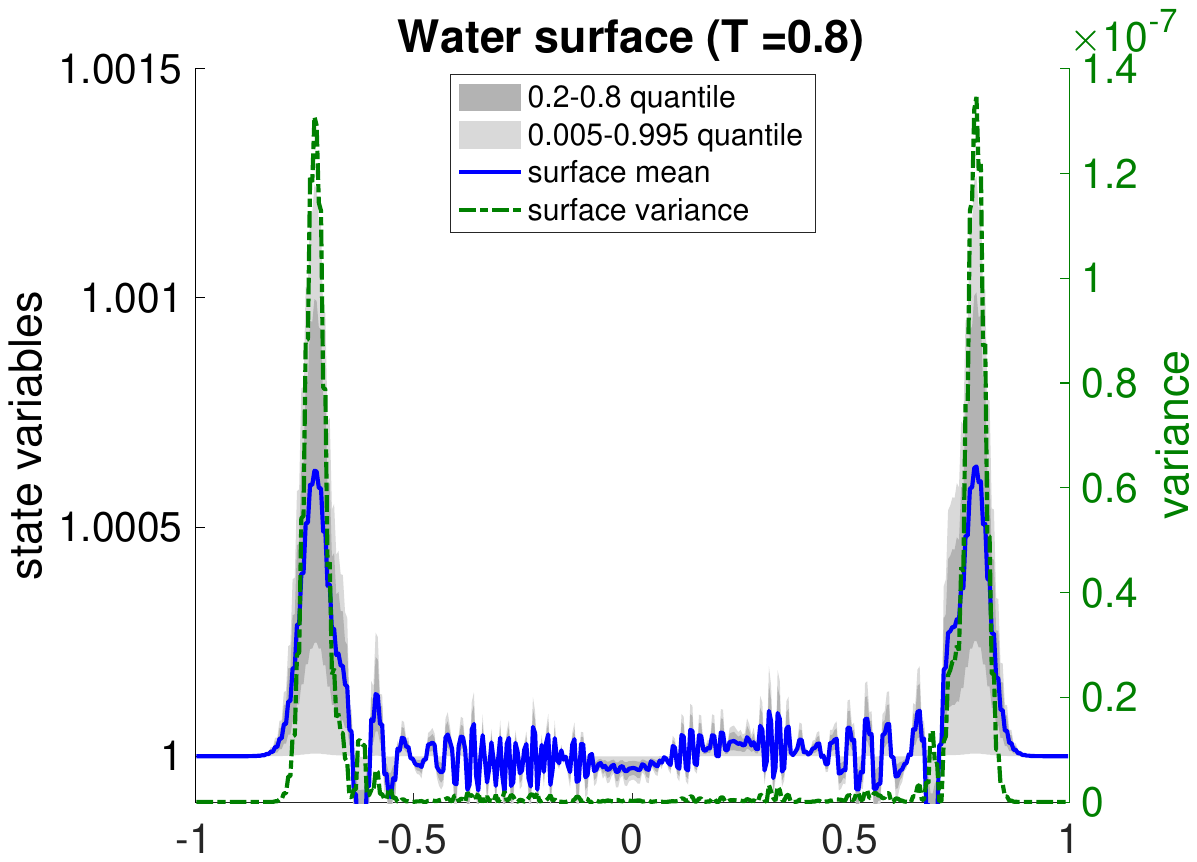}    
    \\
  \includegraphics[width = .32\textwidth, trim={0 6cm 0 6cm}, clip]{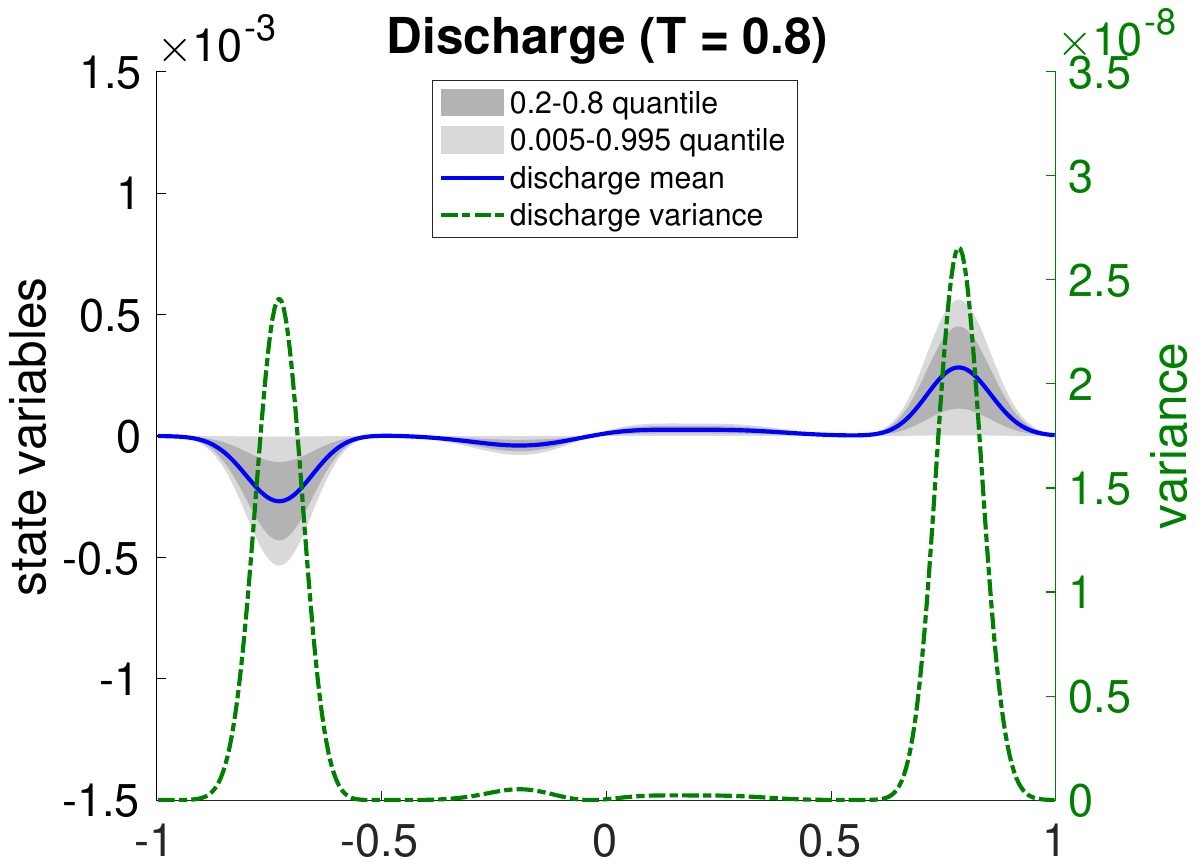}
    \includegraphics[width = .32\textwidth, trim={0 6cm 0 6cm}, clip]{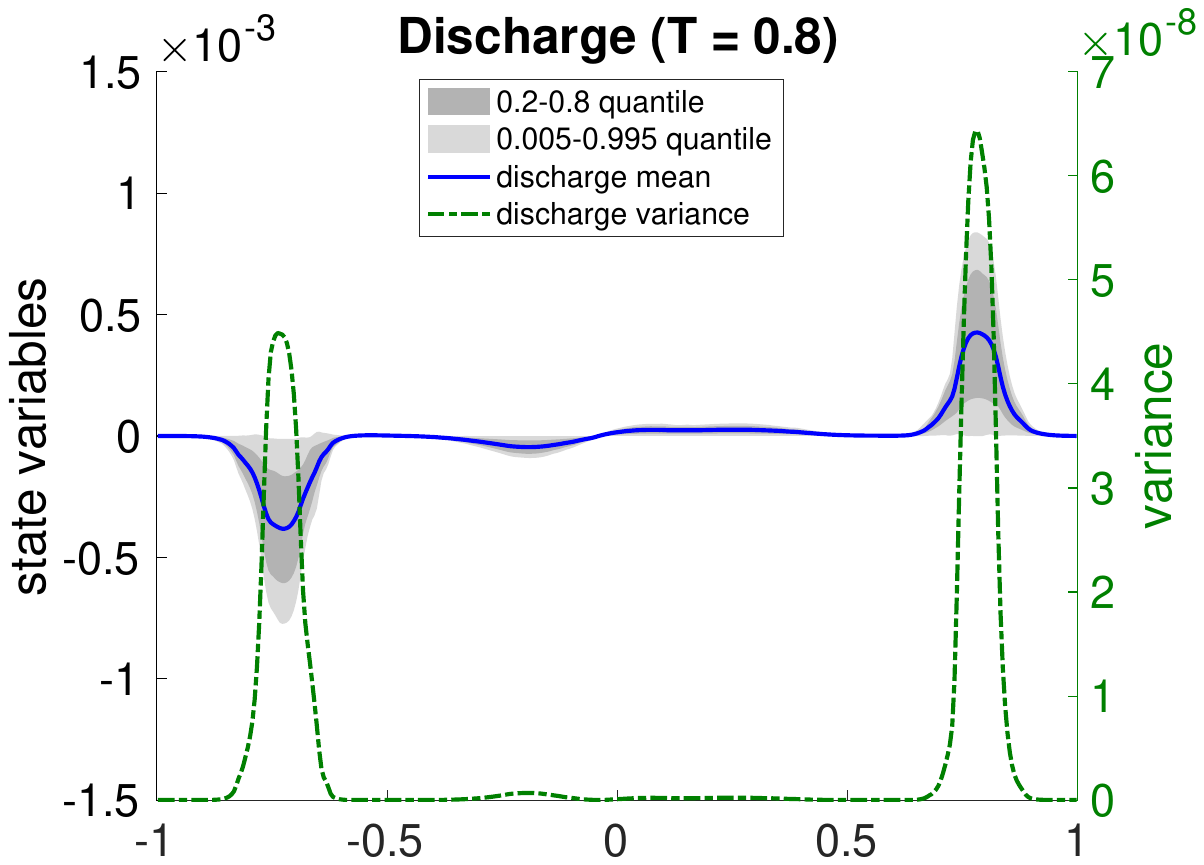}    
    \includegraphics[width = .32\textwidth, trim={0 6cm 0 6cm}, clip]{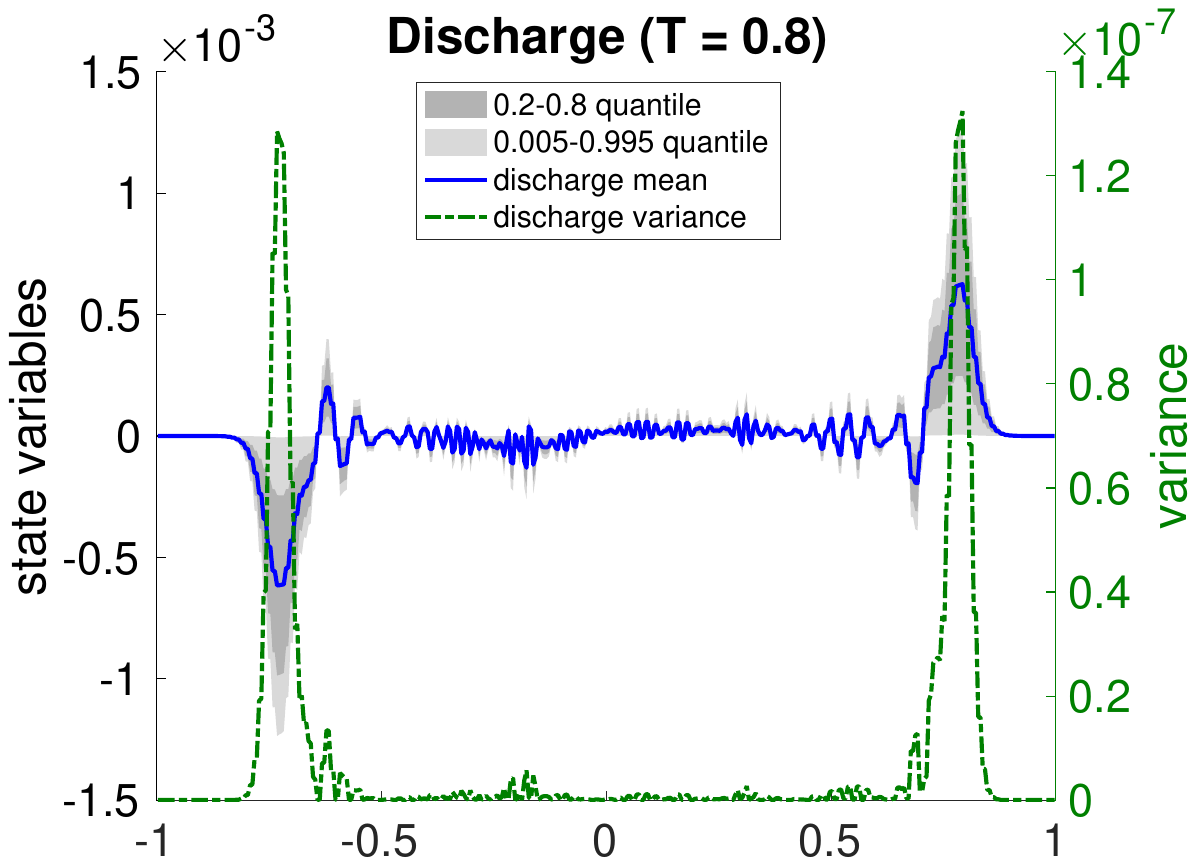}   
    \caption{Comparison of the results for \cref{ssec:wbperturb2} using
      different schemes. Top: water surface. Bottom: discharge. Left:
      ES1, Middle: ES2, Right: EC. Mesh $nx=400$ with $K=9$ at 
    time $T=0.8$.}
    \label{fig:ex2-m400}
  \end{figure}
\begin{figure}[htbp]
    \centering
    \includegraphics[width = .49\textwidth, trim={0 6cm 0 6cm}, clip]{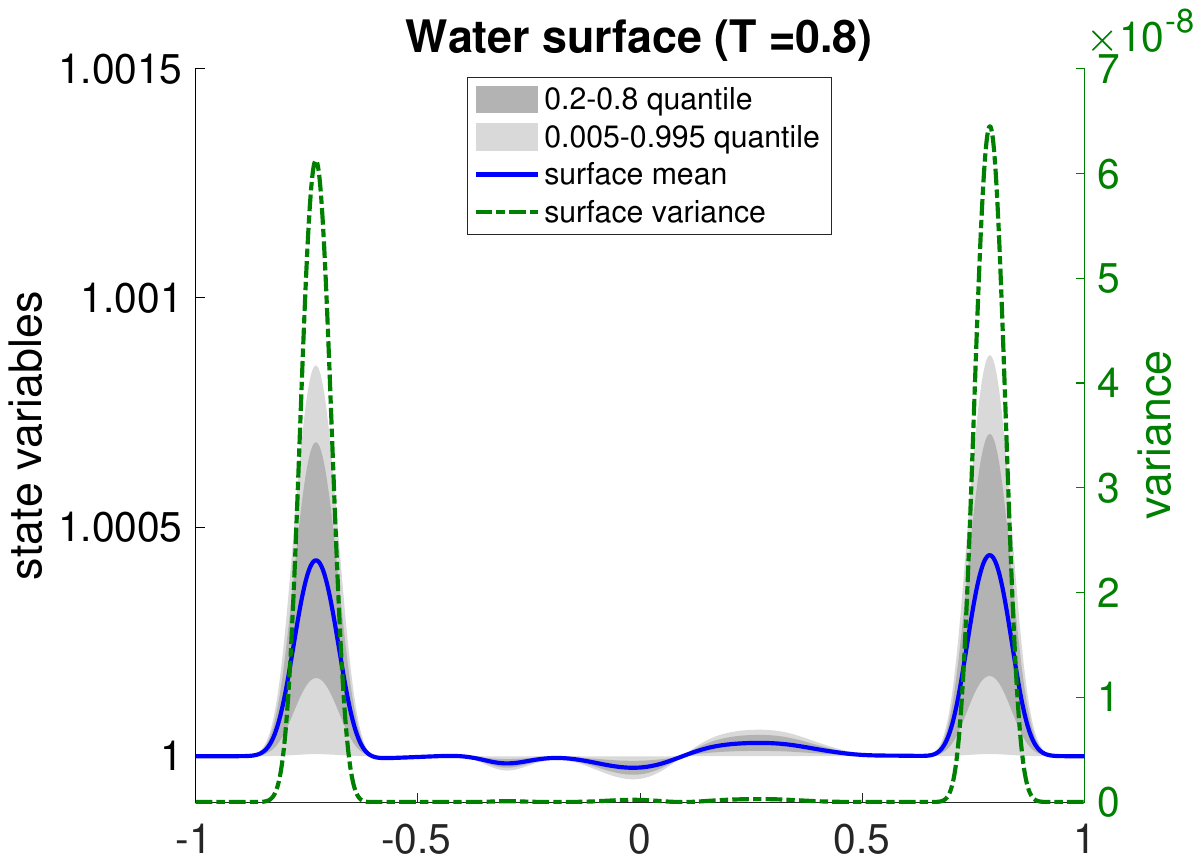}
    \includegraphics[width = .49\textwidth, trim={0 6cm 0 6cm}, clip]{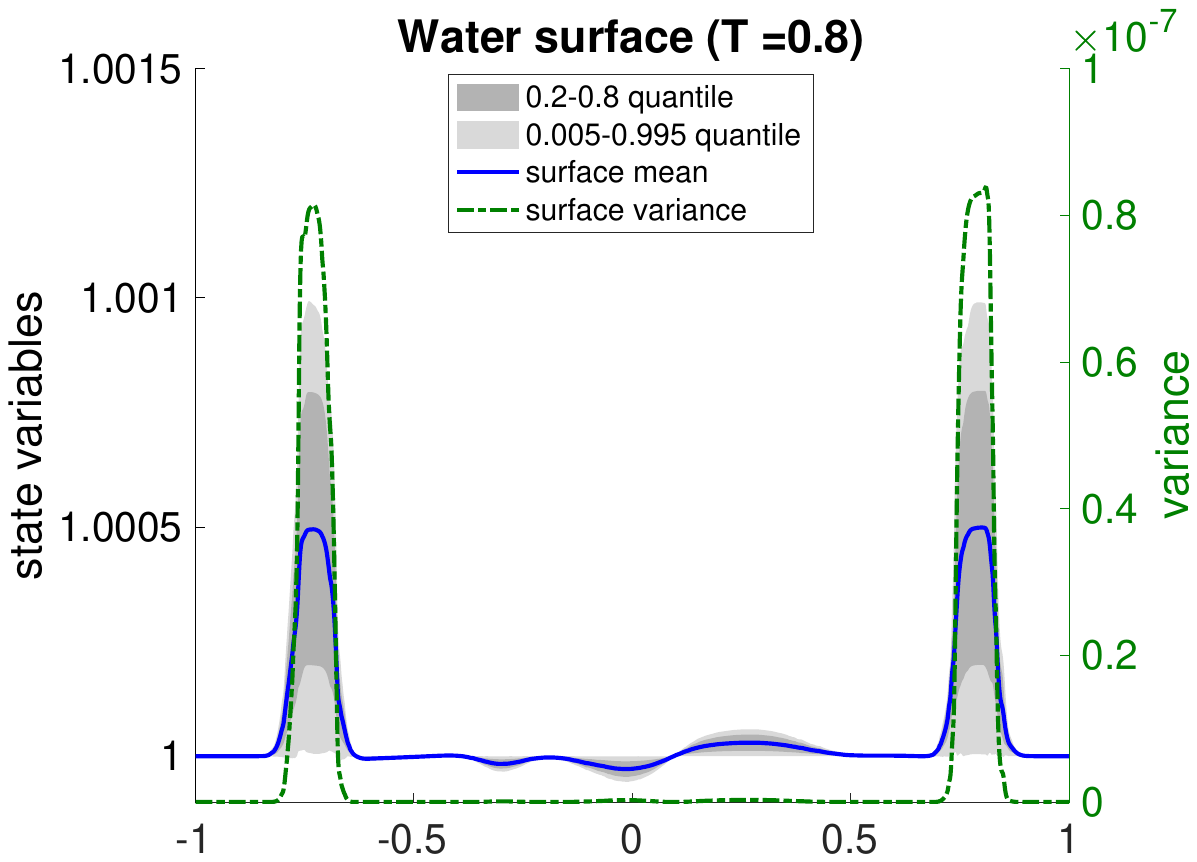}       
    \\
  \includegraphics[width = .49\textwidth, trim={0 6cm 0 6cm}, clip]{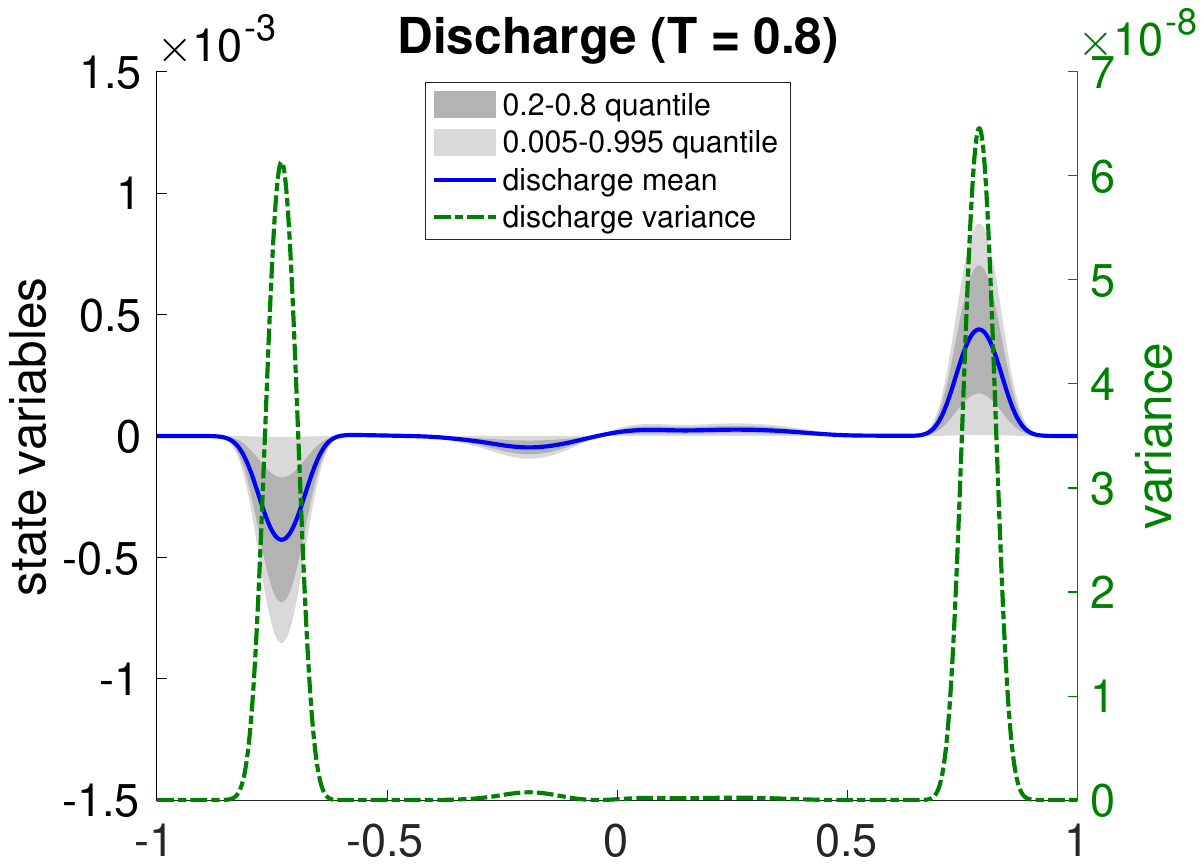}
    \includegraphics[width = .49\textwidth, trim={0 6cm 0 6cm}, clip]{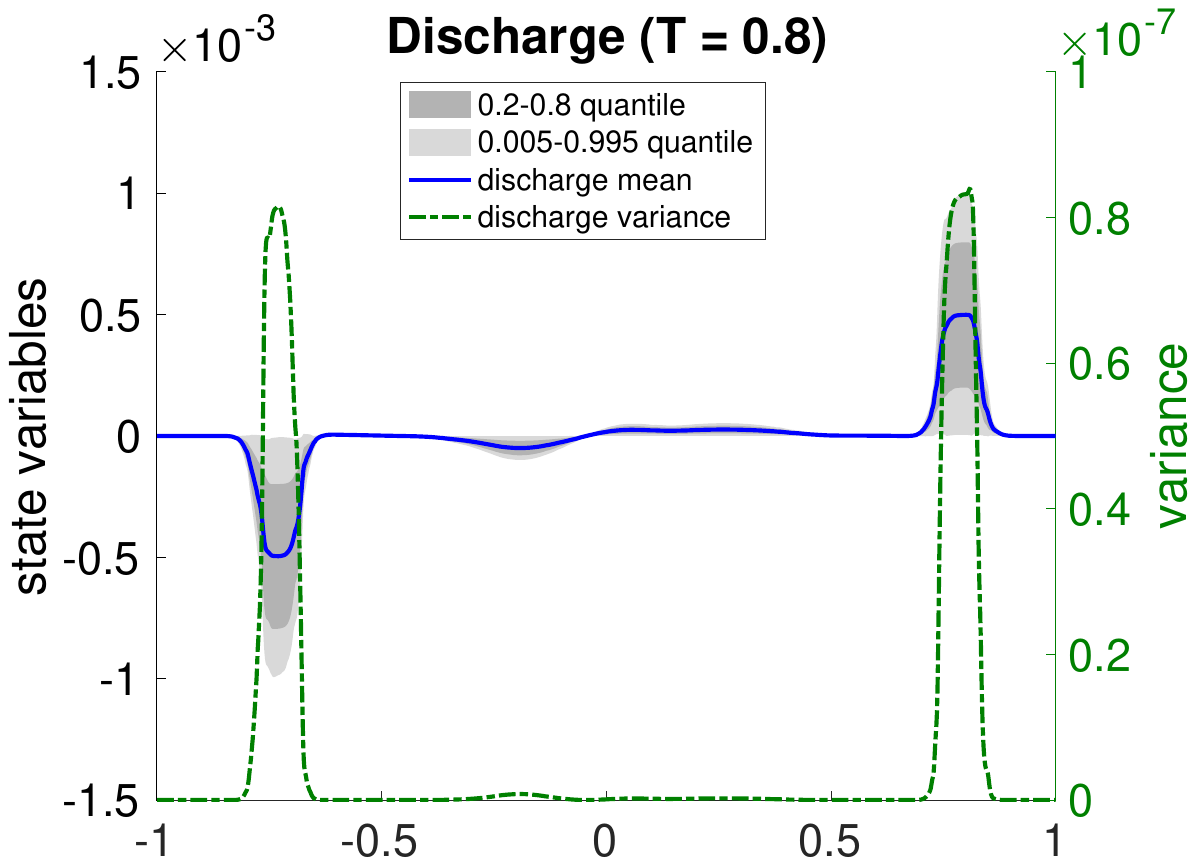}    
    \caption{Comparison of the results for \cref{ssec:wbperturb2} using
      different schemes. Top: water surface. Bottom: discharge. Left:
      ES1, Right: ES2. Mesh $nx=1600$ with $K=9$ at 
    time $T=0.8$.}
    \label{fig:ex2-m1600}
  \end{figure}
  
\begin{figure}[htbp]
    \centering
     \includegraphics[width = .49\textwidth, trim={0 6cm 0 6cm}, clip]{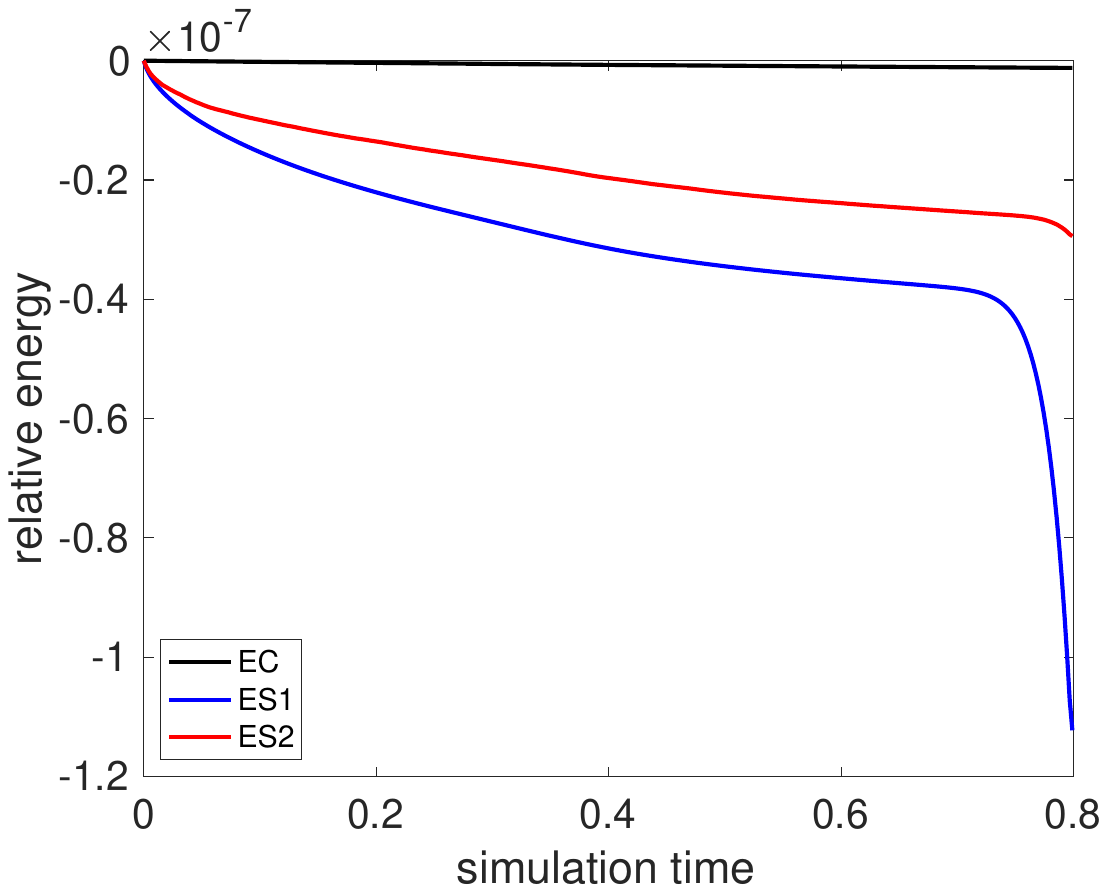}
     \includegraphics[width = .49\textwidth, trim={0 6cm 0 6cm}, clip]{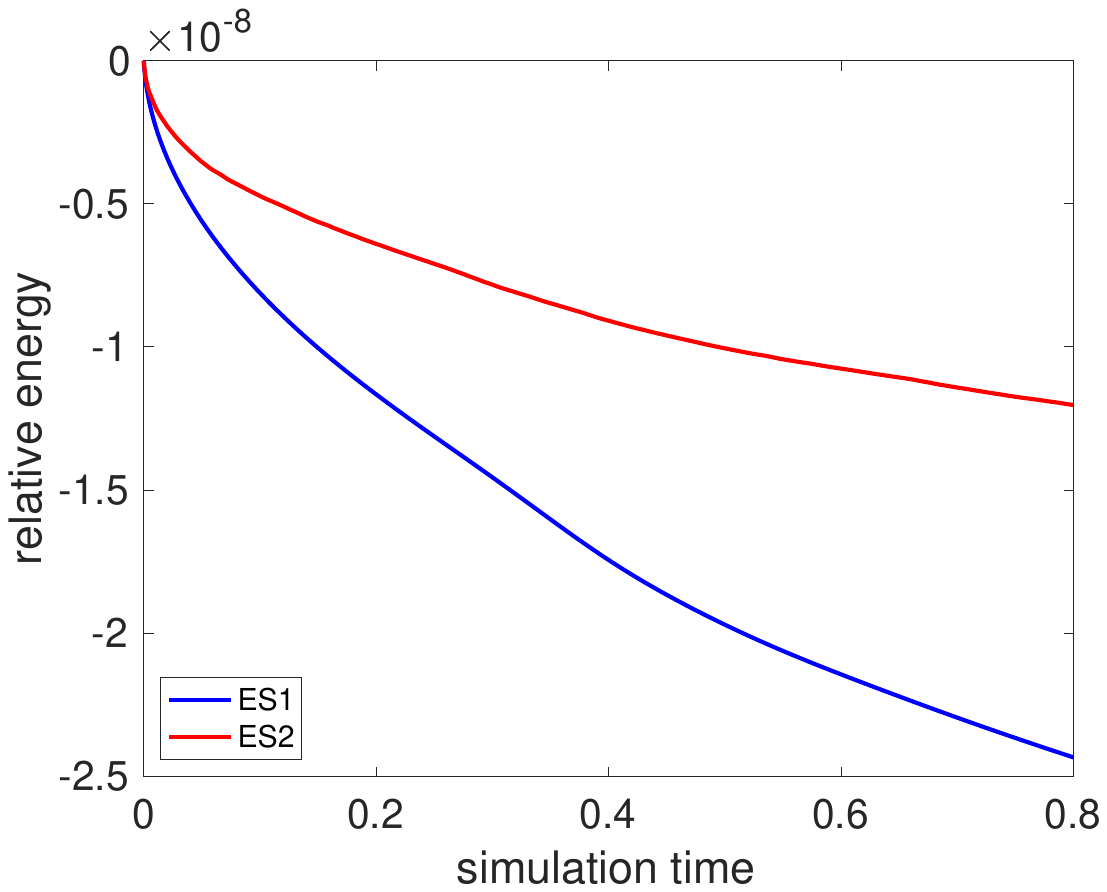}
    \caption{Comparison of the results for \cref{ssec:wbperturb2} using
      different schemes. Relative energy change: Left: EC vs. ES1
    vs. ES2 on mesh $nx=400$. Right: ES1 vs. ES2 on mesh $nx=1600$.}
    \label{fig:ex2-r_E}
  \end{figure}
  \section{Conclusion}\label{sec:conclusion}
 In this work we derived an entropy-entropy flux pair for the
 spatially one-dimensional hyperbolicity-preserving,
 positivity-preserving SG SWE system
 developed in \cite{doi:10.1137/20M1360736}. Such entropy-entropy flux
 pairs are the theoretical starting point for proposing entropy
 admissibility criteria to resolve non-uniqueness of weak solutions.
 Next, using the proposed entropy-entropy flux pair, we designed 
 second-order energy conservative, and first- and second-order energy
 stable finite volume schemes for the SG SWE. The proposed schemes are
 also well-balanced. We provided several numerical experiments to
 illustrate performance of the methods. As a part of future research,
 we plan to extend such methods to models in two spatial dimensions, to
 explore alternative constructions of the diffusion operators, and to investigate 
 other reconstruction approaches for the entropy variables.
 \section{Acknowledgement}
The work of Yekaterina Epshteyn and Akil Narayan was partially supported by NSF DMS-2207207. AN was partially supported by NSF DMS-1848508.
\bibliographystyle{amsplain} 
\bibliography{bibfile}

\end{document}